\documentclass[a4paper,12pt,leqno]{article}
\usepackage{amsmath}
\usepackage{amsthm}
\usepackage{amssymb}
\usepackage{amscd}
\usepackage{graphicx, color}
\usepackage[all]{xy}

\newtheorem{thm}{Theorem}[section]
\newtheorem{definition}[thm]{Definition}
\newtheorem{prp}[thm]{Proposition}%[section]
\newtheorem{lemma}[thm]{Lemma}%[section]
\newtheorem{crl}[thm]{Corollary}%[section]
\newtheorem{remark}[thm]{Remark}%[section]
\numberwithin{equation}{section}

\usepackage[all]{xy}

\newcommand{\Hol}{\mbox{{\rm Hol}}}

\newcommand{\N}{\mathbb{N}}
\newcommand{\Z}{\mathbb{Z}}
\newcommand{\C}{\mathbb{C}}
\newcommand{\R}{\mathbb{R}}
\newcommand{\K}{\mathbb{K}}
\newcommand{\F}{\mathbb{F}}

\renewcommand{\P}{{\rm P}}
\newcommand{\SP}{\mbox{{\rm SP}}}

\newcommand{\RP}{\mathbb{R}{\rm P}}

\newcommand{\Q}{\mbox{{\rm Q}}}

\newcommand{\Map}{\mbox{{\rm Map}}}

\newcommand{\CP}{\mathbb{C}{\rm P}}

\newcommand{\dis}{\displaystyle}
\newcommand{\p}{\prime}

\newcommand{\Ha}{\mathbf{\rm H}}
\newcommand{\I}{\mbox{{\rm (i)}}}
\newcommand{\II}{\mbox{{\rm (ii)}}}
\newcommand{\III}{\mbox{{\rm (iii)}}}
\newcommand{\IV}{\mbox{{\rm (iv)}}}

\newcommand{\SZ}{{\mathcal{X}}^{d}}

\newcommand{\Po}{\mbox{{\rm Poly}}}
\newcommand{\po}{\mbox{{\rm Poly}}}
\newcommand{\pol}{\mbox{{\rm Pol}}}

\title{\bf
Spaces of non-resultant systems of bounded multiplicity with real coefficients
}
\author{\bf Andrzej Kozlowski\footnote{%
Institute of Applied Mathematics and Mechanics,
University of Warsaw, Banacha 2, 02-097 Warsaw, Poland
(E-mail: akoz@mimuw.edu.pl)
}
\  and \ 
Kohhei Yamaguchi\footnote{%
Department of Mathematics,
University of Electro-Communications,  Chofu, Tokyo 182-8585, Japan
(E-mail: kohhei@im.uec.ac.jp)
}
}

\date{}
\begin{document}
\maketitle

%%%(Abstract)%%%%%%%%
\begin{abstract}
%%%%%%%%%%%%%%%%%%%%
For each pair $(m,n)$  of positive integers with $(m,n)\not= (1,1)$
and an arbitrary field $\F$ with  algebraic closure $\overline{\F}$,
let
$\Po^{d,m}_n(\F)$ denote the space of
$m$-tuples
$(f_1(z),\cdots ,f_m(z))\in \F [z]^m$ of $\F$-coefficients monic
polynomials of the same degree $d$ such that 
the polynomials
$\{f_k(z)\}_{k=1}^m$ have no common root in $\overline{\F}$ of multiplicity
$\geq n$.
These spaces $\Po^{d,m}_n(\F)$ were first defined and studied by
B. Farb and J. Wolfson as generalizations of spaces first studied by Arnold, Vassiliev and Segal and others in several different contexts.
In previous paper we determined explicitly the homotopy type  of this space
in the case $\F =\C$.
In this paper, we investigate this  space in the case $\F =\R$.
%We also consider several mistakes given in the main results in \cite{KM2}, \cite{KM3}
%and \cite{KM4}.
%%%%%%%%%%%%%%%%%%%%
\end{abstract}
%%%%%

%%%%(SECTION 1: Introduction)%%%%
%%%%%
\section{Introduction}\label{section: introduction}
%%%%%

%%%(1.1: The motivation of paper etc)
\paragraph{The motivation.}
%%%%%%%%%%%%%%%%%%%%
The motivation of this paper comes from the works of V. I. Arnold \cite{Ar},
V. Vassiliev \cite{Va},
G. Segal \cite{Se}, and B. Farb and J. Wolfson \cite{FW}. 
\par
Arnold considered the space $\SP^d_n(\C)$ of complex coefficients monic polynomials
 of degree $d$
without roots of multiplicity $\geq n$, which plays a role in the theory of singularities. 
For example, if $n=2$, this space is the same as the space of complex monic polynomials of degree $d$ without repeated roots, 
and this is homotopy equivalent to the Eilenberg-McLane space $K({\rm Br}(d),1)$, where
${\rm Br}(d)$ denotes the Artin braid group of $d$ strings.
%\par
%whose fundamental group is the Artin braid group $\mbox{Br}(d)$ of $d$ strings and whose homology is that of the braid group $\mbox{Br}(d)$.
Arnold \cite{Ar} computed the homology of these spaces and established their homology stability.
His results were much extended and generalized by V. Vassiliev \cite{Va}.
\par\vspace{0.5mm}\par
Analogous results were discovered by G. Segal \cite{Se} 
in a different context inspired by control theory.
Segal considered the space $\Hol^*_d(S^2,\CP^{m-1})$ of base-point preserving holomorphic maps of degree $d$ from the Riemann sphere $S^2$ to the $(m-1)$-dimensional complex projective space $\CP^{m-1}$, and   its inclusion
into the space $\Map^*_d(S^2,\CP^{m-1})=\Omega^2_d\CP^{m-1}$ 
of corresponding space of base-point preserving continuous maps. 
Intuitive considerations based on Morse theory suggest that the homotopy type of the first space should approximate  that of the second space more and more
closely as the degree $d$ increases. 
Segal proved that this is true by observing that  the space 
$\Hol^*_d(S^2,\CP^{m-1})$ can be identified with the space of $m$-tuples
$(f_1(z),\cdots ,f_m(z))\in \C [z]^m$
of monic polynomials of the same degree $d$ without common roots.
%%%%%%
He defined a stabilization map 
$\Hol_d^*(S^2,\CP^{m-1})\to \Hol_{d+1}^*(S^2,\CP^{m-1})$ 
and  proved that the induced maps on homotopy groups are isomorphisms up to some dimension increasing with $d$.
%by a modification of Arnold's method. 
Using a different technique  he also proved that 
there is a homotopy equivalence 
$S: \varinjlim \Hol_d^*(S^2,\CP^{m-1}) 
\stackrel{\simeq}{\longrightarrow} \Omega^2_0 \CP^{m-1}$ 
defined by a \lq\lq scanning of particles\rq\rq, 
 and that this equivalence is homotopic to the inclusion of the space of all holomorphic maps into the space of all continuous maps.
%%%% 
\par\vspace{1mm}\par
Inspired by the classical theory of resultants and the algebraic nature of Arnold's and Segal's arguments,  B. Farb and J. Wolfson (\cite{FW},
\cite{FWW}) defined algebraic varieties $\po^{d,m}_n(\F)$, which generalize
the spaces considered by Arnold and Segal.
These varieties  $\po^{d,m}_n(\F)$ are defined as follows. 
\par
For a field $\F$ with its algebraic closure $\bar{\F}$, 
let
 $ \Po^{d,m}_n(\F)$ denote the space of  $m$-tuples
$(f_1(z),\cdots ,f_m(z))\in \F [z]^m$ 
of monic $\F$-coefficients polynomials of the same degree $d$ with no common root 
in $\overline{\F}$ of multiplicity $\geq n$. 
For example, if $\F =\C$,  $\Po^{d,1}_n (\C)=\SP^d_n(\C)$
and 
$\Po^{d,m}_1 (\C)$ can be identified with the space
$\Hol_d^*(S^2,\CP^{m-1})$. 
Note that
the space 
$\text{Poly}_n^{d,m}(\C)$ is an affine variety defined by systems of polynomial equations with integer coefficients
by the classical theory of %discriminants and 
resultants.
Thus this variety can be defined over $\Z$ and over any filed $\F$. 
%Farb and Wolfson proved an intriguing result about the \'etale cohomology of the spaces they defined, which was an analogue of an already known homotopy theoretic result for the case $\F = \C$. Although the conjecture they made was disproved in \cite{ST}, the analogy between the algebraic and the topological situation remains intriguing. 
%%
\par
Farb and Wolfson computed various algebraic and geometric invariants of these varieties (such as the number of points for a finite field $\F_q$, \'etale cohomology etc) and found for the varieties 
$\po^{dn,m}_n(\F)$ and $\po^{d,mn}_1(\F)$  that they were equal. 
They conjectured that these varieties were algebraically isomorphic.
Although this conjecture was disproved in \cite{ST},
analogous results  
(e.g. (\ref{eq: poly stability for C}), (\ref{eq: KY13 stable equiv}))
hold in the homotopy category.
%and
So this analogy between the algebraic and the topological situation remains intriguing.
\par
%%%%
From the topological point of view, there has been a lot of work on the homotopy type of $\po^{d,m}_n(\F)$ for $\F=\C$.
 The space $\po^{d,m}_n(\C)$ has been studied  by V. Vassiliev  \cite{Va}, G. Segal \cite{Se},  M. Guest and the present authors \cite{GKY2}, and the present authors \cite{KY8}.
\par\vspace{2mm}\par
In this article we  study  the homotopy type of the  space 
$ \Po^{d,m}_n(\F)$ for $\F=\R$.
%%%
 Note that
 the homotopy type of the space $\po^{d,m}_n(\R)$
 is already  known for $mn=2$.
 Indeed, if $(m,n)=(2,1)$, the the homotopy type of the space $\po^{d,m}_n(\R)$ was  determined  in 
 \cite[Propositions 7.1 and 7.2]{Se}.
 The case $(m,n)=(1,2)$, can be easily determined by using \cite{Se0}.\footnote{%
 %%%(FootNote 1)%%%%
 Since this does not appear to be stated anywhere,
  we consider this case  in \S \ref{section: appendix}.
% See Theorem \ref{thm: the case (m,n)=(1,2)}
% for further details.
}
%%(End of Footnote 1)%%
%%
Note  that case
%the space $\po^{d,m}_n(\R)$ has already been considered in the case 
 $n=1$ (with $m\geq 3$)
 was also studied in
\cite{KM2}.
\par\vspace{1mm}\par
%%%
The main purpose of this paper is to investigate 
the homotopy type of the space
$\po^{d,m}_n(\R)$ 
for the case $mn\geq 3$.
%%%
In particular, 
we prove that an Atiyah-Jones-Segal type result holds for this space
%by using the Vassilev spectral sequence and several scanning maps, 
and determine its stable homotopy type explicitly.
%for the case $mn\geq 4$.
%%
More precisely, our  main results can be summarized
 as follows
(see Theorems \ref{thm: KY13},
\ref{thm: KY13; stable homotopy type}
and Corollary 
%\ref{crl: KY14}, 
\ref{crl: KY13; stable homotopy type} for further details).
%\par\vspace{3mm}\par
%%%(%%%%%%%
%%%(Theorem 1.1: Main Theorem)%%%%
%%%(Theorem 1.1:Main Theorem)%%%%  
%%%(Theorem 1.1:Main Theorem)%%%%    
%%%%%%%%%%%%%
\begin{thm}
[Theorems \ref{thm: KY13},  \ref{thm: KY13; stable homotopy type} and
Corollary 
%\ref{crl: KY14}, 
\ref{crl: KY13; stable homotopy type}]
\label{thm: Main theorem}
%%%%%%%%%%%%%
Let $m,n,d\geq 1$ be positive integers such that $(m,n)\not= (1,1)$ with $d\geq n$,\footnote{%
%%%%(FootNoe 2)%%%%%
If $d<n$,  the space $\po^{d,m}_n(\R)$ is contractible by (\ref{eq: contractible}), so
assume that $d\geq n$.
}
%%(End of Footnote 2)%%%%
and
let $D(d;m,n)$ denote the positive integer defined by
%%%(1.1)%%
\begin{equation}\label{eq: 1.1}
%%%%%
D(d;m,n)=(mn-2)(\lfloor d/n\rfloor+1)-1,
\end{equation} 
where $\lfloor x\rfloor$ denotes the integer part of a real number $x$.
%\begin{enumerate}
%\item[$\I$]%(Theorems \ref{thm: KY13})
\par\vspace{1mm}\par
$\I$
The natural map $($defined by $($\ref{eq: mapjR}$))$
$$
i^{d,m}_{n,\R}:\Po^{d,m}_n(\R)\to (\Omega^2_d\CP^{mn-1})^{\Z_2}
\simeq
\Omega^2S^{2mn-1}\times \Omega S^{mn-1}
$$
is a homotopy equivalence through dimension
$D(d;m,n)$ if $mn\geq 4$, and
a homology equivalence through dimension $D(d;m,n)$
if $mn=3$.
%%%%%%%%
%$\I$
%If $mn\geq 4$, the natural map $($defined by $($\ref{eq: mapjR}$))$
%$$
%i^{d,m}_{n,\R}:\Po^{d,m}_n(\R)\to (\Omega^2_d\CP^{mn-1})^{\Z_2}
%\simeq
%\Omega^2S^{2mn-1}\times \Omega S^{mn-1}
%$$
%is a homotopy equivalence through dimension
%$D(d;m,n)$.
%%%%%
%%\item[$\II$]
%\par
%$\II$
%Let $(m,n)=(3,1)$.
%Then  $D(d;3,1)=d$ and
%the natural map
%$$
%i^{d,3}_{1,\R}:\Po^{d,3}_1\to (\Omega^2_d\CP^{5})^{\Z_2}\simeq\Omega^2S^{5}\times \Omega S^{2}
%\simeq \Omega^2S^{5}\times \Omega S^{3}\times S^1
%$$
%is a homotopy equivalence through dimension
%$d$ if $d\equiv 1$ $\mbox{$($\rm mod $2)$}$, and it
%is a homotopy equivalence up to dimension
%$d$ if $d\equiv 0$ $\mbox{$($\rm mod $2)$}.$
%%%%%
%\par
%$\II$
%%%%%%%%%%%%%%
%If $(m,n)=(1,3)$,
%$D(d;1,3)=\lfloor d/3\rfloor$ and
%the natural map
%$$
%i^{d,1}_{3,\R}:\Po^{d,1}_3(\R)\to (\Omega^2_d\CP^{2})^{\Z_2}
%\simeq
%\Omega^2S^{5}\times \Omega S^{2}
%\simeq
%\Omega^2S^5\times \Omega S^3\times S^1
%$$
%is a homology equivalence through dimension
%$\lfloor d/3\rfloor$.
%%%%
\par
$\II$
If $mn\geq 3$,
there is a stable homotopy equivalence
\begin{align*}
\po^{d,m}_n(\R)
&\simeq_s
\po^{\lfloor d/2\rfloor, m}_n(\C) \vee B^{d,m}_n
\vee \Q^{d,m}_n(\R)
\\
&\simeq_s
\Big(\bigvee_{i=1}^{\lfloor d/n\rfloor}S^{(mn-2)i}\Big)
\vee
\Big(
\bigvee_{i\geq 0,j\geq 1, i+2j\leq \lfloor d/n \rfloor}
\Sigma^{(mn-2)(i+2j)}D_j\Big),
\end{align*}
where
$B^{d,m}_n$ and 
$D_j$
%=F(\C,k)_+\wedge_{S_k}S^k$ 
denote the spaces defined by  
(\ref{eq: the space B})
and $($\ref{eq: Dd}$)$, respectively.
%%(iii)%%
%\end{enumerate}
%}
\end{thm}
%%%%
%%(End of Theorem 1.1)%%%%%%
%%%%
%%%%
%%%%
%%%%(Remark 1.2)%%%%%%
%%%%
\begin{remark}
%%%
{\rm
Note 
that  homotopy stability holds for the space $\po^{d,m}_n(\R)$ 
when
$mn\geq 4$ as stated in Theorem \ref{thm: Main theorem}.
Since  homology stability holds  and the map $i^{d,m}_{n,\R}$ induces an isomorphism on the fundamental group
$\pi_1(\ )$ for  $mn=3$
(Theorem \ref{thm: Main theorem} and Corollary \ref{lmm: pi1}),
we expect that the corresponding homotopy stability also holds in this case.
We leave this problem to another paper \cite{KY14}.
\qed
%%%%%%%%%%%%
}
%%%
\end{remark}
%%(End of Remark 1.2)%%%
\par\vspace{2mm}\par
%%%%%
%%%%%
%%%%%
%%%%%(The organization of this paper)%%%
%%%%%
%%%%%
The organization of this paper is as follows.
%\par
%%(section 2)%%
In \S \ref{section Definitions} we recall several definitions,
notations, and  known results.
After then we give the main results of this paper
(Theorems \ref{thm: KY13}  and 
\ref{thm: KY13; stable homotopy type}).
%%(SECTION 3)%%
In \S \ref{section 3} we investigate the homotopy type of the spaces
$(\Omega^2_d\CP^{N})^{\Z_2}$ and 
$\Po^{d,m}_n(\R;\Ha_+)$.
%%(section4)%%
In \S \ref{section: spectral sequence} we  construct the Vassiliev spectral sequence converging to
the homology of $\po^{d,m}_n(\R)$, and  compute its $E^1$-terms.
%%(section 5)%%
In \S  \ref{section: loop products}, we construct  loop products  and stabilization maps 
for
the spaces $\po^{d,m}_n(\K)$ $(\K=\R$ or $\C$) and $\Q^{d,m}_n(\R),$ 
and use them to prove Theorem \ref{thm : I}.
%and 
%we also give the second  proof of Proposition \ref{prp : I} by using the comparison theorem of spectral sequences.
%%(section 6)%%
In \S \ref{section: sd} we  prove
the homology stability theorem for the space $\po^{d,m}_n(\R)$ (Theorem \ref{thm: stab1}).
%%(section 7)%%
In \S \ref{section: scanning maps} we consider configuration space models for $\po^{d,m}_n(\K)$
($\K =\R$ or $\C$) and  $\Q^{d,m}_n(\R)$,  define corresponding stabliization maps  and use them to prove the
stable theorems  (Theorems \ref{thm: scanning map} and \ref{thm: stable result2}).
%and recall the important results 
%(Theorems \ref{thm: II} and \ref{thm: Theorem H}).
%%(section 8)%%
In \S \ref{section: the main result} we study the  homotopy type of $\po^{d,m}_n(\R)$
and give the proofs of our main results  
(Theorems \ref{thm: KY13} and \ref{thm: KY13; stable homotopy type}) and their corollaries
(Corollaries \ref{crl: surjection}, \ref{crl: jet embedding}, 
\ref{crl: GKY4 type theorem} and \ref{crl: Q+poH+}).
%%%(SECTION 9)%%%
%In \S \ref{section: (m,n)=(3,1)}, we give the proof of Theorem \ref{thm: KY14}.
%%%(SECTION 9)%%%
Finally in \S \ref{section: appendix} we deal with the case $(m,n)=(1,2)$ for the sake of completeness. 

 %as a completeness of this paper.
%%%%
%%%%
%%%%
%%%%(End of SECTION 1)%%%%%%
%%%%(End of SECTION 1)%%%%%%
%%%%(End of SECTION 1)%%%%%%
%%%
%%%
%%%
%%%(SECTION 2)%%
%%%(SECTION 2)%%
%%%(SECTION 2)%%
%%%%
%%%%
%%%%%%%%%%%%%%%%%%%%%%%%%%%%%%%%%
\section{Basic definitions and the main results}\label{section Definitions}
%%%%%%%%%%%%%%%%%%%%%%%%%%%%%%%%%

Before describing the main results of this paper precisely, we shall recall 
several definitions and facts needed to give the precise statements
of the main results.

%%%(1.2. Basic definitions and notations)%%%%%%
\paragraph{Basic definitions and notations.}
%%%%%%%%%%%%%%%%%%%%%%%
%%%%%
For connected spaces $X$ and $Y,$ let
$\Map(X,Y)$ (resp. $\Map^*(X,Y)$) denote the space
consisting of all continuous maps
(resp. base-point preserving continuous maps) from $X$ to $Y$
with the compact-open topology.
For each element $D\in \pi_0(\Map^*(X,Y))$, let
$\Map^*_D(X,Y)$ denote the path-component of
$\Map^*(X,Y)$ which corresponds to $D$.
When $X$ and $Y$ are complex manifolds,
let $\Hol_D^*(X,Y)\subset \Map_D^*(X,Y)$ denote the subspace
of all based holomorphic maps from $X$ to $Y$.

%%%
Let $\RP^N$ (resp. $\CP^N$) denote the $N$-dimensional
real projective space (resp. $N$-dimensional complex projective space).
Note that the based loop space 
$\Map^*(S^1,\RP^N)=\Omega \RP^N$
has two path-components $\Omega_{\epsilon}\RP^N$
($\epsilon \in \{0,1\}$) for
$N\geq 2$, where
the space $\Omega_0\RP^N$ is
the path-component of null homotopic maps and
$\Omega_1\RP^N$ is the path-component which contains
the natural inclusion of the bottom cell $S^1$ into
$\RP^N$.
Similarly, for each integer $d\in \Z=\pi_0(\Map^*(S^2,\CP^N))$,
let $\Omega^2_d\CP^N=\Map^*_d(S^2,\CP^N)$ denote the path component of
$\Omega^2\CP^N$ of base-point preserving maps
from $S^2$ to $\CP^N$ of degree $d$.

%%%(Definition 2.1)%%%%%%%
\begin{definition}
%%%%%%%
{\rm
Let $\N$ be the set of all positive integers.
From now on, let $d\in \N$, let
$(m,n)\in \N^2$ be a pair of positive integers
such that $(m,n)\not= (1,1)$,  and let
$\F$ be a field with its algebraic closure
$\overline{\F}$.
%%%%%%%%
%%
\par
(i)
%%%
Let $\P_d(\F)$ denote the space of all 
$\F$-coefficients monic polynomials 
$f(z)=z^d+a_1z^{d-1}+\cdots +a_{d-1}z+a_d\in \F [z]$ of degree $d$.
Note that there is a natural homeomorphism
$\P_d(\F)\cong \F^d$ given by
%%%(2.1)%%
\begin{equation}
%%%%%%
f(z)=z^d+\sum_{k=1}^da_kz^{d-k}\mapsto (a_1,\cdots ,a_d).
%%%%%%%
\end{equation}
%%%%%
%given by
%$f(z)=z^d+\sum_{k=1}^da_kz^{d-k}\mapsto (a_1,\cdots ,a_d).$
\par
(ii)
For each $m$-tuple $D=(d_1,\cdots ,d_m)\in \N^m$ of positive integers, 
we denote by
$\po^{D;m}_n(\F)=\po^{d_1,\cdots ,d_m;m}_n(\F)$  the space 
consisting of
all
$m$-tuples
$(f_1(z),\cdots ,f_m(z))\in \P_{d_1}(\F)\times \P_{d_2}(\F)\times
\cdots \times \P_{d_m}(\F)$
of monic polynomials such that
the polynomials
$\{f_j(z)\}_{j=1}^m$ have no common root  in $\overline{\F}$ of multiplicity
$\geq n$.
We call the space $\po^{D;m}_n(\F)$ as 
{\it the space of non-resultant system of bounded multiplicity $n$ with coefficients in $\F$.}%
%%(Footnote 3)%%%%
\footnote{%
%%%%%%%
Recall that the classical resultant of a systems of polynomials vanishes if and only if they have a common solution in an algebraically closed field containing the coefficients. Systems which have no common roots are called \lq\lq non-resultant\rq\rq. This is the intuition behind our choice of the term  \lq\lq non-resultant system of bounded multiplicity.\rq\rq }
%%%(End of Footnote 3)%%%%%%%%
\par
%\par
In particular,
when $D_m=(d,d,\cdots ,d)\in \N^m$ ($m$-times), we write
%%(2.2)%%%%%
\begin{equation}
%%%%%%%%%
\Po^{d,m}_n(\F)=\Po^{D_m;m}_n(\F)
=\po^{d,d,\cdots ,d;m}_n(\F).
%%%%%%%%%%%
\end{equation}
%%%%%%%%%%%
}
%%%%%%%
\end{definition}
%%%%(End of Definition 2.1)%%%
%%
%%
%%%%(Definition 2.2)%%%%
\begin{definition}
%%%%%%%%%
{\rm 
From now on, let us suppose that $\K=\R$ or $\C$. 
%and let
%$D=(d_1,\cdots ,d_m)\in \N^m$ be an $m$-tuple of positive integers.
\par\vspace{1mm}\par
\par
%%(iv)%%%%
\par
(i)
Let $\Q^{d,m}_n(\K)$ denote the space  of
all $m$-tuples
$(f_1(z),\cdots ,f_m(z))\in \P_d(\K)^m$
of  $\K$-coefficients monic polynomials 
of the same degree $d$
such that
the polynomials
$\{f_j(z)\}_{j=1}^m$ have no common \textit{real} root of multiplicity
$\geq n$
(but may have complex common roots of any multiplicity).
%%%%%%%%%%%%
\par
Note that there are the following two inclusions
%%%()%%
%%\begin{equation}
%%%%
%$\Po^{d,m}_n(\C)\supset \Po^{d,m}_n(\R) \subset
%\Q^{d,m}_n(\R).$
%%\end{equation}
%%%%
%Thus there are the following two inclusion maps
%%(2.3)%%
\begin{equation}
\begin{CD}
\Po^{d,m}_n(\C) @<\iota^{d,m}_{n,\C}<\supset<
\Po^{d,m}_n(\R) 
@>\iota^{d,m}_{n,\R}>\subset>
\Q^{d,m}_n(\R).
\end{CD}
%\begin{cases}
%\iota^{d,m}_{n,\C}:\Po^{d,m}_n(\R) 
%\stackrel{\subset}{\longrightarrow}
% \Po^{d,m}_n(\C),
%\\
%\iota^{d,m}_{n,\R}:\Po^{d,m}_n(\R) 
%\stackrel{\subset}{\longrightarrow}
% \Q^{d,m}_n(\R).
%\end{cases}
%%
\end{equation}
%%%%%%
\par
(ii)
For a monic polynomial $f(z)\in \P_d(\K)$, we
define the $n$-tuple $F_n(f)=F_n(f)(z) \in \P_d(\K)^n$
of the monic polynomials of the same degree $d$ by
%%(2.4)%%
\begin{equation}\label{eq: Fn}
%%%%%%%
F_n(f)(z)=
(f(z),f(z)+f^{\prime}(z),f(z)+f^{\prime\prime}(z),
\cdots ,f(z)+f^{(n-1)}(z)).
%%%%%%%%%
\end{equation}
%%%%%%%%%%%%%
Note that $f(z)\in \P_d(\K)$ is not divisible by $(z-\alpha)^n$ for 
some $\alpha\in \K$
if and only if
$F_n(f)(\alpha )\not= {\bf 0}_n$,
where we set
${\bf 0}_n=(0,0,\cdots ,0)\in \K^n.$
\par
%%%%%%
%%(v)%%
(iii)
When $\K=\C$, by identifying $S^2=\C \cup\infty$
we define {\it the  natural map}
%%%%%%
%%(2.5)%%
\begin{align}\label{eq: mapjC}
%%%%%%%%%%
\nonumber
i^{d,m}_{n,\C}&:\Po^{d,m}_n(\C)\to \Omega^2_d\CP^{mn-1}
\simeq \Omega^2S^{2mn-1}
\quad \mbox{by}
%%%
\\
i^{d,m}_{n,\C}({\rm f})(\alpha)
&=
\begin{cases}
[F_n(f_1)(\alpha):F_n(f_2)(\alpha):\cdots :F_n(f_m)(\alpha)]
&
\mbox{if }\alpha\in \C
\\
[1:1:\cdots :1] & \mbox{if }\alpha =\infty
\end{cases}
%%%%%
\end{align}
for ${\rm f}=(f_1(z),\cdots ,f_m(z))\in \Po^{d,m}_n(\C)$
and $\alpha\in \C \cup \infty =S^2$,
where
we choose the points $\infty$ and $*=[1:1:\cdots :1]$ as 
the base-points of $S^2$ and $\CP^{mn-1}$, respectively.
%%%
}
%%%%%%%
\end{definition}
%%(End of Definition 2.2)%%%
%%
%%
%%%(Definition 2.3)%%%%%
\begin{definition}
{\rm
%%(vi)%%%%%%
Let $\Z_2=\{\pm 1\}$ denote the (multiplicative) cyclic group of order $2$.
From now on,
we will regard the two spaces $S^2=\C\cup \infty$ and $\CP^{mn-1}$ as
$\Z_2$-spaces with actions  induced by the complex conjugation on $\C$.
\par\vspace{2mm}\par
(i)
Let  $(\Omega^2_d\CP^{mn-1})^{\Z_2}$ denote the space 
consisting of 
all $\Z_2$-equivariant based maps
$f:(S^2,\infty )\to( \CP^{mn-1},*)$.
\par
(ii)
Since $\Po^{d,m}_n(\R)\subset \Po^{d,m}_n(\C)$ and
$i^{d,m}_{n,\C}(\Po^{d,m}_n(\R))\subset
(\Omega^2_d\CP^{mn-1})^{\Z_2}$,
we  also define {\it the natural map}
%%(2.6)%%
\begin{align}\label{eq: mapjR}
%%%
\nonumber
& \quad
i^{d,m}_{n,\R}:\Po^{d,m}_n(\R) \to (\Omega_d^{2}\CP^{mn-1})^{\Z_2}
%%%%%%%%%%
%\end{equation}
%%%%%%%%%
%\qquad
%\mbox{by the restriction}
\\
\nonumber & \quad \quad \mbox{by the restriction}
%%%%%%%%
\\
i^{d,m}_{n,\R}&=i^{d,m}_{n,\C}\vert
\Po^{d,m}_n(\R):\Po^{d,m}_n(\R)\to
(\Omega^2_d\CP^{mn-1})^{\Z_2}.
\end{align}
%%%
\par
(iii)
%%()%%%%%
When %$\K =\R$ and  
$mn\geq 3$, by identifying $S^1=\R\cup \infty$ 
we define {\it a natural map}
%%%%%
%%(2.7)%%
\begin{align}
%%%%
\nonumber
i^{d,m}_n:&\Q^{d,m}_n(\R)\to \Omega_{[d]_2}\RP^{mn-1}\simeq \Omega S^{mn-1}
\qquad
\mbox{by}
%%%
%\end{equation}
%%%%%
%%()%%
%\begin{equation*}
%%%
\\
i^{d,m}_{n}({\rm f})(\alpha)
&=
\begin{cases}
[F_n(f_1)(\alpha):F_n(f_2)(\alpha):\cdots :F_n(f_m)(\alpha)]
&
\mbox{if }\alpha\in \R
\\
[1:1:\cdots :1] & \mbox{if }\alpha =\infty
\end{cases}
%%%%%
\end{align}
for ${\rm f}=(f_1(z),\cdots ,f_m(z))\in \Q^{d,m}_n(\R)$
and $\alpha\in \R \cup \infty =S^1$, where
$[d]_2\in \{0,1\}$ is the integer $d$ mod $2$ and
we choose the points $\infty$ and $*=[1:1:\cdots :1]$ as the base-points
of $S^1=\R\cup\infty$ and $\RP^{mn-1}$, respectively.
%%%
\par
(iv) For positive integer $n\geq 3$, we define {\it the jet embedding}
%%(2.8)%%
\begin{align}\label{eq: jet embedding}
%%()%%
 &j^d_n:\po^{d,1}_n(\R) \to \po^{d,n}_1(\R)
 \\
 \nonumber &
\quad\quad
\mbox{ by }
\\
\nonumber 
\quad j^d_n(f(z))&=(f(z),f(z)+f^{\p}(z),
f(z)+f^{\p\p}(z),\cdots ,f(z)+f^{(n-1)}(z))
\end{align}
%%%
for $f(z)\in \po^{d,1}_n(\R)$.
}
\end{definition}
%%%(End of Definition 2.3)%%%
%%
%%
%%%(Definition 2.4)%%%%
\begin{definition}
%%%%
{\rm
(i)
Let 
$f:X\to Y$ be a base-point preserving map between based spaces $X$ and $Y$.
The map $f$ 
is called {\it a homotopy equivalence  through dimension} $N$
(resp. {\it a homology equivalence through dimension} $N$)
if the induced homomorphism
$$
f_*:\pi_k(X)\to\pi_k(Y)
\quad
(\mbox{resp. }f_*:H_k(X;\Z) \to H_k(Y;\Z))
$$
is an isomorphism for any integer $k\leq N$
%(resp. an isomorphism for any $k<N$ and a epimorphism for $k=N$).
%\par
%Similarly, 
%the map $f$ 
%is called {\it a homology equivalence  through dimension} $N$
%(resp. {\it a homology equivalence up to dimension} $N$)
%if the induced homomorphism
%$
%f_*:H_k(X;\Z) \to H_k(Y;\Z)
%$
%is an isomorphism for any integer $k\leq N$.
%resp. an isomorphism for any $k<N$ and a epimorphism for $k=N$).
%%%%
\par
(ii)
%Similarly,
Let $G$ be a group and $f:X\to Y$ be a
$G$-equivariant base-point preserving map between $G$-spaces $X$ and $Y$.
\par
Then the map $f$ is called
 {\it a $G$-equivariant homotopy equivalence through dimension} $N$
(resp. a {\it a $G$-equivariant homology equivalence through dimension} $N$)
if the restriction map 
$$
f^H=f\vert X^H:X^H\to Y^H
$$
is a
homotopy  equivalence through dimension $N$ 
(resp. a homology equivalence through dimension $N$) for any subgroup
$H\subset G$, where $W^H$ denotes the $H$-fixed subspace of a $G$-space $W$ given by
%%%(2.9)%%
\begin{equation}
W^H=\{x\in W: h\cdot x=x\mbox{ for any }h\in H\}.
%%%%%%%
\end{equation}
%%%%%%
%}
%%%
%\end{definition}
%%%%%(End of Definition 1.4)%%%
%%
%\paragraph{Some known results. }
%%%
\par (iii)
Recall \cite{Ja}  that there is a following homotopy equivalence
for $N\geq 2$ obtained by using the reduced product 
%%%(2.10)%%
\begin{equation}
%%%%
\Omega S^{N+1}\simeq
S^N \cup e^{2N}\cup e^{3N}\cup \cdots
\cup e^{(k-1)N}\cup e^{kN}\cup
e^{(k+1)N}\cup
\cdots 
\end{equation}
We denote by $J_k(S^N)$ the $kN$-skeleton of $\Omega S^{N+1}$, i.e.
%%(2.11)%%%%
\begin{equation}
%%%%%%%%%%%%
J_k(S^N)=S^N\cup e^{2N}\cup e^{3N}\cup \cdots \cup e^{(k-1)N} \cup e^{kN},
\end{equation}
%%%
which is usually called 
{\it the $k$-th stage James filtration} of
$\Omega S^{N+1}$.
%%%%%
}
\end{definition}
%%%(End of Definition 2.4)%%%
%%
%%
%%(Some known results)%%%
\paragraph{Some known results. }
%%%
Remark that there are homeomorphisms 
%%(2.12)%%
\begin{equation}\label{eq: contractible}
%%%%%
\po^{d,m}_n(\K)\cong \K^{dm}
\ \ (\K=\R,\ \C),\ \mbox{and}
\quad
\Q^{d,m}_n(\R)\cong \R^{dm}
\quad
\mbox{if }d<n.
\end{equation}
Thus, these spaces are contractible if $d<n$.
From now on, in this paper we always assume that $d$ is the positive integer such that
%%(2.13)%%
\begin{equation}
d\geq n.
\end{equation}
\par\vspace{1mm}\par
Now recall the following known two results.
%%
%%%(Theorem 2.5 (Theorem  of KY8)%%
\begin{thm}[\cite{KY8}]\label{thm: KY8}
%%%%%%%%
%\begin{enumerate}
%%(i)%%
Let $m,n\geq 1$ be positive integers such that $mn\geq 3$.
\par\vspace{1mm}\par
%\begin{enumerate}
%\item[$\I$]
%%%%%%
$\I$
The natural map
$$
i^{d,m}_{n,\C}:\Po^{d,m}_n(\C)\to \Omega^2_d\CP^{mn-1}
\simeq \Omega^2S^{2mn-1}
$$
is a homotopy equivalence through dimension
$D(d;m,n;\C)$, where $\lfloor x\rfloor$ denotes the integer part of a
real number $x$ and
the positive integer $D(d;m,n;\C)$ is defined by
%%(2.14)%%
\begin{equation}\label{eq: D(d;m,n)}
%%%%%%%%%
D(d;m,n;\C)=(2mn-3)(\lfloor d/n\rfloor +1)-1.
\end{equation}
%%%
%%%%(ii)%%
\par
$\II$
%\item[$\II$]
%%%%
There is a homotopy equivalence
%%(2.15)%%
\begin{equation}\label{eq: poly stability for C}
\po^{d,m}_n(\C)\simeq
\po^{\lfloor d/n\rfloor ,mn}_1(\C).
%\cong\Hol^*_{\lfloor d/n\rfloor}(S^2,\CP^{mn-1}).
%%%%
\end{equation}
%%%%
%%
%%%(iii)%%
%\item[$\III$]
%%%%
\par
$\III$
There is a stable homotopy equivalence
%%(2.16)%%
\begin{equation}\label{eq: the space poly}
%%%
\Po^{d,m}_n(\C)\simeq_s
\bigvee_{j=1}^{\lfloor d/n\rfloor}
\Sigma^{2(mn-2)k}D_j,
%%%%
\end{equation}
%%%
where $\Sigma^j$ and
$D_j=D_j(S^1)
=F(\C,j)_+\wedge_{S_j}(S^1)^{\wedge j}$
%=F(\C,k)_+\wedge_{S_k}S^k$ 
denote
the $j$-fold reduced suspension and
 the space defined by $($\ref{eq: Dd}$)$, respectively.
 \qed
%\end{enumerate}
%\end{enumerate}
\end{thm}
%%%%%%(End of Theorem 2.5)%% 

%%%(Theorem 2.6: Theorem of KY10)%%
\begin{thm}[\cite{KY10}, \cite{Va}, \cite{Y1}]\label{thm: KY10}
%%%%%%%%
%%%%%%%%
$\I$
%\begin{enumerate}
%%(i)%%
%\item[$\I$] 
The natural map
$$
i^{d,m}_n:\Q^{d,m}_n(\R)\to \Omega_{[d]_2}\RP^{mn-1}\simeq \Omega S^{mn-1}
$$
is a homotopy equivalence through dimension
$D(d;m,n)$ if $mn\geq 4$ and a homology equivalence through dimension $D(d;m,n)$ if $mn=3$,\footnote{%
%%%(Footnote 4)%%%%%%
If $mn\geq 4$, the assertion (i) follows from \cite{KY10}.
The assertion (i) follows from \cite{Va0} and
\cite[Theorem 3 (page 88)]{Va} if $(m,n)=(1,3)$,
and it follows from \cite{Y1} if $(m,n)=(3,1)$.
} 
%%(End of Footnote 4)%%%%%%
%%%%%%%%%%%%%%%%%%%
where the positive integer $D(d;m,n)$ is defined by
%%(2.17)%%
\begin{equation}\label{eq: D(d;m,n)}
%%%
D(d;m,n)=(mn-2)(\lfloor d/n\rfloor +1)-1
\quad
(\mbox{as in }\mbox{\rm $($\ref{eq: 1.1}$)$}).
\end{equation}
%%%%(ii)%%
%\item[$\II$]
%%%%%%%%%
\par
$\II$
If $mn\geq 4$,
there is a homotopy equivalence
%%(2.18)%%
\begin{equation}\label{eq: James}
%%%%%%%
\Q^{d,m}_n(\R)\simeq J_{\lfloor d/n\rfloor}(S^{mn-2}).
\qquad
\qed
\end{equation}
%%%
\end{thm}
%%%%%%%%(End of Theorem 2.6)%%%

%%%(The main result)%%%%
\paragraph{The main results. }
%%%%%%%%%%%%%%%%%%%%%%%%
The main purpose of this paper is to investigate the homotopy type of the space $\po^{d,m}_n(\R)$ and to prove an Atiyah-Jones-Segal type result 
for it.
%%%%%%%%%%
\par
First, we consider the unstable homotopy type of the space
$\po^{d,m}_n(\R)$.
For this purpose, the key observation will be the homotopy equivalence (see Lemma \ref{lmm: A3} below)
%%(2.19)%%
\begin{equation}
%%%%%%%%
(\Omega^2_d\CP^{N})^{\Z_2}\simeq \Omega^2S^{2N+1}\times \Omega S^{N}
\qquad
\mbox{for }N\geq 2.
%%%%%%%%
\end{equation}
%%%%%%
We will use it to prove the following result, closely related to Theorems \ref{thm: KY8} and \ref{thm: KY10}. 

%%%%%(Theorem 2.7)%%
\begin{thm}\label{thm: KY13}
%%%%%%%%%%%%%
The natural map
$$
i^{d,m}_{n,\R}:\Po^{d,m}_n(\R)\to (\Omega^2_d\CP^{mn-1})^{\Z_2}
\simeq
\Omega^2S^{2mn-1}\times \Omega S^{mn-1}
$$
is a homotopy equivalence through dimension
$D(d;m,n)$ if $mn\geq 4$ and it is
a homology equivalence through dimension $D(d;m,n)$
if $mn=3$.
%%%%%%%
\end{thm}
%%%%%%%(End of Theorem 2.7)%%%
%%
%%
%%%
%Recall that a connected space $X$ is called {\it simple up to dimension} $N$
%(resp. {\it simple}) if
%the action of the fundamental group $\pi_1(X)$ on the homotopy group
%$\pi_k(X)$ is trivial for any $k<N$ (resp. for any $k\geq 1$).
%%
%%%
%%%%%%(Theorem 2.8)%%
%\begin{thm}\label{thm: KY14}
%%%%%%%%%%%%%%%%
%Let $(m,n)=(3,1)$.
%Then the space
%$\po^{d,3}_1(\R)$ is simple if 
%if $d\equiv 1$ $\mbox{$($\rm mod $2)$}$ 
%and it is simple up to dimension $d$
%if $d\equiv 0$ $\mbox{$($\rm mod $2)$}.$
%%%%
%\end{thm}
%%%%(End of Theorem 2.8)%%
%%%%
%%%%
%Since $D(d;3,1)=d$,
%by using Theorems \ref{thm: KY13} and \ref{thm: KY14} we also obtain the following result.
%%%%
%%%%(Corollary 2.9)%%%
%\begin{crl}\label{crl: KY14}
%%%%%%%%%%%%%%
%Let $(m,n)=(3,1)$.
%Then
%the natural map
%$$
%i^{d,3}_{1,\R}:\Po^{d,3}_1(\R)\to (\Omega^2_d\CP^{2})^{\Z_2}
%\simeq
%\Omega^2S^{5}\times \Omega S^{2}
%\simeq
%\Omega^2S^5\times \Omega S^3\times S^1
%$$
%is a homotopy equivalence through dimension
%$d$ if $d\equiv 1$ $\mbox{$($\rm mod $2)$}$ 
%and a homotopy equivalence up to dimension
%$d$ if $d\equiv 0$ $\mbox{$($\rm mod $2)$}.$
%\qed
%%%%%%%%
%\end{crl}
%%%%(End of Corollary 2.9)%%%
%%%%%%%%%%%%%%%%%%
\par\vspace{1mm}\par
Since   $\po^{d,m}_n(\C)^{\Z_2}=\po^{d,m}_n(\R)$ and
$(i^{d,m}_{n,\C})^{\Z_2}=i^{d,m}_{n,\R}$, by using
Theorems \ref{thm: KY8} and
\ref{thm: KY13}, 
%and Corollary \ref{crl: KY14}, 
we also obtain the following:
%%%%%(Corollary 2.8)%%
\begin{crl}\label{crl: KY13}
%%%%%%%%%%%%%
%$\I$
%If $mn\geq 4$, 
%or $(m,n)=(3,1)$ with $d\equiv 1$ $\mbox{$($\rm mod $2)$}$,
The natural map
$$
i^{d,m}_{n,\C}:\Po^{d,m}_n(\C)\to \Omega^2_d\CP^{mn-1}\simeq \Omega^2S^{2mn-1}
$$
is a $\Z_2$-equivariant homotopy equivalence through dimension
$D(d;m,n)$ if $mn\geq 4$, and it is a is a $\Z_2$-equivariant homology equivalence through dimension
$D(d;m,n)$ if $mn=3$.
%%%
%\par
%$\II$
%If $(m,n)=(3,1)$ and $d\equiv 0$ $\mbox{$($\rm mod $2)$}$,
%the natural map
%$$
%i^{d,3}_{1,\C}:\Po^{d,3}_1(\C)\to \Omega^2_d\CP^{2}\simeq \Omega^2S^{5}
%$$
%is a $\Z_2$-equivariant homotopy equivalence up to dimension
%$D(d;3,1)=d$.
%%%
%\par
%$\III$
%If $(m,n)=(1,3)$,
%the natural map
%$$
%i^{d,1}_{3,\C}:\Po^{d,1}_3(\C)\to \Omega^2_d\CP^{2}\simeq \Omega^2S^{5}
%$$
%is 
%a $\Z_2$-equivariant homology equivalence through dimension
%$D(d;1,3)=\lfloor d/3\rfloor$
\qed
%%%%%%%
\end{crl}
%%%%%%%(End of Corollary 2.10)%%%
%%%
%%%
%%%(Corollary 2.11)%%%
\begin{crl}\label{crl: surjection}
%%%%%%%%%%%
Let $mn\geq 3$, let 
$I^{d,m}_n:\po^{d,m}_n(\R)\to \Omega^2S^{2mn-1}\times \Omega S^{mn-1}$ 
denote the map defined by $($\ref{eq: map I}$)$, and $q_k$ is the projection to the $k$-th factor
for $k=1,2$
given by $($\ref{eq: projections}$)$.
%%%%
\par\vspace{1mm}\par
%\begin{enumerate}
%\item[$\I$]
$\I$
The map
$
q_1\circ I^{d,m}_n
:
\po^{d,m}_n(\R)\to \Omega^2S^{2mn-1}
$
induces an epimorphism on the homology group
$H_k(\ ;\Z)$ for any $k\leq D(\lfloor d/2\rfloor;m,n;\C)$.
%%
%\item[$\II$]
%%
\par
$\II$
 The map
$
q_2\circ I^{d,m}_n
:
\po^{d,m}_n(\R)\to \Omega S^{mn-1}
$
induces an epimorphism on the homology group
$H_k(\ ;\Z)$ for any $k\leq D(d;m,n)$.
%%
%\end{enumerate}
%%
\end{crl}
%%%(End of Corollary 2.11)%%%
%%%
%%%
%%%(Corollary 2.12)%%
\begin{crl}\label{crl: jet embedding}
%%%%%%%%%%%%%
The jet embedding (defined by (\ref{eq: jet embedding}))
$$
j^d_n:\po^{d,1}_n(\R)\to \po^{d,n}_1(\R)
$$
is a homotopy equivalence through dimension
$(n-2)(\lfloor d/n\rfloor +1)-1$ if $n\geq 4$, and  a homology equivalence
through dimension $\lfloor d/3\rfloor$ if $n=3$. 
%where
%$D(d;1,n)=(n-2)(\lfloor d/n\rfloor +1)-1$.
%%%%%
\end{crl}
%%(End of Corollary 2.12)%%
%%
%%
Next, we consider the stable homotopy type of the space
$\po^{d,m}_n(\R)$.
%%
%%
%%
%%%
%%(Theorem 2.13)%%
\begin{thm}\label{thm: KY13; stable homotopy type}
%%%%
If $mn\geq 3$, there is a stable homotopy equivalence
$$
\po^{d,m}_n(\R)
\simeq_s
\po^{\lfloor d/2\rfloor, m}_n(\C) \vee B^{d,m}_n
\vee \Q^{d,m}_n(\R),
$$
where
%$D_j=D_j(S^1)
%=F(\C,j)_+\wedge_{S_j}(S^1)^{\wedge j}$
%%=F(\C,k)_+\wedge_{S_k}S^k$ 
%denotes the space defined by $($\ref{eq: Dd}$)$, and 
the space $B^{d,m}_n$ is defined by
%%%
%%(2.20)%%
\begin{align}\label{eq: the space B}
%%%%%%%
B^{d,m}_n
%&=
%\bigvee_{i+2j\leqq \lfloor d/n\rfloor;\ i,j\geq 1}
%S^{(mn-2)i}\wedge \Sigma^{2(mn-2)j}D_j
%\\
%\nonumber
&=
\bigvee_{i,j\geq 1,i+2j\leqq \lfloor d/n\rfloor}\Sigma^{(mn-2)(i+2j)}D_j.
\end{align}
\end{thm}
%%%(End of Theorem 2.13)%%%
%%%%
By using the stable homotopy equivalences (\ref{eq: the space poly}),
 (\ref{eq: stable Q}) and 
 the equality 
 $\lfloor \lfloor d/n\rfloor /2\rfloor =\lfloor \lfloor d/2\rfloor/n\rfloor$
 (by (iii) of
 Lemma \ref{lmm: inequality}),
%and Theorem \ref{thm : I},
 we obtain the following result.
 %%%
 %%%
%%%(Corollary 2.14)%%
\begin{crl}\label{crl: KY13; stable homotopy type}
%%%%%%%%%%%%%
If $mn\geq 3$,
%and $d_0=\lfloor d/2\rfloor$, 
 there is a stable homotopy equivalence
\begin{align*}
\po^{d,m}_n(\R)
%&\simeq_s
%\Big(\bigvee_{i=1}^{\lfloor d/n\rfloor}S^{(mn-2)i}\Big)
%\vee
%\Big(\bigvee_{j=1}^{\lfloor d_0/n\rfloor}
%\Sigma^{2(mn-2)j}D_j\Big)\vee
%B^{d,m}_n
%\\
&\simeq_s
\Big(\bigvee_{i=1}^{\lfloor d/n\rfloor}S^{(mn-2)i}\Big)
\vee
\Big(
\bigvee_{i\geq 0,j\geq 1, i+2j\leq \lfloor d/n \rfloor}
\Sigma^{(mn-2)(i+2j)}D_j\Big).
\qquad
\qed
\end{align*}
%%%
%%%%
\end{crl}
%%(End of Corollary 2.14)%%
%%%%
%%
%%
In particular, we obtain the following result, analogous to (\ref{eq: poly stability for C}).
%%
%%
%%%(Corollary 2.15)%%
\begin{crl}\label{crl: GKY4 type theorem}
%%%%%%
If $mn\geq 3$, there is a stable homotopy equivalence
%%(2.21)%%
\begin{equation}\label{eq: KY13 stable equiv}
\po^{d,m}_n(\R)
\simeq_s
\po^{\lfloor d/n\rfloor,mn}_1(\R).
\end{equation}
%%%%%%
\end{crl}
%%%(End of Corollary 2.15)%%
%%
%%
%%
By using Theorems \ref{thm: KY13; stable homotopy type}
and  \ref{thm : I}, we also have the following result.
%%
%%
%%%(Corollary 2.16)%%
\begin{crl}\label{crl: Q+poH+}
%%%
Let $mn\geq 3$.
%\begin{enumerate}
%\item[$\I$]
\par
$\I$
The inclusion map
%%%
$
\iota^{d,m}_{n,\R}:\Po^{d,m}_n(\R)
\stackrel{\subset}{\longrightarrow} \Q^{d,m}_n(\R)
$
induces a split epimorphism on the homology group $H_*(\ ;\Z)$.
%%
%\item[$\II$] 
%%%
\par
$\II$
The inclusion map
$
\iota^{d,m}_{n,\Ha_+}:\Po^{d,m}_n(\R;\Ha_+)
\stackrel{\subset}{\longrightarrow} \po^{d,m}_n(\R)
$
induces a split monomorphism on the homology group $H_*(\ ;\Z)$.%
%%%(FootNote 5)%%%
\footnote{
Here $\Ha_+$ denotes the upper-half plane in the complex plane $\C$.
The map
$\iota^{d,m}_{n,\Ha_+}$ and the space $\Po^{d,m}_n(\R;\Ha_+)$ will be defined in 
Definition \ref{def: Hol(R;H)}.
%\S \ref{section 2} 
}
%%(End of Footnote 5)%%
%\end{enumerate}
%%%%%%%%%%
%%
\end{crl}
%%(End of Corollary 2.16)%%%%
%%
%%%
%%%
%%(End of SECTION 2)%%%
%%(End of SECTION 2)%%%
%%(End of SECTION 2)%%%
%%%
%%%
%%%
%%%
%%%(SECTION 3)%%%
%%%(SECTION 3)%%%
%%%(SECTION 3)%%%%%%%%%%%%%%%%%%%%%%%%%%%%%%%
\section{The spaces $(\Omega^2_d\CP^{N})^{\Z_2}$ and $\Po^{d,m}_n(\R;\Ha_+)$}
\label{section 3}
%%%%%%%%%%%%%%%%%%%%%%%%%%%%%%%%%%%%%%%%%
%%%
%%%
\par
In this section, 
we shall investigate the homotopy type of the spaces
$(\Omega^2_d\CP^{N})^{\Z_2}$ and $\Po^{d,m}_n(\R;\Ha_+)$.
\par\vspace{1mm}\par
%%%%%
We first consider the space $(\Omega^2_d\CP^{N})^{\Z_2}$.
By identifying $S^2=\C\cup \infty$,  $S^2$ acquires a $\Z_2$-action induced by complex conjugation on $\C$.
The space
 $\CP^N$  also has a natural $\Z_2$-action induced by  complex conjugation. 
Thus, we can consider the space $(\Omega^2_d\CP^{N})^{\Z_2}$  of 
based $\Z_2$-equivariant maps
from $S^2$ to $\CP^N.$
%%
%%%%
%%%%
\par\vspace{2mm}\par
%%(Definition 3.1)%%
\begin{definition}
%%%%
{\rm
Let $(X,A)$ and $(Y,B)$ be pairs of connected based spaces.
\par
(i)
We denote by $\Map^*(X,A;Y,B)$
the subspace of $\Map^* (X,Y)$ consisting of all based maps
$f\in \Map^* (X,Y)$ such that $f(A)\subset B$.
\par
(ii)
Let $\Map^*_0(X,Y)$ denote the path component of $\Map^*(X,Y)$
containing null-homotopic maps.
Then
for a based map $g\in \Map^*_0(A,Y)$,
let $F(X,A;Y;g)\subset \Map^*_0(X,Y)$ denote the subspace 
%%(3.1)%%
\begin{equation}
F(X,A;Y;g)=\{f\in \Map^*_0(X,Y):f\vert A=g\}.
\end{equation}
%%%
\par (iii)
The pair $(X,A)$ is called a NDR-pair if the inclusion
$i_A:A\stackrel{\subset}{\longrightarrow} X$ is a cofibration.
In this situation, we have a cofibration sequence
%%(3.2)%%
\begin{equation}\label{eq: cofib}
%%%%%%%%
A \stackrel{i_A}{\longrightarrow} X \stackrel{q_A}{\longrightarrow}X/A,
\end{equation}
%%%
where $q_A:X\to X/A$ denotes the natural pinching map.
%%%
}
%%%%%%
\end{definition}
%%(End of Definition 3.1)%%
%%
%%
%%%%%(Lemma 3.2)%%
%%%%%%%%%%%%%%%%%%%%%%%%%%%%%%%
\begin{lemma}\label{lmm: induced fib}
If $(X,A)$ is a NDR-pair, for each pair of spaces $(Y,B)$ the following 
is a fibration sequence
%%
%%%%(3.3)%%%%%%%
\begin{equation}\label{eq: induced fibration3}
%%%
\Map^*(X/A,Y) \stackrel{}{\longrightarrow} \Map^*(X,A;Y,B)
\stackrel{r_A}{\longrightarrow}
\Map^*(A,B),
\end{equation}
%%%%
where the two maps $q^{\#}_A$ and $r_A$ are defined by
$q^{\#}_A(f)=f\circ q_A$ and
$r_A(g)=g\circ i_A=g\vert A$
%$$
%q^{\#}_A(f)=f\circ q_A
%\quad
%\mbox{and}\quad
%r_A(g)=g\circ i_A=g\vert A
%$$
for $(f,g)\in\Map^*(X/A,Y)\times\Map^*(X,A;Y,B).$
\end{lemma}
%%%(Proof of Lemma 3.2)%%%%
\begin{proof}
%%%%%%%%%%%%%%%%%%%%%%%%%
By (\ref{eq: cofib}) we obtain a fibration sequence
%%
%%
%%(3.4)%%
\begin{equation}\label{eq: induced fibration1}
%%%%%%%%
\Map^*(X/A,Y) \stackrel{q_A^{\#}}{\longrightarrow} \Map^*(X,Y)
\stackrel{re_A}{\longrightarrow}
\Map^*(A,Y).
\end{equation}
Let $i_B:B\to Y$ be an inclusion map and let
$i_{B\#}:\Map^*(A,B)\to\Map^*(A,Y)$ denote the map
given by
$i_{B\#}(h)=i_B\circ h$ for $h\in \Map^*(A,Y).$
%%%
\par
Now consider the following  fibration  induced from (\ref{eq: induced fibration1}) by
the map $i_{B\#}$
%%%
%%%%(3.5)%%
\begin{equation}\label{eq: induced fibration2}
%%%
\begin{CD}
\Map^*(X/A,Y) @>>> E
@>{r_A}>>
\Map^*(A,B)
\\
\Vert @. @VVV @V{i_{B\#}}VV
\\
\Map^*(X/A,Y) @>q_A^{\#}>> \Map^*(X,Y)
@>{re_A}>>
\Map^*(A,Y),
\end{CD}
\end{equation}
%%%%%%
%where
%%
\begin{eqnarray*}
%%%
\mbox{ where}\quad\quad
E&=&
\{(f,g)\in \Map^*(X,Y)\times \Map^*(A,B):re_A(f)=i_{B\#}(g)\}
\\
&=&
\{(f,g)\in \Map^*(X,Y)\times \Map^*(A,B):f\vert A=g\}
\\
&=&
\Map^*(X,A;Y,B)
%%%
\end{eqnarray*}
%%%
Thus, we have obtained the fibration sequence
(\ref{eq: induced fibration3}).
\end{proof}
%%%%(End of proof of Lemma 3.2)%%%
%%%%
%%%%
%%%%
%%(Lemma 3.3)%%
\begin{lemma}\label{lmm: A3}
%%%%%%%%
If
%$\tilde{i}:(\Omega^2_d\CP^N)^{\Z_2}\to \Omega^2\CP^N$ denotes the inclusion map and
 $N\geq 2$, there is a homotopy equivalence
%%(3.6)%%
\begin{equation}\label{eq: homotopy equiv}
%%%
(\Omega^2_d\CP^N)^{\Z_2}
\simeq 
\Omega^2 S^{2N+1}\times \Omega S^N.
\end{equation}
\end{lemma}
%%%%%%
\begin{proof}
%%%%(Proof of Lemma 3.3)%%%
By applying (\ref{eq: induced fibration3}) to the case 
$(X,A;Y,B)=(D^2,S^1;\CP^N,\RP^N)$, we obtain the following fibration sequence
%%%%(3.7)%%%%
\begin{equation}\label{eq: fibration4}
%%%%%%%%%%%%%
\Omega^2_d \CP^N \to 
\Map^*_d(D^2,S^1;\CP^N,\RP^N)
\stackrel{r_{S^1}}{\longrightarrow}
\Omega_{[d]_2}\RP^N.
%%%%%%%%%%%%
\end{equation}
%%%
It follows from the proof of Lemma \ref{lmm: induced fib}
that  (\ref{eq: fibration4}) is the induced fibration from the map
$\Omega j:\Omega_{[d]_2}\RP^N\to \Omega_{[d]_2}\CP^N$,
where
$j:\RP^N\to \CP^N$ denotes the inclusion map.
Let $\tilde{j}:S^N\to S^{2N+1}$ denote the natural inclusion,
and
let $\gamma_N:S^{2N+1}\to \CP^N$ and
$\gamma_{N,\R}:S^N\to \RP^N$ be the Hopf fibering
and the usual double covering, respectively.
%%
%Then
Consider the following commutative diagram
$$
\begin{CD}
\Omega S^N @>\Omega \tilde{j}>> \Omega S^{2N+1}
\\
@V{\Omega\gamma_{N,\R}}V{\simeq}V @VV{\Omega\gamma_N}V
\\
\Omega_0 \RP^N @>\Omega j>> \Omega_0\CP^N
\end{CD}
$$
Since the map $\tilde{j}$ is null-homotopic,
the map $\Omega j:\Omega_0\RP^N\to \Omega_0\CP^N$
is also null-homotopic and hence also the map $j$.
Thus, (\ref{eq: fibration4}) is a trivial fibration and there is a
homotopy equivalence
%%()%%
\begin{equation*}\label{eq: homeo}
%%%%%%%%
\Map^*_d(D^2,S^1;\CP^N,\RP^N)\simeq
\Omega^2_d\CP^N\times \Omega_{[d]_2}\RP^N
\simeq
\Omega^2S^{2N+1}\times \Omega S^N.
%%%%%%
\end{equation*}
%%%%%%
Hence, we have a homotopy equivalence
%%(3.8)%%
\begin{equation}\label{eq: homotopy e.}
%%%
\Map^*_d(D^2,S^1;\CP^N,\RP^N)
\simeq
\Omega^2S^{2N+1}\times \Omega S^N
\quad
\mbox{for any }d\in \Z.
\end{equation}
%%%%%
Now 
we  regard $S^2$ as the union $S^2=D^2_+\cup D^2_{-}$, 
where  
%let us write $\Vert \textbf{\textit{x}}\Vert =\sqrt{x^2+y^2}$
%for $\textbf{\textit{x}}=(x,y)\in \R^2$, 
%and
let
$D^2_+$ and $D^2_{-}$ denote the northern hemisphere and the southern one  given by
%%%
\begin{align*}
D^2_+&=\{\textbf{\textit{x}}=(x,y)\in \R^2:\Vert \textbf{\textit{x}}\Vert \leq 1,y\geq 0\},
\\
D^2_{-}&=\{\textbf{\textit{x}}=(x,y)\in \R^2:\Vert \textbf{\textit{x}}\Vert \leq 1,y\leq 0\}.
\end{align*}
%%%
%%
%Since $\Z_2$-action on $\CP^N$ is induced by complex conjugation,
%by identifying $D^2_+\cong\Ha_+\cup \infty$, 
Then
each $\Z_2$-equivariant map $f:S^2\to \CP^N$ can be identified with the map
$(D^2_+,S^1)\to (\CP^N,\RP^N).$ 
Then by identifying $D^2\cong D^2_+$
there is a homeomorphism
%%(3.9)%%
\begin{equation}\label{eq: space of maps}
%%%%%%%
(\Omega^2_d\CP^N)^{\Z_2}\cong \Map^*_d(D^2,S^1;\CP^N,\RP^N).
\end{equation}
Hence, the assertion easily follows from  (\ref{eq: homotopy e.}) and 
(\ref{eq: space of maps}).
%%%%
\end{proof}
%%%(End of proof of Lemma 3.3)%%%%
%%%
%%%
%%%
%%%(Remark 3.4)%%
\begin{remark}\label{remark: d}
%%%%%%%%%%%%
{\rm
It follows from (\ref{eq: homotopy equiv}) and (\ref{eq: homotopy e.}) that
the homotopy types of the spaces
$(\Omega^2_d\CP^N)^{\Z_2}$ and $\Map^*_d(D^2,S^1;\CP^N,\RP^N)$
do not depend on the choice of the integer $d$. 
%Moreover, by (\ref{eq: homotopy e.}), we also see that there is a homotopy equivalence
%%%(2.10)%%
%\begin{equation}
%%%%
%\Map^*_d(D^2,S^1;\CP^N,\RP^N)
%\simeq
%\Map^*_0(D^2,S^1;\CP^N,\RP^N)
%%%
%\end{equation}
%%%%%%%%
%for any $d\in \Z$.
\qed
}
%%%%
\end{remark}
%%%%%%%
%%%%(End of Remark 3.4)%%
%%
%%
%%%(Lemma 3.5)%%%
\begin{lemma}\label{lmm: restriction}
%%%%%
If $N\geq 1$, there is a homotopy equivalence
%%(3.10)%%
\begin{equation}\label{eq: restriction eq}
%%%%%%
F(D^2,S^1;\CP^N;g)\simeq \Omega^2_0\CP^N\simeq \Omega^2S^{2N+1}
\quad
\mbox{for any $g\in \Omega_0\CP^N$.}
\end{equation}
%%%%%
%for any $g\in \Omega_0\CP^N$.
%%%
\end{lemma}
%%%%%%%
\begin{proof}
%%%%%%%
Let $r:\Map^*_0(D^2,\CP^N)\to \Map_0^*(S^1,\CP^N)=\Omega_0\CP^N$ denote
the restriction map given by $r(f)=f\vert S^1$.
It is easy to see that $r$ is a fibration.
Moreover, since $D^2$ is contractible, the space
$\Map^*_0(D^2,\CP^N)$ is contractible.
Thus, each fiber of $r$ is homotopy equivalent to the space
$\Omega^2_0\CP^N\simeq \Omega^2S^{2N+1}$.
Hence, there is a homotopy equivalence
$$
F(D^2,S^1;\CP^N;g)=r^{-1}(g)\simeq
\Omega^2_0\CP^N\simeq \Omega^2S^{2N+1},
$$
and this completes the proof.
%%%
\end{proof}
%%(End of proof of Lemma 3.5)%%
%%
%%
%%
%%
%%
%%(Lemma 3.6)%%%
\begin{lemma}\label{lmm: key diagram}
%%%%%%%%%%%%%
If $N\geq 2$, there is a following homotopy commutative diagram:
%%(3.11)%%
\begin{equation}\label{CD: key diagram}
\begin{CD}
\Map^*_0(D^2,S^1;\CP^N,*) @>\iota_{\C}>{\simeq}> 
\Omega^2_0\CP^N\simeq\Omega^2 S^{2N+1}
\\
@V{\hat{j}}V{\cap}V \ \ \ @A{q_1}AA
\\
\Map^*_0(D^2,S^1;\CP^N,\RP^N) @>{\iota_{\C\R}}>{\simeq}> 
\Omega^2 S^{2N+1}\times \Omega S^N
\\
@V{r_{S^1}}VV \ \ \ @V{q_2}VV
\\
\Omega_0\RP^N @>\iota_{\R}>\simeq> \ \ \ \Omega S^N
\end{CD}
\end{equation}
%%%
where
$\hat{j}:
\Map^*_0(D^2,S^1;\CP^N,*)\stackrel{\subset}{\longrightarrow}
\Map^*_0(D^2,S^1:\CP^N,\RP^N)$ 
denotes the natural inclusion, 
$r_{S^1}:\Map^*_0(D^2,S^1;\CP^N,\RP^N)\to \Omega_0\RP^N$
is the restriction map given by $r_{S^1} (f)=f\vert S^1$,
$\iota_{\C\R}$ and $\iota_{\C}$  are  homotopy equivalences given by
(\ref{eq: homotopy e.}) and (\ref{eq: restriction eq}),
and
$q_1$ and $q_2$ are the projections to the first factors and the second factor,
respectively.
%%%%
\end{lemma}
%%%%(End of Lemma 3.6)%%
\begin{proof}
%%%%
The assertion easily follows from the naturality of the homotopy equivalences
$\iota_{\C\R}$ and $\iota_{\C}$.
%%%
\end{proof}
%%%(End of proof of Lemma 3.6)%%
%%%
%%%
Next, we define the space  $\po^{d,m}_n(\R;\Ha_+)$.
%%
%%
%%
%%(Definition 3.7)%%
\begin{definition}\label{def: Hol(R;H)}
%%%%
{\rm
Let $\Ha_+=\{\alpha \in \C:\mbox{ Im }\alpha >0\}$ denote the upper half plane in 
the complex plane $\C$.
\par
%%(i)%%%
(i) For each even positive integer $d=2d_0\in \N$, 
let $\P_d(\R;\Ha_+)$ denote the space of monic polynomials
$f(z)\in \R [z]$ of degree $d$ such that it has just $d_0$ roots in $\Ha_+$.
Thus, if $f(z)\in \P_d(\R;\Ha_+)$, it is represented as the form
%%(3.12)%%
\begin{equation}\label{eq: polynomial representation}
f(z)=
%\begin{cases}
\dis \prod_{k=1}^{d_0}
(z-\alpha_k) (z-\overline{\alpha_k}) 
\quad
\mbox{for some }\alpha_k\in \Ha_+
\ (1\leq k\leq d_0).
\end{equation}
%%%%%%
%for some $\alpha_k\in \Ha_+$ $(1\leq k\leq d_0$).
%%
\par
%%(ii)%%%
(ii)
For each $d\in \N$, we define the space
$\po^{d,m}_n(\R;\Ha_+)$ as follows.
%%%
\begin{enumerate}
\item[(ii-0)] When $d\in \N$ is an even integer,
 the  space $\po^{d,m}_n(\R;\Ha_+)$ is
%%
%%
%%(3.13)%%%
\begin{equation}
%%%%%%%%
\po^{d,m}_n(\R;\Ha_+)=
\po^{d,m}_n(\R)\cap (\P_d(\R;\Ha_+))^m.
\end{equation}
%%%%%%%
\item[(ii-1)] When $d\in \N$ is an odd integer,
the space  $\po^{d,m}_n(\R;\Ha_+)$ is the subspace of
$\po^{d,m}_n(\R)$ consisting of all elements of the form
%%
%%
%%(3.14)%%
\begin{equation}
%%%%%%%%
\big((z-1)f_1(z),(z-2)f_2(z),\cdots ,(z-m)f_m(z)\big)
%%%%%%%%
\end{equation}
%%%%%%%%
where $(f_1(z),\cdots ,f_m(z))\in \po^{d-1,m}_n(\R;\Ha_+).$
\end{enumerate}
\par
(iii)
It is easy to see that there is a homeomorphism
%%(3.15)%%
\begin{equation}\label{eq: odd case}
%%%%%%%%
\po^{2d_0+1,m}_n(\R;\Ha_+)\cong \po^{2d_0,m}_n(\R;\Ha_+)
\quad
\mbox{for any }d_0\in \N,
\end{equation}
%%%% 
and that there is an inclusion
$\po^{d,m}_n(\R;\Ha_+)\subset
\po^{d,m}_n(\R)$
for any $d\in \N$.
%%%
%\end{equation}
%%%%%
We denote this natural inclusion map  by
%%(3.16)%%
\begin{equation}
%%%%
 \iota^{d,m}_{n,\Ha_+}:\po^{d,m}_n(\R;\Ha_+)
\stackrel{\subset}{\longrightarrow}
\po^{d,m}_n(\R).
\end{equation}
%%%%
%%%%
\par
%%(iv)%%
(iv) 
By making the identification  $S^2=\Ha_+\cup \infty$,   
we obtain a map
%%%(3.17)%%
\begin{align}\label{eq: mapjH}
%%%%%%%%%%
%\nonumber
i^{d,m}_{n,\Ha_+}&:\Po^{d,m}_n(\R;\Ha_+)\to \Omega^2_d\CP^{mn-1}
\simeq \Omega^2S^{2mn-1}
\quad \mbox{given by}
%%%
\\
\nonumber
i^{d,m}_{n,\Ha_+}(f)(\alpha)
&=
\begin{cases}
[F_n(f_1)(\alpha):F_n(f_2)(\alpha):\cdots :F_n(f_m)(\alpha)]
&
\mbox{if }\alpha\in \Ha_+
\\
[1:1:\cdots :1] & \mbox{if }\alpha =\infty
\end{cases}
%%%%%
\end{align}
%%%%%
for $f=(f_1(z),\cdots ,f_m(z))\in \Po^{d,m}_n(\R;\Ha_+)$
and $\alpha \in S^2=\Ha_+\cup \infty$.
%%%%%%%%%%%%%%%%%%%%
}
\end{definition}
%%%%
The next lemma is a simple but crucial observation. 

%%%%%%(Lemma 3.8)%%
\begin{lemma}\label{lmm: space poly(Ha+)}
%%%%%%
$\I$ If $d\in \N$ is an even positive integer, there is a homeomorphism
$
\po^{d,m}_n(\R;\Ha_+)\cong \po^{\lfloor d/2\rfloor,m}_n(\C).
$
\par
%$\II$ 
%If $d=2d_0+1$ is an odd positive integer, 
%there is a homeomorphism
%$\po^{d,m}_n(\R;\Ha_+)\cong \po^{d-1,m}_n(\R;\Ha_+).$
%\par
$\II$
There is a homeomorphism
$
\po^{\lfloor d/2\rfloor,m}_n(\C)\cong
\po^{d,m}_n(\R;\Ha_+)
$
for any $d\geq 2$.
\end{lemma}
%%(Proof of Lemma 3.8)%%
\begin{proof}
%%%%%%%%%%%%%%%%
(i) Let $d=2d_0$ and let $\psi :\C \stackrel{\cong}{\longrightarrow} \Ha_+$ be any fixed homeomorphism.
Then we have the  homeomorphism 
$\psi_d:\P_{d_0}(\C)
\stackrel{\cong}{\longrightarrow} \P_{d}(\R;\Ha_+)$
given by $\psi_d\big(\prod_{k=1}^{d_0}(z-\alpha_k)\big)
=\prod_{k=1}^{d_0}(z-\psi (\alpha_k))(z-\overline{\psi (\alpha_k)})$ for 
$\alpha_k\in \C$.
This naturally extends to the desired homeomorphism
%%(3.18)%%
\begin{equation} 
%%%%%%%
\po^{d_0,m}_n(\C)
\stackrel{\cong}{\longrightarrow}
\po^{d,m}_n(\R;\Ha_+) 
%%%%
\end{equation}
%%%
 given by
$\big(f_1(z),\cdots ,f_m(z)\big)
\mapsto
\big(\psi_d(f_1(z)),\cdots ,\psi_d(f_m(z))\big)$.
\par
(ii) The assertion (ii) easily follows from (\ref{eq: odd case}) and 
the assertion (i).
%%%%%%
\end{proof}
%%(End of proof of Lemma 3.8)%%
%%
%%
%%(Theorem 3.9)%%
\begin{thm}\label{thm: Theorem H}
%%%%
If $mn\geq 3$, the natural map
$$
i^{d,m}_{n,\Ha_+}:\Po^{d,m}_n(\R;\Ha_+)\to \Omega^2_d\CP^{mn-1}
\simeq \Omega^2S^{2mn-1}
$$
is a homotopy equivalence through dimension $D(\lfloor d/2\rfloor;m,n;\C)$.
%%%%%%%%%%%%%%%%%%%%%
\end{thm}
%%%%%%(Proof of Theorem 3.9)%%
\begin{proof}
%%%%%%%%%%%%%%%%%%
Since there is a homeomorphism
$\po^{d,m}_n(\R ;\Ha_+)
\cong 
\po^{\lfloor d/2\rfloor,m}_n(\C)$
(by Lemma \ref{lmm: space poly(Ha+)}),
the proof of \cite[Theorem 1.8]{KY8} 
works verbatim
by replacing $\C$ by $\Ha_+$ in the case when $d$ is even. The case of odd $d$ can be easily reduced to the even one by (\ref{eq: odd case}). 
%%%
%%
\end{proof}
%%%(End of proof of Theorem 3.9)%%
%%
%%%(Lemma 3.10)%%%
\begin{lemma}\label{lmm: inequality}
%%%
If $d\in \N$ and $d_0=\lfloor d/2\rfloor$,
the following assertions hold:
\begin{enumerate}
%%(i)%%
\item[$\I$]
$\lfloor d/n\rfloor =2\lfloor d_0/n\rfloor$ or $\lfloor d/n\rfloor =2\lfloor d_0/n\rfloor +1.$
%%(ii)%%
%%%
\item[$\II$]
%%%
$\lfloor d/(2n)\rfloor =\lfloor d_0/n\rfloor.$
%%%(iii)%%
\item[$\III$]
$\dis \lfloor \lfloor d/n\rfloor /2\rfloor =\lfloor d_0/n\rfloor$, and
%%
%%(iii)%%
%\item[$\III$]
$D(d;m,n)<D(\lfloor d/2\rfloor;m,n;\C)$.
%%%(iv)%%
\item[$\IV$]
%%%%
Define a finite subset $\mathcal{F}^{d,m}_n\subset \N^2$ by
\end{enumerate}
%%%(3.19)%%
\begin{equation}\label{eq: set F}
\mathcal{F}^{d,m}_n=\{(i,j)\in \N^2:i+2j\leq \lfloor d/n\rfloor\}.
\end{equation}
%%%%%%
Then if $(i,j)\in \mathcal{F}^{d,m}_n$, 
$1\leq i< \lfloor d/n\rfloor $
and $1\leq j\leq  \lfloor d_0/n\rfloor,$
 and
%%%%
%%(3.20)%%
\begin{equation}\label{eq:(i,j): condition}
\mathcal{F}^{d,m}_n
\subset
\{(i,j)\in \N^2:
1\leq i< \lfloor d/n\rfloor,
1\leq j\leq  \lfloor d_0/n\rfloor,
i+2j\leq \lfloor d/n\rfloor\}.
\end{equation}
%%%
\end{lemma}
%%%%
\begin{proof}
%%%
(i)
Let us write  $q=\lfloor d_0/n\rfloor$.
Then we can also write
%%%
\begin{align*}
 d&=2d_0+\epsilon_0 \ \mbox{ with }\ \epsilon_0\in \{0,1\},
\ \  
d_0=nq+\epsilon \ \mbox{ with }\ 0\leq \epsilon \leq n-1.
\\
%%% 
\mbox{Since }\ 
&
\begin{cases}
0&\leq 2\epsilon +\epsilon_0\leq 2(n-1)+1=2n-1<2n, \mbox{ and}
\\
d&=2d_0+\epsilon_0=2(nq+\epsilon)+\epsilon_0=2nq+(2\epsilon+\epsilon_0),
\end{cases}
\\
2nq&\leq d<(2q+2)n.
\quad
\mbox{So }2q\leq \lfloor d/n\rfloor <2q+2.
\end{align*}
%%%
%$2nq\leq d<(2q+2)n$ and so $2q\leq \lfloor d/n\rfloor <2q+2$.
Hence, 
$
\lfloor d/n\rfloor =2q=2\lfloor d_0/2\rfloor
\  \mbox{or}\  \lfloor d/n\rfloor =2q+1=2\lfloor d_0/2\rfloor +1.
$
%\par\noindent
Thus, we have proved  (i). 
\par
%%(ii)%%
(ii) Since
$2nq\leq d<(2q+2)n=2n(q+1)$ by (i), 
$q\leq d/(2n)<q+1$.
Thus,
we also have
$\lfloor d/(2n)\rfloor =q=\lfloor d_0/n\rfloor$
and (ii) was obtained.
\par
%%(iii)%%
(iii) By using (i), we have
$$
\lfloor d/n\rfloor /2=q=\lfloor d_0/2\rfloor\ \ \mbox{ or }\ \ 
\lfloor d/n\rfloor /2=q+\frac{1}{2}=\lfloor d_0/2\rfloor +\frac{1}{2}.
$$
Thus, 
$\lfloor \lfloor d/n\rfloor /2\rfloor =q=\lfloor d_0/2\rfloor$.
%and we have obtained (ii).
%%%
%\par
Moreover,
since 
$\lfloor d/n\rfloor\leqq 2q+1$ (by (i)),
%%%
\begin{align*}
\delta &= D(d;m,n)-D(\lfloor d/2\rfloor;m,n;\C)
\\
&=
\big\{(mn-2)(\lfloor d/n\rfloor +1)-1\big\}-
\big\{(2mn-3)(q+1)-1\big\}
\\
&\leqq
(mn-2)(2q+2)-(2mn-3)(q+1)
\leqq -(q+1)<0
\end{align*}
and the assertion (iii) follows.
\par
%%(iv)%%
(iv)
%%%%%
Suppose that $(i,j)\in \mathcal{F}^{d,m}_n$.
Since we can see that the first assertion holds and it remains show the second one.
Since $i\geq 1$, $1\leq 2j<\lfloor d/n\rfloor$.
\par
Hence, by using (iii), we have
$1\leq j\leq 
\lfloor \lfloor d/n\rfloor /2\rfloor =\lfloor d_0/n\rfloor$, and we obtain the assertion (iv).
%%%%
\end{proof}
%%%%%(End of Lemma 3.10)%%

%%
%%
%%%(End of SECTION 3)%%%%
%%%
%%%
%%%
%%%
%%%%%%%%%%%%%%%%%
%%%(SECTION 4)%%%%%%%%%%%%%%%
\section{The Vassiliev spectral sequence}
\label{section: spectral sequence}
%%%%%%%%%%%%%%%%%%%%%%%%

In this section we construct a 
Vassiliev type spectral sequence converging to 
the homology of
$\Po^{d,m}_n(\R)$
by means of a
{\it non-degenerate} simplicial resolutions  of discriminants, and compute its $E^1$-terms.
%and we give the proof of Theorem
%\ref{thm : I} by using these spectral sequences.
%%%
\par\vspace{1mm}\par
First, we  summarize  the basic facts of the theory of  non-degenerate simplicial resolutions  
(\cite{Va}, \cite{Va2}; cf.  \cite{Mo2}) and the spectral sequences associated with them.
%%%%%
%%(Definition 4.1)%%
\begin{definition}\label{def: def}
%%%%%%
{\rm
For a finite set $\textbf{\textit{v}} =\{v_1,\cdots ,v_l\}\subset \R^N$,
let $\sigma (\textbf{\textit{v}})$ denote the convex hull spanned by 
$\textbf{\textit{v}}.$
%%%
%%%
Suppose that $h:X\to Y$ is a surjective map such that
$h^{-1}(y)$ is a finite set for any $y\in Y$, and let
$i:X\to \R^N$ be an embedding.
\par\vspace{1mm}\par
(i)
Then
let  $\mathcal{X}^{\Delta}$  and $h^{\Delta}:{\mathcal{X}}^{\Delta}\to Y$ 
denote the space and the map
defined by
%%%
%%(4.1)%%%%%%%%
\begin{equation}
%%%%%%%%%%%%%%%
\mathcal{X}^{\Delta}=
\big\{(y,\textbf{\textit{u}})\in Y\times \R^N:
\textbf{\textit{u}}
\in \sigma (i(h^{-1}(y)))
\big\}\subset Y\times \R^N,
\ h^{\Delta}(y,\textbf{\textit{u}})=y.
\end{equation}
%%%%%%%%%%%%%%%%
The pair $(\mathcal{X}^{\Delta},h^{\Delta})$ is called
{\it the simplicial resolution of }$(h,i)$.
In particular, $(\mathcal{X}^{\Delta},h^{\Delta})$
is called {\it a non-degenerate simplicial resolution} if for each $y\in Y$
any $k$ points of $i(h^{-1}(y))$ span $(k-1)$-dimensional simplex of $\R^N$.
%%%%
\par
(ii)
For each $k\geq 0$, let $\mathcal{X}^{\Delta}_k\subset \mathcal{X}^{\Delta}$ be the subspace
given by 
%%(4.2)%%
\begin{equation}
%%%%%%%%
\mathcal{X}_k^{\Delta}=
\big\{(y,\textbf{\textit{u}})\in \mathcal{X}^{\Delta}:
\textbf{\textit{u}} \in\sigma (\textbf{\textit{v}}),
\textbf{\textit{v}}=\{v_1,\cdots ,v_l\}\subset i(h^{-1}(y)),
l\leq k\big\}.
\end{equation}
%%%%
%%%%
We make identification $X=\mathcal{X}^{\Delta}_1$ by identifying 
 $x\in X$ with %the pair
$(h(x),i(x))\in \mathcal{X}^{\Delta}_1$,
and we note that  there is an increasing filtration
%%%(4.3)%%
\begin{equation*}\label{equ: filtration}
%%%
\emptyset =
\mathcal{X}^{\Delta}_0\subset X=\mathcal{X}^{\Delta}_1\subset \mathcal{X}^{\Delta}_2\subset
\cdots \subset \mathcal{X}^{\Delta}_k\subset \mathcal{X}^{\Delta}_{k+1}\subset
\cdots \subset \bigcup_{k= 0}^{\infty}\mathcal{X}^{\Delta}_k=\mathcal{X}^{\Delta}.
\end{equation*}
}
\end{definition}
%%%%%
Since the map $h^{\Delta}$ is a proper map,
it extends the map
$h^{\Delta}_+:\mathcal{X}^{\Delta}_+\to Y_+$
between one-point compactifications, where
$X_+$ denotes the one-point compactification of a locally compact space
$X$.

%%%(Theorem 4.2)%%
\begin{thm}[\cite{Va}, \cite{Va2} 
(cf. \cite{KY7}, \cite{Mo2})]\label{thm: simp}
%%%%%%%%
%%%%%%%%
Let $h:X\to Y$ be a surjective map such that
$h^{-1}(y)$ is a finite set for any $y\in Y,$ 
$i:X\to \R^N$ an embedding, and let
$(\mathcal{X}^{\Delta},h^{\Delta})$ denote the simplicial resolution of $(h,i)$.
\par
\begin{enumerate}
%%(i)%%
\item[$\I$]
If $X$ and $Y$ are semi-algebraic spaces and the
two maps $h$, $i$ are semi-algebraic maps, then
$h^{\Delta}_+:\mathcal{X}^{\Delta}_+\stackrel{\simeq}{\rightarrow}Y_+$
is a homology equivalence.
%\footnote{%
%%%(FootNote 2)%%%%
%It is known that the map $h^{\Delta}_+$ is a homotopy equivalence 
%\cite[page 156]{Va2}. %(cf. \cite[Theorem in page 43]{GM}).
%However, in this paper we do not need such a stronger assertion.
%}
%%(End of FootNote 2)%%%
\item[$\II$]
If there is an embedding $j:X\to \R^M$ such that its associated simplicial resolution
$(\tilde{\mathcal{X}}^{\Delta},\tilde{h}^{\Delta})$
is non-degenerate,
the space $\tilde{\mathcal{X}}^{\Delta}$
is uniquely determined up to homeomorphism and
there is a filtration preserving homotopy equivalence
$q^{\Delta}:\tilde{\mathcal{X}}^{\Delta}\stackrel{\simeq}{\rightarrow}{\mathcal{X}}^{\Delta}$ such that $q^{\Delta}\vert X=\mbox{id}_X$.
\item[$\III$]
A non-degenerate simplicial resolution exists even if the map $h$ is not finite to one.
%%%
\end{enumerate}
%%%
\end{thm}
%%%
\begin{proof}
See the proof of \cite[Theorem 2.2]{KY10} for the details. 
\end{proof}
%%%(End of Theorem 4.2)%%
%%
%%(REMARK 4.3)%%
\begin{remark}
%%%
{\rm
It is known that the map $h^{\Delta}_+$ is a homotopy equivalence 
\cite[page 156]{Va2}.
%(cf. \cite[Theorem in page 43]{GM}).
However  this stronger assertion is not needed in this paper.}
\qed
\end{remark}
%%(End of Remark 4.3)%%
%%
\par\vspace{1mm}\par
Now
recall several basic definitions and
notations.

%%%(Definition 4.4)%%%%
\begin{definition}
%%%%%%%%%%%%%%%%%%%%%%%
{\rm
(i)
For connected space $X$, let $F(X,d)$ denote 
{\it the ordered configuration space}
of distinct $d$ points of $X$ defined by
%%%%(4.3)%%%
\begin{equation}
%%%%%%%%%%%%
F(X,d)=\{(x_1,\cdots ,x_d)\in X^d:x_i\not= x_j\mbox{ if }i\not= j\}.
\end{equation}
%%%%
\par
(ii)
Let $S_d$ denote the symmetric group of $d$-letters.
Then the group $S_d$
acts on $F(X,d)$ by the coordinate permutation and
let $C_d(X)$ denote {\it the unordered configuration space of $d$-distinct points} of $X$ defined by the orbit space
%%(4.4)%%
\begin{equation}
%%%
C_d(X)=F(X,d)/S_d.
\end{equation}
%%
%%%%
\par
(iii) For connected space $X$, let $D_j(X)$ denote 
{\it the equivariant half-smash product of }$X$
defined by
%%(4.5)%%
\begin{equation}
%%%%
D_j(X)=F(X,j)_+\wedge_{S_j}X^{\wedge j},
%%%
\end{equation}
%%%
where we set $F(X,j)_+=F(X,j)\cup \{*\}$
\ (disjoint union),
$X^{\wedge j}=X\wedge X\wedge \cdots \wedge X$ 
($j$-times) and
the $j$-th symmetric group $S_j$ acts on $X^{\wedge j}$ by the coordinate permutation.
%%%
In particular, for $X=S^1$, we set
%%(4.6)%%%
\begin{equation}\label{eq: Dd}
%%%%
D_j=D_j(S^1)
=F(\C,j)_+\wedge_{S_j}(S^1)^{\wedge j}.
%=F(\C,k)_+\wedge_{S_k}S^k.
%%%%
\end{equation}
%%%
}
\end{definition}
%%%%(End of Definition 4.4)%%%

Let $mn\geq 3$ and we shall construct the Vassiliev-type 
 spectral sequence.

%%(Definition 4.5)%%%%%%%
\begin{definition}\label{Def: 3.1}
{\rm
%%%%%(i)%%%
%Let $\F =\R$ or $\C$.
%\par
(i)
Let
$\Sigma^{d,m}_{n}$  denote \emph{the discriminant} of
$\Po^{d,m}_n(\R)$ in $\P_d(\R)^m$ 
given by
the complement
%%%%%%%%%%%%%%%%%
%%%(4.7)%%%
\begin{equation}
%%%
\Sigma^{d,m}_{n} =
\P_d(\R)^m\setminus \Po^{d,m}_n(\R).
%%%%
\end{equation}
%%%%
%%%%
\par
(ii)
For each $m$-tuple
$(f_1(z),\cdots ,f_m(z))\in \P_d(\R)^m$ 
%of $\F$-coefficients monic polynomials of the same degree $d$, 
let $F^m_n(f_1,\cdots ,f_m)$ denote the $mn$-tuple of
$\R$-coefficients monic polynomials of the same degree $d$
defined by
%%%(4.8)%%
\begin{equation}\label{eq: bfF}
%%%%
F^m_n(f_1,\cdots ,f_m)=F^m_n(f_1,\cdots ,f_m)(z)=
(F_n(f_1),\cdots ,F_n(f_m)).
%%%
\end{equation}
Since $F^m_n(f_1,\cdots ,f_m)$ is an $mn$-tuple of $\R$-coefficients polynomials,
$$
F^m_n(f_1,\cdots ,f_m)(\alpha )=0
\Leftrightarrow
F^m_n(f_1,\cdots ,f_m)(\overline{\alpha} )=0
\quad
\mbox{for }\alpha\in \Ha_+.
$$
%%%%
Thus, the space $\Sigma^{d,m}_{n}$ %and $\tilde{\Sigma}^{d,m}_n$
is
given by
%%(4.9)%%
\begin{align}\label{eq: sigma definition}
%%%%%%%%%%
\Sigma^{d,m}_{n}&=
\{f\in \P_d(\R)^m :
F^m_n(f)(x)={\bf 0}
\mbox{ for some } x\in \overline{\Ha}_+\},
\end{align}
%%%%%%%
where ${\bf 0}\in \C^{mn}$,
$F^m_n(f)=F^m_n(f_1,\cdots ,f_m)$
for $f=(f_1(z),\cdots ,f_m(z))\in \P_d(\R)^m$
and   $\overline{\Ha}_+$ denotes the space defined by
%%(4.10)%%
\begin{equation} 
\overline{\Ha}_+=
\Ha_+\cup \R=\{\alpha \in \C: \mbox{Im }\alpha \geq 0\}.
\end{equation}
%%%
%%%%
\par
(iii)
Let
$Z^{d,m}_{n}\subset \Sigma^{d,m}_{n}\times\C$ 
%and
%$\widetilde{Z}^{d,m}_n\subset \Sigma^{d,m}_n\times\R$ 
denote 
{\it the tautological normalization} of $\Sigma^{d,m}_{n}$ 
%and $\widetilde{\Sigma}^{d,m}_n$
given by
%{\small
%%()%%
\begin{equation*}
%%%%%%
%\begin{cases}
%%%%%%
Z^{d,m}_{n}=\{
((f_1(z),\cdots ,f_m(z)),x)\in \Sigma^{d,m}_{n}\times
\overline{\Ha}_+:
F^m_n(f_1,\cdots ,f_m)(x)={\bf 0}
\}.
%%%
\end{equation*}
%%%%%%%%%%%%
Projection on the first factor  gives the surjective map
%%(4.11)%%
\begin{equation}
%%%%%
\pi^{d,m}_{n}:Z^{d,m}_{n}\to \Sigma^{d,m}_{n}.
%\ \mbox{ and }\ 
%\tilde{\pi}^{d,m}_n:\widetilde{Z}^{d,m}_n\to \widetilde{\Sigma}^{d,m}_n.
%%
\end{equation}
%%%%%%%%
%%%%%%%
\par
(iv)
Let $\varphi_{\R} :\P_d(\R)^{nm}
\stackrel{\cong}{\rightarrow}
\R^{dmn}$ 
be any fixed homeomorphism.
We identify $\C =\R^2$ and define the embedding
$i_{\R}:Z^{d,m}_{n}\to \R^{dmn}\times \C = \R^{dmn+2}$
%%%%()%%%
%%%%%%%%% 
by
%%(4.12)%%
\begin{equation}\label{3.10}
%%%%%%%%%%%%%%%
i_{\R}((f_1,\cdots ,f_{m}),x)=(\varphi_{\R} (F_n(f_1),\cdots ,F_n(f_m)),
x)
%%%%%
\end{equation}
%%%%%
for $((f_1,\cdots ,f_{m}),x)\in Z^{d,m}_{n}.$
%%%
\par\vspace{1mm}\par
(v)
Let 
$(\SZ_{},\pi^{\Delta}_{}:\SZ_{} \to \Sigma^{d,m}_{n})$ 
be non-degenerate simplicial resolution of $(\pi^{d,m}_{n},i_{\R})$.
%%%
Then it is easy to see that there is a
natural increasing filtration
%%%%%
\begin{eqnarray*}
\emptyset 
&=&
\SZ_{0}
\subset \SZ_{1}\subset 
\SZ_{2}\subset \cdots \cdots\subset
\bigcup_{k= 0}^{\infty}\SZ_{k}=\SZ_{},
%\\
%\emptyset 
%&=&
%\tilde{\mathcal{X}}^d_0
%\subset \tilde{\mathcal{X}}^d_1\subset 
%\tilde{\mathcal{X}}^d_2\subset \cdots \cdots\subset
%\bigcup_{k= 0}^{\infty}\tilde{\mathcal{X}}^d_k=\tilde{\mathcal{X}}^d,
%%%%
\end{eqnarray*}
%%%%%
such that
%%%%%()%%
%%%%
$\mathcal{X}^d_{k}=\mathcal{X}^d_{}$ 
%and
%$\tilde{\mathcal{X}}^d=\tilde{\mathcal{X}}_k^d$ 
if $k>\lfloor d/n\rfloor .$
\qed
%%%%%%%
%%%%%
}
%%%%%%%%%%
\end{definition}
%%%%%%(End of Definition 4.2)%%%%%
%%%
%%%
%%%
%%%
%%%
%%%%%%%%%%%%%%%%%%%%%%%%%%%%%%%%%%
By Theorem \ref{thm: simp},
the map
%%%%%%%%%
$\pi_{+}^{\Delta}:\SZ_{+}\stackrel{\simeq}{\rightarrow}{\Sigma^{d,m}_{n +}}$
%%%%%
is a homology equivalence.
%%
%\par
%%%
The filtration on ${\mathcal{X}^{d}_{+}}$ gives rise to a spectral sequence
%%%%%%%%%%%
%$$
%%(4.13)%%
\begin{equation}
%%%%%
\big\{E_{t;d}^{k,s},
d_t:E_{t;d}^{k,s}\to E_{t;d}^{k+t,s+1-t}
\big\}
\Rightarrow
H^{k+s}_c(\Sigma_{n}^{d,m};\Z),
\end{equation}
%$$
where 
$H_c^k(X;\Z)$ denotes the cohomology group with compact supports 
of a locally compact space $X$ given by 
$
H_c^k(X;\Z)= \tilde{H}^k(X_+;\Z)
$ and
%%%(4.14)%%
\begin{equation}
%%%%% 
E_{1;d}^{k,s}=
\tilde{H}^{k+s}({\mathcal{X}_{k}^{d}}_+/{\SZ_{k-1}}_+;\Z) 
=H^{k+s}_c(\SZ_{k}\setminus\SZ_{k-1};\Z)
%%%
\end{equation}
%%%%
(since ${\mathcal{X}_{k}^{d}}_+/{\SZ_{k-1 +}}
\cong (\SZ_{k}\setminus \SZ_{k-1})_+$). 
%%
%%%%%%%%%%%%%
%\par\vspace{2mm}\par
%%%%%
\par
Since there is a homeomorphism
$\P_d(\R)^m\cong \R^{dm}$,
by Alexander duality  there is a natural
isomorphism
%%%(4.15)%%%
\begin{equation}\label{Al}
%%%%%%%%%%%
\tilde{H}_k(\Po^{d,m}_n(\R);\Z)
\cong
H_c^{dm-k-1}(\Sigma_{n}^{d,m};\Z)
\quad
\mbox{for any }k.
\end{equation}
%%%
By
reindexing we obtain 
{\it the Vassiliev-type
spectral sequence}
%%
%%%(4.16)%%%
\begin{eqnarray}\label{SS}
%%%%%%%%%%%%%%%%%%%
&&
\big\{E^{t;d}_{k,s}, \ d^{t}_{}:E^{t;d}_{k,s}\to E^{t;d}_{k+t,s+t-1}
\big\}
\Rightarrow \tilde{H}_{s-k}(\Po^{d,m}_n(\R);\Z),
\end{eqnarray}
%%%%%%%
%if $s-k\leq 2nd-2$,
where we set
%%(4.17)%%
\begin{equation}
%%%%%
E^{1;d}_{k,s}=
H^{dm+k-s-1}_c(\SZ_{k}\setminus\SZ_{k-1};\Z).
\end{equation}
%%%%%%
%%
%%(Definition 4.6)%%
\begin{definition}\label{rmk: roots}
%%%%%%%%
{\rm
Let $\Ha_+=\{\alpha \in \C: \mbox{Im }\alpha >0\}$ denote the upper half space, and
let $(d,k)\in \N^2$ be a pair of positive integers
such that $1\leq k\leq \lfloor d/n\rfloor$. 
\par\vspace{1mm}\par
(i)
Let
$\Sigma^{d,m}_{n}(k)\subset \Sigma^{d,m}_{n}$
denote the subspace
consisting of all $m$-tuples
$(f_1(z),\cdots ,f_m(z))\in \Sigma^{d,m}_{n}$
such that
the polynomials
$\{f_t(z)\}_{t=1}^m$ have exactly
$k$ common roots of multiplicity $\geq n$ in $\overline{\Ha}_+$.
\par (ii)
Let $(f_1(z),\cdots ,f_m(z))\in \Sigma^{d,m}_{n}(k)$.
Since each of these polynomials has real coefficients, these common roots of multiplicity $\geq n$ can be uniquely represented as a set 
%(4.18)%%%%%%%%%
\begin{equation}\label{eq: roots}
%%%
\tilde{c}=
\{x_1,\cdots ,x_i; \alpha_1,\overline{\alpha}_1,\cdots .\alpha_j,
\overline{\alpha}_j\}
\quad (i+j=k,0\leq i\leq \lfloor k/2\rfloor)
\end{equation}
%%%
of 
$(k+j)$-complex numbers
which satisfy the following two conditions:
%%%%
\begin{enumerate}
%%(4.18.1)%%
\item[(\ref{eq: roots}.1)]
%%%%%%%%%%
$x_s\in \R$ for each $1\leq s\leq i$ and $x_i\not= x_l$ if $i\not=l$.
%%%(4.18.2)%%%
\item[(\ref{eq: roots}.2)]
%%%%%%%%%
$\alpha_t\in \Ha_+$
for each $1\leq t\leq j$ and
$\alpha_t\not= \alpha_l$ if $t\not= l$.
%where
%$\Ha_+$ denotes the upper half space given by
%$\Ha_+ =\{\alpha \in \C: \mbox{Im }\alpha >0\}.$
%%
\end{enumerate}
%%%%%%
We define a subspace $\Sigma^{d,m}_{n}(i,j)\subset \Sigma^{d,m}_{n}(k)$
as the space of all $m$-tuples $(f_1(z),\cdots ,f_m(z))\in \Sigma^{d,m}_{n}(k)$ whose common roots of multiplicity $\geq n$
satisfy the  two above conditions  (\ref{eq: roots}.1) and (\ref{eq: roots}.2).
%%%
\par\vspace{1mm}\par
(iii)
%%%
Note that $(\pi_{}^{\Delta})^{-1}(\Sigma^{d,m}_{n}(k))
=\SZ_{k}\setminus \SZ_{k-1}$.
For each pair $(i,j)\in (\Z_{\geq 0})^2$ with $i+j=k$,
define a subspace
$\SZ_{k}(i,j) \subset\SZ_{k}\setminus \SZ_{k-1}$
by
%%%(4.19)%%%
\begin{equation}
%%%%%% 
\SZ_{k}(i,j)=(\pi^{\Delta}_{})^{-1}(\Sigma^{d,m}_{n}(i,j)).
\end{equation}
%%%%%%%
%%%
%%%
It is easy to see that
%$
%\dis \SZ_{k}\setminus \SZ_{k-1}=\coprod_{i+j=k}\SZ_{k}(i,j)
%$ \ (disjoint union),
%where
%%%(4.20)%%
\begin{align}\label{eq: path-components}
%%%%%%%%%
\dis \SZ_{k}\setminus \SZ_{k-1}&=\coprod_{i+j=k}\SZ_{k}(i,j)
\quad \mbox{(disjoint union)},
%%%%%
\end{align}
%%%%
where
$\{\SZ_{k}(i,j):i\geq 0,j\geq 0 ,
i+j=k\}$
are the path components of $\SZ_{k}\setminus \SZ_{k-1}$.
%%%
%%%%
}
%%%
%%%%%%%
\end{definition}
%%%(End of Definition 4.6)%%%%%
%%%%%%%%
%
%%
%%
%%
%%
%%
%%
%%
%%%%(Lemma 4.7)%%%%
\begin{lemma}\label{lemma: vector bundle}
%%%%%%%%%%%%%%%%%%
%%%
If $mn\geq 3$,
$1\leq k\leq \lfloor d/n\rfloor$ and
$(i,j)\in (\Z_{\geq 0})^2$ with $i+j=k$, the space
$\SZ_{k}(i,j)$
is homeomorphic to the total space of a real affine
bundle  $\xi^d_k(i,j)$ over $C_i(\R)\times C_j(\Ha_+)$ with real rank 
$l_{d,k}(i,j)=m(d-nk-nj)+k-1$.
%%%
%%%%%%%%%%%%%%%%%%
\end{lemma}
%%%%%%%%(Proof of Lemma 4.7)%%%
\begin{proof}
%%%%%%%%%%%
%Since the proofs of (i) and (ii) are similar, we only give the proof of (i).
%Moreover,
The argument is exactly analogous to the one used in the proof of
 \cite[Lemma 3.2]{KY10} or     
\cite[Lemma 3.3]{KY8}.
%%%
%%%
Namely, an element of $\SZ_{k}(i,j)$ is represented by
the $(m+1)$-tuple 
$(f_1(z),\cdots ,f_m(z),\textbf{\textit{u}})$, where 
$(f_1(z),\cdots ,f_m(z))$ is an $m$-tuple of monic polynomials of the same
degree $d$
in $\Sigma^{d,m}_{n}$ and $\textbf{\textit{u}}$ is an element of the interior of
the span of the images of $k$ distinct points 
$$
c=(\{x_1,\cdots, x_i\},\{\alpha_1,
%\overline{\alpha}_1,
\cdots ,
\alpha_j\})\in
%\overline{\alpha}_j\})\in
C_i(\R)\times C_{j}(\Ha_+)
%\ \ 
%(\{\alpha_s\}_{s=1}^j\in C_j(\Ha_+))
$$ 
under a suitable embedding $i_k.$\footnote{%
%%(Footnote 6)%%%%%%
This means that the embedding $i_k$ satisfies the condition 
\cite[(2.3)]{KY10}.
}
%%%%(End of Footnote 6)%%%%
%%%%%
Note that the following $(k+j)$-points
$$
\tilde{c}=
\{x_1,\cdots ,x_i,\alpha_1,
\overline{\alpha}_1,
\alpha_2,
\overline{\alpha}_2,
\cdots ,\alpha_j,\overline{\alpha}_j\}
$$
are  common roots of 
$\{f_s(z)\}_{s=1}^m$ of multiplicity $n$.
%%%%
%%%%%%%%%%%%%%%%%%%
\par
By the definition of the non-degenerate simplicial resolution
and
\cite[Lemma 2.5]{KY10},
the $k$ distinct points $c$ 
are uniquely determined by $\textbf{\textit{u}}$.
%%%%
Thus, there is the projection map
%%%%%%%%%%%%%
%%%(4.21)%%%%
\begin{equation}\label{eq: projection}
%%%%%%%%% 
\pi_{k;i,j}^d :{\cal X}^{d}_{k}(i,j)
\to C_i(\R)\times C_{j}(\Ha_+)
\end{equation}
%%%%%%%%%%%%
defined by
$((f_1,\cdots ,f_m),\textbf{\textit{u}}) \mapsto 
(\{x_1,\cdots ,x_k\}, \{\alpha_1,\cdots ,\alpha_j\})$. 
%%%
\par
%%%
%%%%%(Fiber of pi_k)%%%%%%
Now suppose that $1\leq k\leq \lfloor d/n\rfloor$,
$i,j\geq 0$ with $i+j=k$, and
let
$c=(\{x_1,\cdots, x_i\},\{\alpha_1,\cdots ,
\alpha_j \})\in
C_i(\R)\times C_{j}(\Ha_+)$
 be any fixed element. Consider the fibre  $(\pi^d_{k;i,j})^{-1}(c)$.
%%%
It is easy to see that the condition
that a polynomial $f_s(z)\in\P_d(\R)$
is divisible by
$\prod_{u=1}^i(z-x_u)^n$,
is equivalent to the following the condition:
%%%  
%%%(4.22)%%%
\begin{equation}\label{equ: equation}
%%%%%%%%%%%
f^{(t)}_s(x_u)=0
\quad
\mbox{for all }0\leq t<n,\ 1\leq u\leq i.
\end{equation}
%%%%%%%%%%%
%%
In general, for each $0\leq t< n$ and $1\leq u\leq i$,
the condition $f^{(t)}_s(x_u)=0$ 
gives
one  linear condition on the coefficients of $f_s(z)$,
and this determines an affine hyperplane in $\P_d(\R)\cong \R^d$. 
%%%
For example, if $f_s(z)=z^d+\sum_{l=1}^da_{l}z^{d-l}$,
then
$f_s(x_u)=0$ for all $1\leq u\leq i$
if and only if
%%%%%%
%%()(matrix equation)%%
\begin{equation*}\label{equ: matrix equation}
%%%%%%%%%
\begin{bmatrix}
1 & x_1 & x_1^2 & x_1^3 & \cdots & x_1^{d-1}
\\
1 & x_2 & x_2^2 & x_2^3 & \cdots & x_2^{d-1}
\\
\vdots & \ddots & \ddots & \ddots & \ddots & \vdots
%\\
%1 & x_{k-1} & x_{k-1}^2 & \cdots & x_{k-1}^{d-1}
\\
1 & x_i & x_i^2 & x_i^3 & \cdots & x_i^{d-1}
\end{bmatrix}
%%%%
\cdot
\begin{bmatrix}
a_{d}\\ a_{d-1} \\ \vdots %\\ a_{2,t} 
\\ a_{1}
\end{bmatrix}
=
-
\begin{bmatrix}
x_1^d\\ x_2^d \\ \vdots %\\ s_{t,k-1}-x_{k-1}^d 
\\ x_i^d
\end{bmatrix}.
\end{equation*}
%%%%
%%%%
%%%%
Similarly, $f^{\p}_s(x_u)=0$ for all $1\leq u\leq i$
if and only if
%%()(matrix equation)%%
\begin{equation*}\label{equ: matrix equation2}
%%%%%%%%%
\begin{bmatrix}
0 &1 & 2x_1 & 3x_1^2 & \cdots & (d-1)x_1^{d-2}
\\
0 & 1 & 2x_2 & 3x_2^2 & \cdots & (d-1)x_2^{d-2}
\\
\vdots & \vdots & \ddots & \ddots & \ddots & \vdots
%\\
%1 & x_{k-1} & x_{k-1}^2 & \cdots & x_{k-1}^{d-1}
\\
0 &1 & 2x_i & 3x_i^2 & \cdots & (d-1)x_i^{d-2}
\end{bmatrix}
%%%%
\cdot
\begin{bmatrix}
a_d \\
 a_{d-1} \\ \vdots %\\ a_{2,t} 
\\ a_{1}
\end{bmatrix}
=
-
\begin{bmatrix}
dx_1^{d-1}\\ dx_2^{d-1} \\ \vdots %\\ s_{t,k-1}-x_{k-1}^d 
\\ dx_i^{d-1}
\end{bmatrix}
\end{equation*}
%%%%
and
$f^{\p\p}_s(x_u)=0$ for all $1\leq u\leq i$
if and only if
%%()(matrix equation)%%
\begin{equation*}\label{equ: matrix equation2}
%%%%%%%%%
\begin{bmatrix}
0 & 0 & 2 & 6x_1 & \cdots & (d-1)(d-2)x_1^{d-3}
\\
0 & 0 & 2 & 6x_2 & \cdots & (d-1)(d-2)x_2^{d-3}
\\
\vdots & \vdots & \ddots & \ddots & \ddots & \vdots
%\\
%1 & x_{k-1} & x_{k-1}^2 & \cdots & x_{k-1}^{d-1}
\\
0 & 0 & 2 & 6x_i & \cdots & (d-1)(d-2)x_i^{d-3}
\end{bmatrix}
%%%%
\cdot
\begin{bmatrix}
a_d \\
 a_{d-1} \\ \vdots %\\ a_{2,t} 
\\ a_{1}
\end{bmatrix}
=
-
%%%%%
\begin{bmatrix}
d(d-1)x_1^{d-2}\\ 
d(d-1)x_2^{d-1} 
\\ \vdots %\\ s_{t,k-1}-x_{k-1}^d 
\\ d(d-1)x_i^{d-2}
\end{bmatrix}
\end{equation*}
%%%%
%%%%
%%%%
and so on.
Since $1\leq i\leq k\leq \lfloor d/n\rfloor$ and
 $\{x_u\}_{u=1}^i\in C_i(\R)$, it follows from
the properties of Vandermonde matrices  and Gaussian elimination as in the proof of
\cite[Lemma 3.2]{KY10}
that the the condition  
(\ref{equ: equation}) 
is equivalent to exactly $ni$ affinely independent conditions on the coefficients of 
$f_s(z)$.
Hence,
we see that 
the space of $m$-tuples $(f_1(z),\cdots ,f_m(z))\in\P_d(\R)^m$ 
%of monic polynomials 
which satisfy
the condition (\ref{equ: equation}) for each $1\leq s\leq m$
is the intersection of $mni$ real affine hyperplanes in general position, and
has real codimension $mni$ in $\P_d(\R)^m$.
%%%%%
\par
Arguing in exactly the same manner, we see that the condition that each polynomial
$f_s(z)$  is divisible by
$\prod_{t=1}^j(z-\alpha_t)^n$ 
%(resp.  by $\prod_{t=1}^j(z-\overline{\alpha_t})^n$)
for each $1\leq s\leq m$,
gives a subspace of  {\it complex} codimension $mnj$ in $\P_d(\C)^m$.
Since $f_s(z)$ is a real coefficient polynomial,
%we easily see 
%note that
$z=\alpha_t$ is a root of $f_s(z)$ of multiplicity $n$
if and only if  the same holds for 
$z=\overline{\alpha}_t$. 
%and we obtain a subspace of real  codimension $2mnj$ in $\P_d(\R)^m$.
%%
%Thus, since $c=(\{x_1,\cdots ,x_i\},\{\alpha_1,\cdots ,\alpha_j\})\in C_i(\R)\times C_j(\Ha_+),$
Thus, the fibre $(\pi_{k;i,j}^d)^{-1}(c)$ is homeomorphic  to the product of an open $(k-1)$-simplex 
%spaned by $i_k(c)$
 with the real affine space of dimension
 $dm -(mni +2mnj)=dm-mn(k+j)
 =m(d-nk-nj)$.
% where $i_k(c)=\{i_k(x_u),i_k(\alpha_t):1\leq u\leq i,1\leq t\leq j\}.$
%%
We can check that local triviality holds.
Hence, we see that
$\SZ_{k}(i,j)$ is a real affine bundle over $C_i(\R)\times C_{j}(\Ha_+)$ of rank $l_{d,k}(i,j)
=m(d-nk-nj)+k-1$.
%%%%%%%%%
\end{proof}
%%(End of proof of Lemma 4.7)%%%
%%

%%
%%%%%%(Lemma 4.8)%%
\begin{lemma}\label{lemma: E1}
%%%%%%
%%%%%%%%%%%%
If $1\leq k\leq  \lfloor d/n\rfloor$
and $mn\geq 3$,
there is a natural isomorphism
$$
E^{1;d}_{k,s}\cong
\Big(\bigoplus_{j=1}^{k}
\tilde{H}_{s-(mn-1)k}(\Sigma^{(mn-2)j}D_j;\Z)\Big)
\oplus \tilde{H}_{s-(mn-1)k}(S^0;\Z).
$$ 
%%%%%(End of Lemma 4.8)%%%
\end{lemma}
%%%%
\begin{proof}
%%%(Proof of Lemma 4.8)%%
First, consider the case $j=0$.
Then $1\leq k=i\leq \lfloor d/n\rfloor$.
Since $C_k(\R)\cong \R^k$, the affine bundle $\xi^d_k(k,0)$ is trivial.
Hence,
there is a homeomorphism
$\SZ_{k}(k,0)_+\cong (\R^k\times \R^{l_{d,k}(k,0)})_+=
S^{dm-(mn-1)k+k-1}.$
Hence, there is an isomorphism
\begin{align*}
H^{dm+k-s-1}_c(\SZ_{k}(k,0);\Z)
&\cong
\tilde{H}^{dm+k-s-1}(S^{dm-(mn-1)k+k-1};\Z)
\\
&
\cong
\tilde{H}_{s-(mn-1)k}(S^0;\Z).
\nonumber
\end{align*}
%%%%%%%%%%%%
\par
Next consider the case  $j\geq 1$.
$$
\mbox{Since }\ 
\begin{cases}
dm+k-s-1-l_{d,k}(i,j)&=mn(k+j)-s,
%\quad
%\mbox{(by Lemma \ref{lemma: vector bundle}),}
\\
2j-\{mn(k+j)-s-i\}&=
s-(mn-1)k-(mn-1)j,
\end{cases}
$$
by the Thom isomorphism
and Poincare duality,  there are isomorphisms
%%%%
%%%%
%%%%
\begin{align*}
%%%%%%%
&H^{dm+k-s-1}_c(\SZ_{k}(i,j);\Z)
\cong
H_c^{dm+k-s-1-l_{d,k}(i,j)}(C_i(\R)\times C_j(\Ha_+);\pm \Z)
\\
&
=
H_c^{mn(k+j)-s}(C_i(\R)\times C_j(\Ha_+);\pm \Z)
%\\
%&
\cong
H_c^{mn(k+j)-s-i}(C_j(\Ha_+);\pm \Z)
\\
&\cong
\tilde{H}_{s-(mn-1)k-(mn-1)j}(C_j(\Ha_+);\pm \Z)
%\\
%& 
\cong
\tilde{H}_{s-(mn-1)k-(mn-1)j}(C_j(\C);\pm \Z).
\end{align*}
%%%
%where $\pm \Z$ denotes the twisted coefficient system on $C_j(\C)$
%comes from the Thom isomorphism 
%and it
%is induced by the sign representation of the symmetric group
%\cite[page 114 and 254]{Va}.
%%%%%%%%%%%%%%%%%
\par
Hence,
by  (\ref{eq: path-components})
we have the following isomorphisms
%%()%%
\begin{align*}\label{eq: E1-decompsition}
E^{1;d}_{k,s}
&=
H_c^{dm+k-s-1}(\SZ_{k}\setminus \SZ_{k-1};\Z)
%\\
%&
= \bigoplus_{i+j=k}
H_c^{dm+k-s-1}(\SZ_{k}(i,j);\Z)
\\
&\cong
\Big(
\bigoplus_{j=1}^{k}
%\lfloor k/2\rfloor}
\tilde{H}_{s-(mn-1)k-(mn-1)j}(C_j(\C);\pm \Z)\Big)
\oplus
\tilde{H}_{s-(mn-1)k}(S^0;\Z).
%\\
%&=
%A_{k,s}\oplus
%\tilde{H}_{s-(mn-1)k}(S^0;\Z)
%%%
\end{align*}
%%%
It follows from \cite{CMM}  that 
$D_j=D_j(S^1)$ is the Thom space of
the following $j$-dimensional vector bundle over $C_j(\C)$,
%%(4.23)%%
\begin{equation}
%%%
F(\C,j)\times_{S_j}\R^j \to F(\C,j)\times_{S_j}\{*\}=F(\C,j)/S_j=C_j(\C).
\end{equation}
%%%%%%%%%
%%%
Thus, by the Thom isomorphism theorem,
there is an isomorphism
%%(4.24)%%
\begin{equation}\label{eq: homology Cj(C)}
\tilde{H}_{*+j}(D_j;\Z)\cong \tilde{H}_{*}(C_j(\C);\pm \Z).
\end{equation}
%%%
Hence, we have the isomorphisms
\begin{align*}
%%%
E^{1;d}_{k,s}
&\cong
\Big(\bigoplus_{j=1}^{k}
\tilde{H}_{s-(mn-1)k-(mn-1)j}(C_j(\C);\pm \Z)\Big)
\oplus \tilde{H}_{s-(mn-1)k}(S^0;\Z)
\\
&\cong
\Big(
\bigoplus_{j=1}^{k}
\tilde{H}_{s-(mn-1)k-(mn-2)j}(D_j;\Z)\Big)
\oplus
\tilde{H}_{s-(mn-1)k}(S^0;\Z)
\\
&\cong
\Big(
\bigoplus_{j=1}^{k}
\tilde{H}_{s-(mn-1)k}(\Sigma^{(mn-2)j}D_j;\Z)\Big)
\oplus
\tilde{H}_{s-(mn-1)k}(S^0;\Z),
%%%%%%%%
\end{align*}
and this completes the proof.
\end{proof}
%%%(End of proof of Lemma 4.8)%%%%%%
%%
%%
%%
%%
%%
%%
%%

%%%(Corollary 4.9)%%%
\begin{crl}\label{crl: Er}
%%%%%
%$\I$
If $mn\geq 3$, 
there is a natural isomorphism
$$
E^{1;d}_{k,s}=
\begin{cases}
\dis
A_{k,s}
\oplus \tilde{H}_{s-(mn-1)k}(S^0;\Z)
&\mbox{if }1\leq k\leq \lfloor d/n\rfloor
\mbox{ and }s\geq (mn-1)k,
\\
\quad
0 & \mbox{otherwise,}
\end{cases}
$$
where $A_{k,s}$ denotes the abelian group defined by
%%(4.25)%%
\begin{equation}\label{eq: abelian group A}
%%%%
A_{k,s}=
\bigoplus_{j=1}^{k}
\tilde{H}_{s-(mn-1)k}(\Sigma^{(mn-2)j}D_j;\Z).
\end{equation}
%%%%
%%%%%%%%%%
\end{crl}
%%%%%
\begin{proof}
%%%%
The  assertion easily follows from Lemma \ref{lemma: E1}.
%%%
\end{proof}
%%%%
%%%%%%(End of Proof of Corollary 4.9)%%%
%%
%%
%%
The following results will be needed in the proof of 
Theorem \ref{thm: KY13; stable homotopy type}.
%%

%%(Lemma 4.10)%%
\begin{lemma}\label{lmm: stable type of Q}
If $mn\geq 3$, there is a stable homotopy equivalence
%%(4.26)%%
\begin{equation}\label{eq: stable Q}
%%%
\Q^{d,m}_n(\R)\simeq_s
\bigvee_{k=1}^{\lfloor d/n\rfloor}
S^{k(mn-2)}.
%%%%%
\end{equation}
\end{lemma}
%%%
\begin{proof}
%%%
First, consider the case $mn\geq 4$.
By (\ref{eq: James}), there is a homotopy equivalence
$\Q^{d,m}_n(\R)\simeq J_{\lfloor d/n\rfloor}(S^{mn-2})$.
From \cite{Ja} it follows that there is a homotopy equivalence
$
\dis\Sigma \Q^{d,m}_n(\R)\simeq
\bigvee_{k=1}^{\lfloor d/n\rfloor}
S^{k(mn-2)+1},$ 
and we obtain
the stable homotopy equivalence
(\ref{eq: stable Q}).
\par
Next, consider the case $mn=3$, i.e. 
$(m,n)=(3,1)$ or $(1,3)$.
\par
If $(m,n)=(3,1)$, the assertion easily follows from
\cite[Theorems A and B]{Y1}.
If $(m,n)=(1,3)$, then by using Theorem \ref{thm: KY10}, \cite{Va0}
 and
\cite[Theorem 3 (page 88)]{Va}, we obtain the following result:
\begin{enumerate}
\item[$(**)$]
The natural map
$i^{d,1}_3:\Q^{d,1}_3(\R)\to \Omega S^2$
is a homology equivalence through dimension 
$D(d;1,3)=\lfloor d/3\rfloor$, and
$H_i(\Q^{d,1}_3(\R);\Z)=0$ for any $i>\lfloor d/3\rfloor$.
\end{enumerate}
%%%%%
Let us consider the stable map given by
the composite of stable maps
$$
\Q^{d,1}_3(\R)\stackrel{i^{d,1}_3}{\longrightarrow}
\Omega S^2\stackrel{\simeq_s}{\longrightarrow}
\bigvee_{i=1}^{\infty}S^i
\stackrel{q}{\longrightarrow}
\bigvee_{i=1}^{\lfloor d/3\rfloor}S^i,
$$
where $q$ is the pinching map.
It is easy to see that this map induces an isomorphism on the homology groups
$H_*(\ ;\Z)$ and hence gives a stable homotopy equivalence
$\dis \Q^{d,1}_3(\R)\simeq_s\bigvee_{i=1}^{\lfloor d/3\rfloor}S^i$.
\end{proof}

%%(End of proof of Lemma 4.10)%%
%%
%%
%%
%%(Definition 4.11)%%
\begin{definition}\label{def: stable space}
%%%%%%%%%%%%%
{\rm
Let $\P^{d,m}_n$ denote the space given by
%%%%%%%%
%%(4.27)%%
\begin{align}\label{eq: stable type of Poly}
%%%%%%%
\P^{d,m}_n&=
\po^{\lfloor d/2\rfloor, m}_n(\C) \vee B^{d,m}_n\vee \Q^{d,m}_n(\R),
%%%%%%%
\end{align}
%%%%
where 
%the  space $B^{d,m}_n$ is given by 
%%%%(4.28)%%
$B^{d,m}_n
=
\bigvee_{i,j\geq 1,i+2j\leq \lfloor d/n\rfloor}
\Sigma^{(mn-2)(i+2j)}D_j$
%\begin{align}
%%%%%%%%
%B^{d,m}_n
%&=
%\bigvee_{i,j\geq 1,i+2j\leq \lfloor d/n\rfloor}
%\Sigma^{(mn-2)(i+2j)}D_j
%%%%
%\end{align}
%%%%%%%%%%%
as in (\ref{eq: the space B}).
%%%%%%%%%%%%%%
}
%%%%%%
\end{definition}
%%%%
%%

%%(Lemma 4.12)%%
\begin{lemma}\label{lmm: total E}
%%%%%%%%%%%%%%%%
If $mn\geq 3$, there is an isomorphism
$$
E^1_s:=
\bigoplus_{k\in \Z}E^{1;d}_{k,k+s}
\cong
H_s(\P^{d,m}_n;\Z)
\quad
\mbox{for any }s\geq 1.
$$
\end{lemma}
%%%%%
\begin{proof}
%%%%%%%
Since, in general,  the total number of roots of multiplicity $n$ is
$k+j=i+2j$,
we only need to consider the case:
%%(4.28)%%
\begin{equation}\label{eq: condition (i,j)}
%%%
i\geq 0,\ j\geq 0,\ 
i+2j\leq \lfloor d/n\rfloor.
\end{equation}
%%%
%%
Suppose that $s\geq 1$.
Since $k=i+j$, by (\ref{eq: set F}), we have
\begin{align*}
\bigoplus_{k\in \Z}A_{k,k+s}
&=
\bigoplus_{i\geq 0,j\geq 1,i+2j\leq \lfloor d/n\rfloor}
\tilde{H}_{s-(mn-2)(i+j)}(\Sigma^{(mn-2)j}D_j;\Z)
%%%
\\
&\cong
\bigoplus_{(i,j)\in \mathcal{G}^{d,m}_n}
\tilde{H}_{s}(\Sigma^{(mn-2)(i+2j)}D_j;\Z),
\end{align*}
%%%%
where we set  
$\mathcal{G}^{d,m}_n=\{(i,j):i\geq 0, j\geq 1,
i+2j\leq \lfloor d/n\rfloor\}.$
\par
Thus,
by  Corollary \ref{crl: Er} and
 Lemma \ref{lmm: stable type of Q}, 
%%%
%%%%%
\begin{align*}
%%%%
E^1_s
&=
\bigoplus_{k\in \Z}E^{1;d}_{k,k+s}
=\bigoplus_{k=1}^{\lfloor d/n\rfloor}\big(A_{k,s}\oplus
\tilde{H}_{s-(mn-2)k}(S^0;\Z)\big)
%%%
\\
&\cong
\big(\bigoplus_{(i,j)\in \mathcal{G}^{d,m}_n}
\tilde{H}_{s}(\Sigma^{(mn-2)(i+2j)}D_j;\Z)\big)
\oplus
\big(\bigoplus_{k=1}^{\lfloor d/n\rfloor}
\tilde{H}_{s-(mn-2)k}(S^0;\Z)\big)
%%%%%
\\
%%%%%%%
&\cong
\big(\bigoplus_{(i,j)\in \mathcal{G}^{d,m}_n}
\tilde{H}_{s}(\Sigma^{(mn-2)(i+2j)}D_j;\Z)\big)
\oplus
\tilde{H}_s(\bigvee_{k=1}^{\lfloor d/n\rfloor}S^{(mn-2)k};\Z)
\\
&\cong
\Big(\bigoplus_{(i,j)\in \mathcal{G}^{d,m}_n}
\tilde{H}_{s}(\Sigma^{(mn-2)(i+2j)}D_j;\Z)\Big)
\oplus
\tilde{H}_s(\Q^{d,m}_n(\R);\Z)
%%%%%%%%%
\\
&\cong
\tilde{H}_s
\Big(\big(\bigvee_{(i,j)\in \mathcal{G}^{d,m}_n}
\Sigma^{(mn-2)(i+2j)}D_j\big)\vee \Q^{d,m}_n(\R);\Z\Big).
\end{align*}
%%%%%%
%%%%
Let $d_0=\lfloor d/2\rfloor$ and let
%%%
$\dis
C^{d,m}_n=
\bigvee_{(i,j)\in \mathcal{G}^{d,m}_n}
\Sigma^{(mn-2)(i+2j)}D_j.
$
%%
%\end{equation*}
%%
\par\vspace{1mm}\par
%%%
\noindent
Then by Lemma \ref{lmm: inequality} and (\ref{eq: the space poly}),
%%%
\begin{align*}
&C^{d,m}_n=\bigvee_{i\geq 0,j\geq 1,i+2j\leq\lfloor d/n\rfloor}
\Sigma^{(mn-2)(i+2j)}D_j
\\
%%%%%
&=
\Big(\bigvee_{i,j\geq 1,i+2j\leq\lfloor d/n\rfloor}
\Sigma^{(mn-2)(i+2j)}D_j\Big)
\vee
\Big(\bigvee_{2j\leq \lfloor d/n\rfloor ,\ j\geq 1 }
\Sigma^{2(mn-2)j}D_j\Big)
\\
&= B^{d,m}_n\vee \Big(\bigvee_{j=1}^{\lfloor d_0/n\rfloor}
\Sigma^{2(mn-2)j}D_j\Big)
\simeq_s B^{d,m}_n\vee \po^{d_0,m}_n(\C)
\\
&
=
B^{d,m}_n\vee \po^{\lfloor d/2\rfloor,m}_n(\C).
\end{align*}
Hence, for $s\geq 1$, there are isomorphisms
\begin{align*}
E^1_s&=
\bigoplus_{k\in \Z}E^{1;d}_{k,k+s}
\cong H_s(C^{d,m}_n\vee \Q^{d,m}_n(\R);\Z)
\\
&=H_s(B^{d,m}_n\vee \po^{\lfloor d/2\rfloor,m}_n(\C)\vee \Q^{d,m}_n(\R);\Z)
=H_s(\P^{d,m}_n;\Z)
\end{align*}
and the assertion follows.
%%%
\end{proof}
%%%(End of proof of Lemma 4.13)%%

%
%%%(SECTION 5)%%%
\section{Loop products and stabilization maps}
\label{section: loop products}
%%%%%%%%%%%%%%%%

In this section we construct  loop products and stabilization maps
on the spaces $\Po^{d,m}_n(\K)$  and $\Q^{d,m}_n(\R)$,
and use them to prove Theorem
\ref{thm : I}.
%We also give an alternative proof of Proposition \ref{prp : I} by means of a  spectral sequence argument. 

%%(Definition 5.1)%%
\begin{definition}\label{def: 4.1}
%%%%%%%%%%%%%%%%%%
{\rm
Let $\varphi :\C \stackrel{\cong}{\longrightarrow}(0,\infty )\times \R$
and
$\psi :\C \stackrel{\cong}{\longrightarrow}(-\infty,0)\times \R$
be any fixed homeomorphisms satisfying the following two conditions:
\begin{enumerate}
%%(5.1.1)%%%
\item[(\ref{def: 4.1}.1)]
$\begin{cases}
\varphi (\R)=(0,\infty)\times \{0\}, & \varphi ({\rm H}_+)=(0,\infty)\times (0,\infty),
\\
\psi(\R)=(-\infty,0)\times \{0\},  & \psi ({\rm H}_+)=(-\infty,0)\times (0,\infty).
\end{cases}$
%and
%$\psi(\R)=(-\infty,0)\times \{0\}.$
%%(5.1.2)%%%%%
\item[(\ref{def: 4.1}.2)]
$\varphi (\overline{\alpha})=\overline{\varphi (\alpha)}$
and
$\psi (\overline{\alpha})=\overline{\psi (\alpha)}$
for any $\alpha \in \C$.
\end{enumerate}
For each monic polynomial
$f(z)=\prod_{k=1}^d(z-x_k)\in \P_d(\C)$,
let $\varphi (f)$ and $\psi (f)$ denote the monic
polynomials of the same degree $d$ given by
%%(5.1)%%
\begin{equation}\label{eq: loop-poly}
%%%%
\widetilde{\varphi}(f)=\prod_{k=1}^d(z-\varphi (x_k))
\quad
\mbox{and}\quad
\widetilde{\psi} (f)=\prod_{k=1}^d(z-\psi(x_k)).
%%%%
\end{equation}
%%%%%
%%%
}
%%%%%
\end{definition}
%%%%%%%%%%%%%%%%

%%(Remark 5.2)%%
\begin{remark}\label{rmk: loop product}
%%%%%%%%%
{\rm
Let $\K=\R$ or $\C$. 
It is easy to see that the following hold:
%%%
\begin{enumerate}
%%%(1)%%
\item[(i)]
%%%%%%%%
If $f(z)\in \P_d(\R)$, then
$\widetilde{\varphi}(f)\in \P_d(\R)$ and $\widetilde{\psi} (f)\in \P_d(\R).$ 
%%%%%%%%%%
%%%(ii)%%
\item[(ii)]
%%%%%%%%
If $f=(f_1(z),\cdots ,f_m(z))\in \Po^{d_1,m}_n(\K)$
and $g=(g_1(z),\cdots ,g_m(z))\in \Po^{d_2,m}_n(\K)$,
$
\big(\widetilde{\varphi} (f_1)\widetilde{\psi} (g_1),\cdots ,
\widetilde{\varphi}(f_m)\widetilde{\psi} (g_m)\big)
\in \Po^{d_1+d_2,m}_n(\K).
$
%%%(iii)%%
\item[(iii)]
%%%%%%%%
If $f=(f_1(z),\cdots ,f_m(z))\in \Q^{d_1,m}_n(\R)$
and $g=(g_1(z),\cdots ,g_m(z))\in \Q^{d_2,m}_n(\R)$,
$
\big(\widetilde{\varphi} (f_1)\widetilde{\psi} (g_1),\cdots ,
\widetilde{\varphi}(f_m)\widetilde{\psi} (g_m)\big)
\in \Q^{d_1+d_2,m}_n(\R).
$
\qed
%%%%%%%%%%
%%%
\end{enumerate}
%%%%%%
}
%%%
\end{remark}
Using an idea from \cite[Definition 4.9]{BM2} and  Remark \ref{rmk: loop product} we  will now define loop products.
%%
%%
%%(Definition 5.3)%%
\begin{definition}
%%%%
{\rm
\par
(i)
Define the loop product
%%(5.2)%%
\begin{equation}\label{def: productC}
%%%%
\mu_{d_1,d_2}^{\C}:
\Po^{d_1,m}_n(\C)
\times
\Po^{d_2,m}_n(\C)
\to 
\Po^{d_1+d_2,m}_n(\C)
\quad
\mbox{by}
%%%
\end{equation}
%%%%
%%()%%
\begin{equation*}
%%%
\mu_{d_1,d_2}^{\C}(f,g)=
\big(\widetilde{\varphi} (f_1)\widetilde{\psi} (g_1),\cdots ,
\widetilde{\varphi}(f_m)\widetilde{\psi} (g_m)\big)
%%%
\end{equation*}
%%%
for 
$(f,g)\in \Po^{d_1,m}_n(\C)\times
\Po^{d_2,m}_n(\C)$, 
where we write
$$
(f,g)=((f_1(z),\cdots ,f_m(z)),
(g_1(z),\cdots ,g_m(z))).
$$
\item[(ii)]
It is easy to see that
%%%%(5.3)%%
\begin{equation}
%%%
\mu_{d_1,d_2}^{\C}
(\Po^{d_1,m}_n(\R)\times \Po^{d_2,m}_n(\R))
\subset
\Po^{d_1+d_2,m}_n(\R).
\end{equation}
%%%%
Hence, one can define the loop product
%%%%%(5.4)%%%%%
\begin{equation}
%%%%
\mu_{d_1,d_2}^{\R}:
\Po^{d_1,m}_n(\R)
\times
\Po^{d_2,m}_n(\R)
\to 
\Po^{d_1+d_2,m}_n(\R)
%%%
\end{equation}
as the restriction
%%%
$\mu_{d_1,d_2}^{\R}=
\mu_{d_1,d_2}^{\C}
\vert
\Po^{d_1,m}_n(\R)
\times
\Po^{d_2,m}_n(\R).$
%%%%%
%\end{equation}
%%%%%%%%%
%%%%
\par
\item[(iii)]
Since the following relation also holds
%%(5.5)%%
\begin{equation}
%%%
\mu_{d_1,d_2}^{\R}\big(\Po^{d_1,m}_n(\R;\Ha_+)
\times
\Po^{d_2,m}_n(\R;\Ha_+)\big)\subset \po^{d,m}_n(\R;\Ha_+),
\end{equation}
%%%%
define the loop product
%%(5.6)%%
\begin{equation}\label{def: product H}
%%%%
\mu_{d_1,d_2}^{\Ha_+}:
\Po^{d_1,m}_n(\R;\Ha_+)
\times
\Po^{d_2,m}_n(\R;\Ha_+)
\to 
\Po^{d_1+d_2,m}_n(\R;\Ha_+)
%%%
\end{equation}
%%%%
as the restriction
$\mu^{\Ha_+}_{d_1,d_2}=\mu^{\R}_{d_1,d_2}\vert
\Po^{d_1,m}_n(\R;\Ha_+)
\times
\Po^{d_2,m}_n(\R;\Ha_+).$
%%
%%%%
\par
%%%%%%%%
(iv)
Similarly, we define the loop product
%%(5.7)%%
\begin{equation}\label{eq: loop product Q}
%%%%
\mu_{d_1,d_2}:\Q^{d_1,m}_n(\R)\times
\Q^{d_2,m}_n(\R)\to \Q^{d_1+d_2,m}_n(\R)
\quad
\mbox{by}
%%%
\end{equation}
%%
%%()%%
\begin{equation*}
%%%
\mu_{d_1,d_2}(f,g)=
\big(\widetilde{\varphi} (f_1)\widetilde{\psi} (g_1),\cdots ,
\widetilde{\varphi}(f_m)\widetilde{\psi} (g_m)\big)
%%%
\end{equation*}
%%%
for 
$(f,g)=((f_1(z),\cdots ,f_m(z)),
(g_1(z),\cdots ,g_m(z)))\in 
\Q^{d_1,m}_n(\R)
\times \Q^{d_2,m}_n(\R)$.
%%%%%%%%
}
\end{definition}
%%%(End of Definition 5.3)%%
%%
\par
Next, recall the definitions of stabilization maps. 
%%%(Definition 5.4)%%
\begin{definition}\label{def: stabd}
%%%
{\rm 
\par
Let $\K=\R$ or $\C$. 
For each integer $d\geq 1$,
let $\{x_{d,i}:1\leq i\leq m\} \subset (d,d+1)$ be any fixed
real numbers such that $x_i\not=x_k$ if $i\not=k$, 
and
let
$\varphi_d:\C\stackrel{\cong}{\longrightarrow} 
\C_d=\{\alpha \in \C:\mbox{Re }(\alpha )<d\}$
be any homeomorphism satisfying the following condition:
%%(5.7.1)%%
\begin{enumerate}
\item[(\ref{eq: loop product Q}.1)]
$\varphi_d(\R)=(-\infty,d)\times \R$
and $\varphi_d(\overline{\alpha})=\overline{\varphi_d(\alpha)}$
for any $\alpha\in \C$, where
we identify $\C=\R^2$ in a usual way.
\end{enumerate}
%%%
\par
(i) 
Define the stabilization map
$s^{d,m}_{n,\K}:\Po^{d,m}_n(\K)\to \Po^{d+1,m}_n(\K)$ by
%%%%%(5.8)%%%
\begin{equation}\label{equ: stabilization map for F}
%%%%%%%%%%%
s^{d,m}_{n,\K}(f_1(z),\cdots ,f_m(z))=
\big(
(z-x_{d,1})\widetilde{\varphi_d}(f_1),\cdots ,
(z-x_{d,m})\widetilde{\varphi_d}(f_m)\big)
%%%%
\end{equation}
%%%%
for $(f_1(z),\cdots ,f_m(z))\in\Po^{d,m}_n(\K)$, where
we set $\widetilde{\varphi_d}(f)=\prod_{k=1}^d(z-\varphi_d(x_k))$
if $f=f(z)=\prod_{k=1}^d(z-x_k)\in \P_d(\K).$
\par\vspace{1mm}\par
%%%
Note that the map $s^{d,m}_{n,\K}$ depends on the choice of
points $\{x_{d,k}\}_{k=1}^m$ and the homeomorphism $\varphi_d$, but
its homotopy type does not, as in \cite[Def. 3.11]{KY9}.
%%%
\par
%(ii)
%\par
(ii)
Let
$s^{d,m}_n:\Q^{d,m}_n(\R)\to \Q^{d+1,m}_n(\R)$
%%%
be the stabilization map given 
by
%%(5.9)%%%
\begin{equation}\label{equ: stabilization map for Q}
%%%%%%%%
s^{d,m}_{n}(f_1(z),\cdots ,f_m(z))=
\big(
(z-x_{d,1})\widetilde{\varphi_d}(f_1),\cdots ,
(z-x_{d,m})\widetilde{\varphi_d}(f_m)\big)
%%%%%%%%
\end{equation}
%%%%%%%%
for $(f_1(z),\cdots ,f_m(z))\in\Q^{d,m}_n(\R)$.
%%%%%
\par
(iii)
%Recall the stabilization map
%$s^{d,m}_{n,\R}:\Po^{d,m}_n(\R)\to \Po^{d+1,m}_{n,\R}(\R)$
%given in
%(\ref{equ: stabilization map for R}).
Since one can easily see that $s^{d,m}_{n,\R}(\po^{d,m}_{n,\Ha_+})\subset \po^{d+1,m}_{n,\Ha_+}$,
we define the stabilization map
$s^{d,m}_{n,\Ha_+}:\Po^{d,m}_n(\R;\Ha_+)\to \Po^{d+1,m}_{n}(\R;\Ha_+)$ by the restriction 
%%(5.10)%%
\begin{equation}\label{eq: stab H}
%%%
s^{d,m}_{n,\Ha_+}=s^{d,m}_{n,\R}\vert \Po^{d,m}_n(\R;\Ha_+).
\end{equation}
%%
%%%
}
%%%
\end{definition}
%%%%%%%%(End of definition 5.4)%%%%

%%%(Remark 5.5)%%
\begin{remark}
%%%%%%%
{\rm
It follows from the definitions of  (\ref{equ: stabilization map for F}),
%(\ref{equ: stabilization map for R}) 
(\ref{equ: stabilization map for Q}) and (\ref{eq: stab H})
that the following diagram is commutative:
%%%(5.11)%%
\begin{equation}\label{CD: diagram-stab}
%%%%
\begin{CD}
\po^{d,m}_n(\R;\Ha_+) 
@>\iota^{d,m}_{n,\Ha_+}>\subset>
\Po^{d,m}_n(\R) 
@>\iota^{d,m}_{n,\R}>\subset> 
\Q^{d,m}_n(\R) 
%@>i^{d,m}_{n,\C}>\subset> 
%\Po^{d,m}_n(\C)
%%%
\\
@V{s^{d,m}_{n,\Ha_+}}VV @V{s^{d,m}_{n,\R}}VV @V{s^{d,m}_{n}}VV 
\\
%%%
\po^{d+1,m}_n(\R;\Ha_+)
@>\iota^{d+1,m}_{n,\Ha_+}>\subset>
\Po^{d+1,m}_n(\R) 
@>\iota^{d+1,m}_{n,\R}>\subset> 
\Q^{d+1,m}_n(\R) 
\end{CD}
\end{equation}
%%%%%%%%%
%%%%%%
Moreover,
by using the method invented by C. Boyer and B. Mann \cite[Def. 4.9]{BM2} 
we obtain the following two
homotopy commutative diagrams:
%%
%%%%(5.12)%%%%%%%%%%%%%%
\begin{equation}\label{CD: loop sum2}
%%%%%%
\begin{CD}
\Po^{d_1.m}_n(\R;\Ha_+)\times \Po^{d_2,m}_n(\R;\Ha_+)
@>\mu_{d_1,d_2}^{\Ha_+}>>
\Po^{d_1+d_2,m}_n(\R)
\\
@V{\iota^{d_1,m}_{n,\Ha_+}\times
\iota^{d_2,m}_{n,\Ha_+}}V{\bigcap}V 
@V{\iota^{d_1+d_2,m}_{n,\Ha_+}}V{\bigcap}V
\\
\Po^{d_1.m}_n(\R)\times \Po^{d_2,m}_n(\R)
@>\mu_{d_1,d_2}^{\R}>>
\Po^{d_1+d_2,m}_n(\R)
\\
@V{\iota^{d_1,m}_{n,\R}\times
\iota^{d_2,m}_{n,\R}}V{\bigcap}V 
@V{\iota^{d_1+d_2,m}_{n,\R}}V{\bigcap}V
\\
%%%
%%%
\Q^{d_1,m}_n(\R)\times \Q^{d_2,m}_n(\R) 
@>{\mu_{d_1,d_2}}>> 
\Q^{d_1+d_2,m}_n(\R)
\\
@V{i^{d_1,m}_{n}\times i^{d_2,m}_{n}}VV @V{i^{d_1+d_2,m}_{n}}VV
\\
\Omega S^{mn-1}\times \Omega S^{mn-1}
%\Omega_{[d_1]_2}\RP^{mn-1}\times
%\Omega_{[d_2]_2}\RP^{mn-1} 
@>l_{S}>>
\Omega S^{mn-1}
%\Omega_{[d_1+d_2]_2}\RP^{mn-1}
%%%%
\end{CD}
%%%%%%%%%%%%%%
\end{equation}
%%%
where 
%$l_{\C}$ and 
$l_{S}$
denotes the loop product on the loop space
%$\Omega^2 S^{2mn-1}$ and 
$\Omega S^{mn-1}$.
\qed
%%%
}
\end{remark}
%%%%%(End of Remark 5.5)%%%

Now we are ready to give a proof of the following result. 
%%
%%%(Theorem 5.6)%%%
\begin{thm}\label{thm : I}
%%%%%%%
If  $mn\geq 3$, the 
inclusion map
%%%
$$
\iota^{d,m}_{n,\R}:\Po^{d,m}_n(\R)
\stackrel{\subset}{\longrightarrow} \Q^{d,m}_n(\R)
$$
induces a split epimorphism on the homology group $H_*(\ ;\Z)$. 
%%%%%
\end{thm}
%%%%%%
%%%%%%%%%%%%%%%%%%%%%%%%%%%
\begin{proof}
%%%%%%%%%%%%%%%%%%%%%%%%%
Recall that there is an isomorphism $H_*(\Omega S^{mn-1};\Z)=\Z[\iota_{mn-2} ]$ for
the generator $\iota_{mn-2}\in H_{mn-2}(\Omega S^{mn-1};\Z)\cong \Z$.
It follows from (\ref{eq: James})
%Theorem \ref{thm: KY10} 
that
there is a generator $\iota_{Q}\in H_{mn-2}(\Q^{d,m}_n(\R);\Z)\cong \Z$
which satisfies the equality
%%%(5.13)%%
\begin{equation}
(i^{mn-2,m}_n)_*(\iota_Q )=\iota_{mn-2},
\end{equation}
%%%%
where $n_0=\lfloor d/n\rfloor$ and
$H_*(\Q^{d,m}_n(\R);\Z)=\Z [\iota_Q]/((\iota_Q)^{n_0+1}).$
\par
Note that if a polynomial $f(z)\in \R [z]$ has a complex root
$\alpha\in \C\setminus \R$,
its conjugate $\overline{\alpha}$ is also a root of $f(z)$.
Hence, we  see that
%%%(5.14)%%
\begin{equation}
\Po^{d,m}_n(\R)=\Q^{d,m}_n(\R) 
\quad
\mbox{ if $d<2n$.}
\end{equation}
Thus, the inclusion map $\iota^{d,m}_{n,\R}$ is the identity map if $d<2n$.
\par
Since there is a homotopy equivalence
$\Po^{n,m}_n(\R)\cong \R^{mn}\setminus \R
\simeq S^{mn-2}$,
there is a generator
$
\iota_P\in H_{mn-2}(\Po^{n,m}_n(\R);\Z)\cong\Z
$
such that
%%(5.15)%%
\begin{equation} 
(\iota^{mn-2,m}_{n,\R})_*(\iota_P)=\iota_Q.
\end{equation}
Thus, for any $1\leq k\leq d_0=\lfloor d/n\rfloor$,
%since $i^{d,m}_n\vert \Q^{kn,m}_n(\R)$,
by using the diagrams (\ref{CD: loop sum2}) and 
(\ref{CD: diagram-stab}) we see that
$
(\iota_Q)^k=((\iota^{mn-2,m}_{n,\R})_*(\iota_P))^k
=(\iota^{k(mn-2),m}_{n,\R})_*((\iota_P)^k).
$
%%%
However, since $k(mn-2)\leq d$,
$$(\iota_Q)^k=(\iota^{k(mn-2),m}_{n,\R})_*((\iota_P)^k)
\in (\iota^{d,m}_{n,\R})_*(H_{k(mn-2)}(\Po^{d,m}_n(\R);\Z)).
$$
Hence,
 the map $\iota^{d,m}_{n,\R}$ induces a split epimorphism  on
the homology group $H_k(\ ;\Z)$ for any $k$.
\end{proof}
%%%(End of Proof of Proposition 5.6)%%%%%

%%%%%%%%%%%%%%%%%%%%%%%%%%%%%%%%%%%%%%%%%%%%%
%%(SECTION 6)%%%
\section{The homology stability theorem}\label{section: sd}
%%%%%%%%%%%%%%%%%%%%%%%%%%%%%%%%%%%%%%%%%%%%%
%%
%%
%%
In this section we consider the homology stability of the stabilization maps
$s^{d,m}_{n,\R}$
 and prove the homology stability
theorem (Theorem \ref{thm: stab1}).

\par\vspace{2mm}\par
Consider the stabilization map
$s^{d,m}_{n,\R}:\Po^{d,m}_n(\R)\to \Po^{d+1,m}_n(\R)$
given by (\ref{equ: stabilization map for F}).
Note that
the map $s^{d,m}_{n,\R}$ clearly extends to a map
%%(6.1)%%
\begin{equation}
%%%
\P_d(\R)^m\to\P_{d+1}(\R)^m
\end{equation}
and
its restriction gives a stabilization map
%%%%%
%%()%%
%\begin{equation}
$\tilde{s}^{d,m}_{n,\R}:\Sigma^{d,m}_{n}\to \Sigma^{d+1,m}_{n}$
%\end{equation}
%%%%%%%
between discriminants.
It is easy to see that it  also extends to an open embedding
%%
%%%(6.2)%%
\begin{equation}\label{equ: open-stab}
%%%%%%%
\tilde{s}^{d,m}_{n,\R}:\Sigma^{d,m}_{n}\times\R^m\to \Sigma^{d+1,m}_{n}.
\end{equation}
%%%%%%%%%%%%%
%%%%
Since one-point compactification is contravariant for open embeddings,
it
induces the map
%%(6.3)%%
\begin{equation}\label{equ: embedding3}
%%%%%%%%%%%%%%%
(\tilde{s}^{d,m}_{n,\R})_+: 
(\Sigma^{d+1,m}_{n})_+
\to
(\Sigma^{d,m}_{n}\times \R^{m})_+=\Sigma^{d,m}_{n +}\wedge S^{m}
\end{equation}
%%%%%%
between one-point compactifications.
Thus we see that  the following diagram is commutative
%%%(6.4)%%
\begin{equation}\label{diagram: discriminant}
%%%%%%%%%%%%
\begin{CD}
%%%
\tilde{H}_k(\Po^{d,m}_n(\R);\Z) 
@>{(s^{d,m}_{n,\R})}_*>>\tilde{H}_k(\Po^{d+1,m}_n(\R);\Z)
\\
@V{Al}V{\cong}V @V{Al}V{\cong}V
\\
H^{dm-k-1}_c(\Sigma^{d,m}_{n};\Z) 
@>(\tilde{s}^{d,m}_{n,\R+})^*>>
H^{(d+1)m-k-1}_c(\Sigma^{d+1,m}_{n};\Z)
%%%%
\end{CD}
%%%%%%%%%%
\end{equation}
%%%% 
where $Al$ denotes the Alexander duality isomorphism and
the homomorphism
 $(\tilde{s}^{d,m}_{n,\R+})^*$ denotes the composite of the
%homomorphisms 
suspension isomorphism with the homomorphism
$\tilde{s}^{d,m*}_{n,\R+}$,
{\small
$$
H^{dm-k-1}_c(\Sigma^{d,m}_{n};\Z)
\stackrel{\cong}{\longrightarrow}
H^{(d+1)m-k-1}_c(\Sigma^{d,m}_{n}\times\R^m;\Z)
\stackrel{\tilde{s}^{d,m*}_{n,\R+}}{\longrightarrow}
H^{(d+1)m-k-1}_c(\Sigma^{d+1,m}_{n};\Z).
$$
}
\newline
The map $\tilde{s}^{d,m}_{n,\R}$  naturally extends to a filtration preserving
open map 
%%%(6.5)%%
\begin{equation}
%%%%%%%
\hat{s}^{d,m}_{n,\R}:\SZ_{} \times \R^m\to \mathcal{X}^{d+1}_{}
\end{equation}
%%%%
and this extends to a filtration preserving map
$
(\hat{s}^{d,m}_{n,\R})_+: \mathcal{X}^{d+1}_{+}\to \SZ_{+}\wedge S^m.
$
This map  induces
 a homomorphism of spectral sequences
%%%%%%%%%%%
%%%(6.6)%%%%
\begin{equation}\label{equ: theta1}
%%%%%%%%%%%%
\{ \theta_{k,s}^t:E^{t;d}_{k,s}\to E^{t;d+1}_{k,s}\}.
\end{equation}
%%%
It is easy to see that
$\hat{s}^{d,m}_{n,\R}(\SZ_{k}(i,j)\times \R^m)\subset
\mathcal{X}^{d+1}_{k}(i,j)$ if $1\leq k\leq \lfloor d/n\rfloor$
and $i+j=k$.
Hence, if $i+j=k$ and
$1\leq k\leq \lfloor d/n\rfloor$,
one can define a map
%%%(6.7)%%
\begin{equation}
%%%%
\hat{s}^{d,m}_{n;i,j}:\SZ_{k}(i,j)\times \R^m \to \mathcal{X}^{d+1}_{k}(i,j)
\end{equation}
%%%%
as
the restriction 
$\hat{s}^{d,m}_{n;i,j}=\hat{s}^{d,m}_{n,\R}\vert
\SZ_{k}(i,j)$.
%%%
Since $\hat{s}^{d,m}_{n;i,j}$ is an open embedding, it induces a map
%%(6.8)%%
\begin{equation}\label{eq: 4.9}
%%%
\hat{s}^{d,m}_{n;i,j+}:\mathcal{X}^{d+1}_{k}(i,j)_+\to
\mathcal{X}^d_{k}(i,j)_+\wedge S^m
\end{equation}
%%%%%%%%%%
such that 
%it induces the homomorphism
%%()%%
 the following equality holds:
%%(6.9)%%
\begin{align}\label{eq: stab-decomp}
%%%%%%%
\theta^1_{k,s}&=\sum_{i+j=k}
(\hat{s}^{d,m}_{n;i,j+})^*
:E^{1;d}_{k,s}\to E^{1;d+1}_{k,s}
\quad
\mbox{if $1\leq k\leq \lfloor \frac{d}{n}\rfloor$, }
\\
\nonumber &  \quad \mbox{ where}
\\
\nonumber
(\hat{s}^{d,m}_{n;i,j+})^*&:
H_c^{dm+k-s-1}(\SZ_{k}(i,j);\Z)
\to
H_c^{(d+1)m+k-s-1}(\mathcal{X}^{d+1}_{k}(i,j);\Z),
\\
\nonumber
\mbox{and }&
\quad
\begin{cases}
E^{1;d}_{k,s}
&= \dis \bigoplus_{i+j=k}
H_c^{dm+k-s-1}(\SZ_{k}(i,j);\Z),
\\ 
E^{1;d+1}_{k,s}
&=\dis  \bigoplus_{i+j=k}
H_c^{(d+1)m+k-s-1}(\mathcal{X}^{d+1}_{k}(i,j);\Z).
\end{cases}
\end{align}
%%

%%%(Lemma 6.1)%%%
\begin{lemma}\label{lmm: E1}
%%%%%%%%%%%%%%%%%
If $1\leq k\leq \lfloor d/n\rfloor$, 
$\theta^1_{k,s}:
E^{1;d}_{k,s}\stackrel{\cong}{\rightarrow} 
{E}^{1;d+1}_{k,s}$ is
an isomorphism for any $s$.
\end{lemma}
%%%%%%%%%%%%%%%%
\begin{proof}
%%%%%%%%%%%%%%%%
Suppose that
$1\leq k\leq \lfloor d/n\rfloor$ with $i+j=k.$
By (\ref{eq: stab-decomp})
it suffices to show that
$(\hat{s}^{d,m}_{n;i,j+})^*$ is an isomorphism.
%%%%
Note that the projection 
$\pi^d_{k;i,j}$ (defined in (\ref{eq: projection}))
also naturally extends to a map
$\hat{\pi}^d_{k;i,j}:\SZ_{k}(i,j)\times \R^m\to C_i(\R)\times \C_j(\Ha_+)$
and  it is easy to check that
the following diagram is commutative.
$$
\begin{CD} 
\SZ_{k}(i,j)\times \R^{m} @>\hat{\pi}_{k;i,j}^d>> C_i(\R)\times \C_j(\Ha_+)
\\
@V{\hat{s}^{d,m}_{n;i,j}}VV \Vert @.
\\
\mathcal{X}^{d+1}_{k}(i,j)  
@>\pi_{k;i,j}^{d+1}>> C_{i}(\R)\times C_j(\Ha_+)
\end{CD}
$$
It follows from the naturality  of
the Thom isomorphism that the homomorphism
$(\hat{s}^{d,m}_{n;i,j+})^*$ is indeed an isomorphism
if $1\leq k\leq \lfloor d/n\rfloor$.
%%%%%%%%
\end{proof}
%%(End of proof of Lemma 6.1)%%%%

Now we can prove the following result.
%%%(Theorem 6.2)%%
\begin{thm}\label{thm: stab1}
%%%%%%%%%%%%%%%%
If $mn\geq 3$,
the stabilization map 
$$
s^{d,m}_{n,\R}:\Po^{d,m}_n(\R)\to
\Po^{d+1,m}_n(\R)
$$ 
is a homology equivalence for 
$\lfloor d/n\rfloor =\lfloor (d+1)/n\rfloor$, and a
homology equivalence through dimension
$D(d;m,n)$ for
$\lfloor d/n\rfloor <\lfloor (d+1)/n\rfloor$,
where $D(d;m,n)$ denotes the positive integer given by
(\ref{eq: D(d;m,n)}).
%$D(d;m,n)=(mn-2)(\lfloor \frac{d}{n}\rfloor +1)-1$.
%%%%%%%%%%%%%%
\end{thm}
%%%%%%%%%%%%%%
%%%(Proof of Theorem 6.2)%%%%
\begin{proof}
%%%%%%%%%%%%%%%%%%%%%%%%%%%%%
First, consider the case 
$\lfloor d/n \rfloor =\lfloor (d+1)/n\rfloor$.
In this case, by using Corollary \ref{crl: Er} and 
Lemma \ref{lmm: E1} it is easy to show that
$\theta^{1}_{k,s}:E^{1;d}_{k,s}\stackrel{\cong}{\longrightarrow}
E^{1;d+1}_{k,s}$ 
is an isomorphism for any $(k,s)$.
Hence, $\theta^{\infty}_{k,s}$ 
is an isomorphism for any $(k,s).$
Since $\theta^t_{k,s}$ is induced from $\hat{s}^{d,m}_{n,\R}$,
it follows from (\ref{diagram: discriminant}) that
the map $s^{d,m}_{n,\R}$ is a homology equivalence.
%%%
\par
%%%
Next assume that
$\lfloor d/n\rfloor <\lfloor (d+1)/n\rfloor$,
i.e. $\lfloor (d+1)/n\rfloor =\lfloor d/n\rfloor +1$. and let $\epsilon \in \{0,1\}$.
In this case,
by considering the differential
$d^t:E^{t;d+\epsilon}_{k,s}
\to E^{t;d+\epsilon}_{k+t,s+t-1}$,
 Lemma \ref{lmm: E1}
and Corollary \ref{crl: Er}, one can show that
$\theta^t_{k,s}:E^{t;d}_{k,s}\to E^{t;d+1}_{k,s}$ 
is an isomorphism for any $(k,s)$ and any $t\geq 1$ as long as
the condition $s-t\leq D(d;m,n)$ is satisfied.
Hence, if $s-t\leq D(d;m,n)$,
$\theta^{\infty}_{k,s}$ is always an isomorphism and 
so the map $s^{d,m}_{n,\R}$ is a homology equivalence
through dimension $D(d;m,n)$.
%%%
%%%%%%%%%%%%%%%%%%%%%%%%%%%%%%%%
%%%%
\end{proof}
%%%(End of proof of Theorem 6.2)%%%%%%
%%%%%%%%%%%%%%%%%%%%%%%%%%%%%%%%%%%%%%%

The following two results will be needed for the proof of Theorem
\ref{thm: II}.
%(Theorem \ref{thm: I} and Corollary \ref{thm: II}).

%%(Lemma 6.3)%%%
\begin{lemma}\label{lmm: abelian}
%%%%%%%%%%%%%%%
The space $\Po^{d,m}_n(\R)$ is 
simply connected if $mn\geq 4$, and
$\pi_1(\po^{d,m}_n(\R))=\Z$ 
if $mn=3$ and $d\geq n$.
%%%%%%%%%%%%%%
\end{lemma}
%%%%%%%%%%%%%%
\begin{proof}
%%%%%%%%%%%%%
%Assume that $mn\geq 3$.
We use the description of the fundamental group in terms of braids as in 
\cite[Lemma 3.5]{KY10}.
We represent elements of the group as  strings of $m$-different colors, which can move continuously,  with only the following case not allowed to occur:
\begin{enumerate}
\item[$(*)$]
All strings of multiplicity $\geq n$ of $m$-different colors pass through a single point.
\end{enumerate}
%%%
%that different two kinds of strings can  freely pass through one another.
Since $mn\geq 3$, this representation shows  that 
$\pi_1(\Po^{d,m}_n(\R))$ is an abelian group.
Hence, there is an isomorphism
$\pi_1(\Po^{d,m}_n(\R))\cong H_1(\Po^{d,m}_n(\R);\Z)$.
Now consider the spectral sequence (\ref{SS}).
Then, by   Corollary  \ref{crl: Er} , we see that 
$$
E^{1;d}_{k,k+1}=
\begin{cases}
\Z & \mbox{ if }k=1 \mbox{ and } mn=3,
\\
0 & \mbox{ otherwise}.
\end{cases}
$$
First, consider the case $mn\geq 4$.
We see that
$H_1(\Po^{d,m}_n(\R); \Z)=0$, and
the space $\po^{d,m}_n(|\R)$ is simply connected.
\par
Next, consider the case $mn=3$.
Since $E^{1;d}_{2,2}=0$ by Corollary  \ref{crl: Er}, 
by considering the differential 
$d^t:E^{t;d}_{k,s}\to E^{t;d}_{k+t,s+t+1}$, 
% it follows  from dimensional reasons that
we see that
$
\bigoplus_{k\in \Z}E^{1;d}_{k,k+1}=E^{1;d}_{1,2}=\Z=
E^{\infty;d}_{1,2}.
$
Hence, if $mn=3$,
$$
\pi_1(\po^{d,m}_n(\R))\cong H_1(\po^{d,m}_n(\R);\Z)\cong \Z.
$$
This completes the proof.
%Hence, the space $\Po^{d,m}_n(\R)$ is 
%simply connected.
%%%%%%%
\end{proof}
%%%%%%%%%%%(End of proof of Lemma 6.3)%%%%
%%
%%
%%
%%%
%%%
%%%
%(Corollary 6.4)%%
\begin{crl}\label{crl: stabilization}
%%%%%%%
If $mn\geq 4$, the stabilization map
$$
\mbox{s}^{d,m}_{n,\R}:\Po^{d,m}_n(\R)\to \Po^{d+1,m}_n(\R)
$$
is a homotopy equivalence if
$\lfloor d/n\rfloor =\lfloor (d+1)/n\rfloor$
and a homotopy equivalence through dimension $D(d;m,n)$
otherwise.
\end{crl}
%%%%%%%%%%
\begin{proof}
%%%%
This follows from Theorem \ref{thm: stab1}
and Lemma \ref{lmm: abelian}.
\end{proof}
%%(End of proof of Corollary 6.4)%%%
%%
%%%(Definition 6.5)%%%
\begin{definition}
%%%%%%%%%%%%%%%%%%%%%
{\rm
Let 
$D=(d_1,\cdots ,d_m)\in \N^m$ 
be an $m$-tuple of positive integers.
Let $x_{d}\in (d,d+1)$ be any fixed
real number 
and
let
$\varphi_d:\C\stackrel{\cong}{\longrightarrow} 
\C_d=\{\alpha \in \C:\mbox{Re }(\alpha )<d\}$
be any fixed homeomorphism satisfying the condition
(\ref{def: stabd}.1).
\par
Then for each $1\leq i\leq m$ and $\K=\R$ or $\C$, let
%%%(6.10)%%
\begin{equation}
s^{D,i;m}_{n,\K}:\Po^{d_1,\cdots ,d_m;m}_n(\K)
\to
\po^{d_1,\cdots ,d_{i-1},d_i+1,d_{i+1},\cdots ,d_m;m}_n(\K)
\end{equation}
%%%%%%
denote the stabilization map defined by
$$
s^{D,i;m}_{n,\K}(f)=\big(\widetilde{\varphi_d}(f_1),\cdots ,
\widetilde{\varphi_d}(f_{i-1}),(z-x_d)\widetilde{\varphi_d}(f_i),
\widetilde{\varphi_d}(f_{i+1}),\cdots ,\widetilde{\varphi_d}(f_{m})\big)
$$
for $f=(f_1(z),\cdots ,f_m(z))\in\po^{d_1,\cdots ,d_m;m}_n(\K)$.
}
%%%%%
\end{definition}
%%%%(End of Definition 6.5)%%%%%%%

%%%%%%(Theorem 6.6)%%
\begin{thm}\label{thm: generalized stabilization}
%%%%%%%
Let $1\leq i\leq m$, $mn\geq 3$ and 
$D=(d_1,\cdots ,d_m)\in \N^m$ 
be an $m$-tuple of positive integers.
\par
$\I$
The stabilization map
$$
s^{D,i;m}_{n,\R}:\Po^{d_1,\cdots ,d_m;m}_n(\R)
\to
\po^{d_1,\cdots ,d_{i-1},d_i+1,d_{i+1},\cdots ,d_m;m}_n(\R)
$$
is a homology equivalence if
$\lfloor d_i/n\rfloor =\lfloor (d_i+1)/n\rfloor$
and a homology equivalence through dimension $D(d_i;m,n)$
otherwise.
\par
$\II$
The stabilization map
$$
s^{D,i;m}_{n,\C}:\Po^{d_1,\cdots ,d_m;m}_n(\C)
\to
\po^{d_1,\cdots ,d_{i-1},d_i+1,d_{i+1},\cdots ,d_m;m}_n(\C)
$$
is a homotopy equivalence if
$\lfloor d_i/n\rfloor =\lfloor (d_i+1)/n\rfloor$
and a homotopy equivalence through dimension $D(d_i;m,n;\C)$
otherwise.
\end{thm}
%%%%%%%%%%%%%%%%%
\begin{proof}
%%%%%%%%%%%%%%%%%
The assertion (i) can be proved completely same way as  that of Theorem \ref{thm: stab1}
%and Lemma \ref{lmm: abelian}.
 and the assertion (ii)
can be proved analogously to \cite[Theorem 1.8]{KY8}.
%Since these assertion are not used in this paper, we omit the detail.
%%%%%%%
\end{proof}
%%%%%%%%%%(End of proof of Theorem 6.6)%%%%%%
%%
%%
%%

%%%(SECTION 7)%%%
\section{Configuration spaces and scanning maps}
\label{section: scanning maps}
%%%%%%%%%%%%%%%%%

In this section we define the  \lq\lq horizontal scanning maps\rq\rq and then use them to prove our stable results  (Theorems \ref{thm: scanning map} and \ref{thm: stable result2}).
%and recall the two important results
%(Theorems \ref{thm: II} and \ref{thm: Theorem H}).
%We continue to assume that $\K =\R$ or $\C$.

%%(Definition 7.1)%%%
\begin{definition}
%%%%%%
{\rm
For a space $X$, let $\SP^d(X)$ denote the $d$-th {\it symmetric product}
defined by the quotient space 
%%%(7.1)%%
\begin{equation}
\SP^d(X)=X^d/S_d,
\end{equation}
%%%%
where
the symmetric group $S_d$ of $d$ letters acts on $X^d$ by the
permutation of coordinates.
Since $F(X,d)$ is an $S_d$-invariant subspace of $X^d$ and
$C_d(X)=F(X,d)/S_d$, there is a natural inclusion
$C_d(X)\subset \SP^d(X).$
\par
Note that
an element $\alpha\in\SP^d(X)$ may be identified with
the formal linear combination
%%%(7.2)%%
\begin{equation}\label{combination}
%%%%%%%%%%%%%
\alpha =\sum_{i=1}^kd_ix_i,
\quad\quad
\mbox{where }\{x_i\}_{i=1}^k\in C_k(X)\ \  \mbox{and }\  \sum_{i=1}^kd_i=d.
\end{equation}
%%%%%%%%%%%
%where $\{x_i\}_{i=1}^k\in C_k(X)$ and  $\sum_{i=1}^kd_i=d$.  
We shall refer to $\alpha$
as {\it a configuration} (or {\it $0$-cycle}) %of the point $x_i$ 
having 
{\it a multiplicity}
$d_i$ at the point $x_i$.
%%%
}
\end{definition}
%%%

%%%(Remark 7.2)%%%
\begin{remark}
{\rm
(i)
If we use the notation (\ref{combination}),
the space $\P_d(\C)$ can be easily identified with the space $\SP^d(\C)$ by the homeomorphism
$\Phi_d:\P_d(\C)\stackrel{\cong}{\longrightarrow}\SP^d(\C)$ defined by
%%(7.3)%%
\begin{equation}\label{eq: varphi}
%%%%%%%
\Phi_d\Big(\prod_{i=1}^k(z-\alpha_i)^{d_i}\Big)=  \sum_{i=1}^{k}d_i\alpha_i.
\end{equation}
%%%%%%%%%%
\par
(ii)
Since there is a natural inclusion
$\P_d(\R)\subset \P_d(\C)$, one can define a subspace
$\SP^d_{\R}\subset \SP^d(\C)$ as the image 
%%(7.4)%%
\begin{equation}
\SP^d_{\R}=\Phi_d(\P_d(\R)).
\end{equation}
%%%
Then the restriction gives a homeomorphism
%%%%(7.5)%%
\begin{equation}
\Phi_d\vert \P_d(\R):\P_d(\R)
\stackrel{\cong}{\longrightarrow}
\SP^d_{\R}.
\end{equation}
%%%%%%%%%%
Note that any element $\alpha \in \SP^d_{\R}$ can be represented as the formal sum
%%%
%%(7.6)%%
\begin{equation}
%%%%%%
\alpha = \sum_{k=1}^sd_kx_k+\sum_{j=1}^te_j(y_j+\overline{y_j})
\quad
(d_k,e_j\in \N,\ x_k\in \R,\ y_j\in \Ha_+),
%%%%%%%
\end{equation}
%%%%%
where $x_k\not= x_l$ if $k\not= l$, 
$y_j\not= y_i$ if $j\not= i$ and $\sum_{k=1}^sd_k+2\sum_{j=1}^te_j=d.$
}
\end{remark}
%%%%(End of Remark 7.2)%%%%
%%
%%%(Definition 7.3)%%%%%%
\begin{definition}
%%%%%%%%%%%%%%
{\rm 
%%%%%%%%%%%
Let $X$ be a space and
$D=(d_1,\cdots ,d_m)\in (\Z_{\geq 0})^m$ be an $m$-tuple of non-negative  integers.
\par
(i)
Let $\SP^0(X)=\{\emptyset\}$ and let
$\SP^D(X)$ denote the space given by
%%%(7.7)%%%
\begin{equation}
%%%%%%%%%
\SP^D(X)=\SP^{d_1,\cdots ,d_m}(X)
=
\SP^{d_1}(X)\times \SP^{d_2}(X)\times \cdots \times \SP^{d_m}(X). 
\end{equation}
%%%%%%
We define the space
$\pol^{D;m}_n(X)=\pol^{d_1,\cdots ,d_m;m}_n(X)\subset \SP^D(X)$ by
%%(7.8)%%
\begin{equation}
\pol^{D;m}_n(X)=
\pol^{d_1,\cdots ,d_m;m}_n(X)
=\{(\xi_1,\cdots ,\xi_m)\in \SP^D(X): (*)_n\},
%%%%%%%
\end{equation}
%%%%%%%
where the condition $(*)_n$ is given by
\begin{enumerate}
\item[$(*)_n$:]
$\cap_{i=1}^m\xi_i$ does not contain any element of multiplicity $\geq n$.
\end{enumerate}
%%%%%%%%
\par
(ii)
When $X\subset \C$, let $\SP^D_{\R}(X)$ denote the space defined by
%%%(7.9)%%%
\begin{equation}
%%%%%%%%%
\SP^D_{\R}(X)=(\SP^{d_1}_{\R}\times \SP^{d_2}_{\R}\times \cdots \times \SP^{d_m}_{\R})\cap \SP^D(X). 
\end{equation}
%%%%%%
We define the space
$\mathcal{Q}^{D;m}_n(X)
=\mathcal{Q}^{d_1,\cdots ,d_m;m}_n(X)\subset \SP^D_{\R}(X)$ by
%%(7.10)%%
\begin{equation}
%%%%%%%%%
\mathcal{Q}^{D;m}_n(X)=
\mathcal{Q}^{d_1,\cdots ,d_m;m}_n(X)
=
\{(\xi_1,\cdots ,\xi_m)\in \SP^D_{\R}(X): (*)_{n}^{\R}\},
%%%%%%%
\end{equation}
%%%%%%%
where the condition $(*)_{n}^{\R}$ is given by
\begin{enumerate}
\item[$(*)_{n}^{\R}:$]
$(\cap_{i=1}^m\xi_i)\cap \R$ does not contain any element of multiplicity $\geq n$.
\end{enumerate}
%%%%%%%%
In particular, when $D_m=(d,\cdots ,d)\in \N^m$
$(m$-times),
we write
%%%(7.11)%%
\begin{equation}
\pol^{d,m}_n(X)=\pol^{D_m;m}_n(X)
\ \ 
\mbox{ and }\ \ 
\mathcal{Q}^{d,m}_n(X)=
\mathcal{Q}^{D_m;m}_n(X).
\end{equation}
%%%%%
\par
(iii) If $A\subset X$ is a closed subspace, we define the space $\pol^{D;m}_n(X,A)$ by
%%(7.12)%%
\begin{equation}
\pol^{D;m}_n(X,A)=\pol^{d_1,\cdots ,d_m;m}_n(X,A)=
\pol^{D;m}_n(X)/\sim ,
\end{equation}
where the equivalence relation \lq\lq$\sim$\rq\rq
is defined by
$$
(\xi_1,\cdots ,\xi_m)\sim
(\eta_1,\cdots ,\eta_m)
\ \mbox{ if } \
\xi_i\cap (X\setminus A)=\eta_i\cap (X\setminus A)
$$
for all $1\leq i\leq m$.
Therefore, points in $A$ are ignored.
%\par
When $A\not= \emptyset$, there is a natural inclusion
$$
\pol^{d_1,\cdots ,d_i,\cdots ,d_m;m}_n(X,A)\subset
\pol^{d_1,\cdots ,d_i+1,\cdots ,d_m;m}_n(X,A)
$$
by adding a point of $A$ into the $i$-th part.
Let $\pol^{0,\cdots ,0;m}_n(X,A)=\{\emptyset\}$
and
define the space $\pol^m_n(X,A)$ by the union
%%(7.13)%%
\begin{equation}
\pol^m_n(X,A)=\bigcup_{d_1\geq 0,\cdots ,d_m\geq 0}\pol^{d_1,\cdots ,d_m;m}_n(X,A).
\end{equation}
%%%%
\par
(iv) If $X\subset \C$ and $A\subset X$ is a closed subspace, we also define the space 
$\mathcal{Q}^{D;m}_n(X,A)$ by
%%(7.14)%%
\begin{equation}
\mathcal{Q}^{D;m}_n(X,A)
=\mathcal{Q}^{d_1,\cdots ,d_m;m}_n(X,A)
=\mathcal{Q}^{D;m}_n(X)/\sim ,
\end{equation}
where the equivalence relation \lq\lq$\sim$\rq\rq
is defined by
$$
(\xi_1,\cdots ,\xi_m)\sim
(\eta_1,\cdots ,\eta_m)
\ \mbox{ if } \
\xi_i\cap (X\setminus A)\cap \R=\eta_i\cap (X\setminus A)\cap \R
$$
for all $1\leq i\leq m$.
%%%
%%%
Again, when $A\not= \emptyset$, there is a natural inclusion
$$
\mathcal{Q}^{d_1,\cdots ,d_i,\cdots ,d_m;m}_n(X,A)\subset
\mathcal{Q}^{d_1,\cdots ,d_i+1,\cdots ,d_m;m}_n(X,A)
$$
by adding a point of $A$ into the $i$-th part.
Let
$\mathcal{Q}^{0,\cdots ,0;m}_n(X,A)=\{\emptyset\},$
and
define the space $\mathcal{Q}^m_n(X,A)$ by the union
%%(7.15)%%
\begin{equation}
\mathcal{Q}^m_n(X,A)
=\bigcup_{d_1\geq 0,\cdots ,d_m\geq 0}
\mathcal{Q}^{d_1,\cdots ,d_m;m}_n(X,A).
\end{equation}
%%%%
%%
\par
(v)
Let $\K=\R$ or $\C$, and
define the stabilized spaces $\po^{\infty,m}_n(\K)$ and
$\Q^{\infty,m}_{n}(\R)$ 
by the colimits
%%(7.16)%%
\begin{equation}\label{eq: stabilized space}
%\begin{cases}
\po^{\infty,m}_{n}(\K)
%&
= \dis \ \lim_{d\to\infty}\po^{d,m}_n(\K),
%\\
%\po^{\infty,m}_{n}(\R) &= \dis \ \lim_{d\to\infty}\po^{d,m}_n(\R),
%\\
\quad
\Q^{\infty,m}_{n}(\R)
%&
= \dis  \ \lim_{d\to\infty}\Q^{d,m}_n(\R)
%\end{cases}
\end{equation}
%%%
taken over the stabilization maps $\{s^{d,m}_{n,\K}\}$ and $\{s^{d,m}_n\}$ given by
(\ref{equ: stabilization map for F}) 
%(\ref{equ: stabilization map for R}) 
and (\ref{equ: stabilization map for Q}), respectively.
}
\end{definition}
%%%%%%%%%%(End of Definition 7.3)%%
%%%%

We will need  two kinds of horizontal scanning maps.
First, we define the scanning map for a configuration space of
particles.
From now on, we make the identification $\C =\R^2$
by identifying $\C\ni x+\sqrt{-1}y$ with $(x,y)\in \R^2$ in a usual way.
%%
%%(Definition 7.4)%%
\begin{definition}
%%%%%%%%%%%%%%%
{\rm 
For a rectangle $X$ in $\C=\R^2$, let $\sigma X$ denote
the union of the sides of $X$ which are parallel to the $y$-axis, and
for a subspace $Z\subset \C=\R^2$, let $\overline{Z}$ denote the closure of $Z$.
\par
For example, if $X=[a,b]\times [c,d]=\{(t,s)\in \R^2: a\leq t\leq b,c\leq s\leq d\}$, then
$\sigma X=\{a,b\}\times [c,d]=\{(t,s)\in \R^2:t\in \{a,b\},c\leq s\leq d\}.$
\par
From now on, let $I$ denote the interval
$I=[-1,1]$
and
let $\epsilon>0$ be a fixed positive real number.
For each $x\in\R$, let $V(x)$ be the set defined by
%%(7.17)%%
\begin{align}\label{eq: Vx}
%%%%
V(x)&=\{w\in\C: \vert \mbox{Re}(w)-x\vert  <\epsilon , \vert\mbox{Im}(w)\vert<1 \}
\\
&=(x-\epsilon ,x+\epsilon )\times (-1,1),
\nonumber
\end{align}
%%%%%%%%
%%
and let us identify $I\times I=I^2$ with the closed unit rectangle
$$
\{t+s\sqrt{-1}\in \C: -1\leq t,s\leq 1\}\subset \C.
$$ 
%%%
\par
(i)
First,
we define the {\it horizontal scanning map}
%%%(7.18)%%%%
\begin{equation}\label{eq: scan Q}
%%%%%%%%%%
sc^{d,m}_n:
\Q^{d,m}_{n}(\R )\to 
\Omega \mathcal{Q}^{m}_{n}(I^2,\partial I\times I)
=\Omega \mathcal{Q}^{m}_{n}(I^2,\sigma I^2)
\end{equation}
%%%
of the space $\Q^{d,m}_n(\R)$
as follows.
\par\vspace{2mm}\par
%%%%%%%%%%%%%%%%%%%%
For each $m$-tuple 
$\alpha =(\xi_1,\cdots ,\xi_m)\in \mathcal{Q}^{d,m}_{n}(\R)$
of configurations,
let
$
sc^{d,m}_n(\alpha):\R \to \mathcal{Q}^{d,m}_{n}(I^2,\partial I\times I)
=\mathcal{Q}^{d,m}_{n}(I^2,\sigma I^2)$
denote the map given by
%%%%%
\begin{align*}
%%%%%%
\R\ni x
&\mapsto
(\xi_1\cap\overline{V}(x),\cdots ,\xi_m\cap\overline{V}(x))
\in
\mathcal{Q}^{m}_{n}(\overline{V}(x),\sigma \overline{V}(x))
\cong \mathcal{Q}^{m}_{n}(I^2,\sigma I^2),
\quad
%\mbox{if }x\in \R
%\\
%\infty &\mapsto
%(\emptyset ,\emptyset,\cdots ,\emptyset )\in Q^{m}_{n,\K}(I^2,\partial I\times I)
%%%%%
\end{align*}
%for $x\in \R$,
where 
we use the canonical identification
%\begin{equation}
$(\overline{V}(x),\sigma \overline{V}(x))
\cong (I^2,\sigma I^2).$
%\end{equation}
%%%
\par\vspace{2mm}\par
%%%%%
Since $\dis\lim_{x\to\pm \infty}sc^{d,m}_n(\alpha)(x)
=(\emptyset ,\cdots ,\emptyset)$, %and $\dis\lim_{x\to-\infty}sc^{d,m}_n(\alpha)(x)=(\emptyset ,\cdots ,\emptyset)$, 
by setting $sc^{d,m}_n(\alpha)(\infty)=(\emptyset ,\cdots ,\emptyset)$
we obtain a based loop
$sc^{d,m}_n(\alpha)\in \Omega \mathcal{Q}^{m}_{n}(I^2,\sigma I^2),$
where we identify $S^1=\R \cup \infty$ and
we choose the empty configuration
$(\emptyset ,\cdots ,\emptyset)$ as the base point of
$\Q^{m}_{n}(I^2,\sigma I^2)$.
%%%
\par
%%
%%%%%%%%
 If we identify $\Q^{d,m}_n(\R)\cong
\mathcal{Q}^{d,m}_{n}(\R)=\mathcal{Q}^{D_m;m}_n(\R)$,
 finally we obtain the map
%%%%%%%%%%
%\begin{equation}
%%%%%%%%%%
$sc^{d,m}_n:
\Q^{d,m}_{n}(\R)\to 
\Omega \mathcal{Q}^{m}_{n}(I^2,\sigma I^2).
$
%%%%%%%%%
\par\vspace{1mm}\par
Since  $sc^{d+1,m}_{n}\circ s^{d,m}_{n}\simeq sc^{d,m}_{n}$
(up to homotopy equivalence),
by setting $\dis S_{}=\lim_{d\to\infty}sc^{d,m}_n$, 
we also obtain {\it the stable horizontal scanning map}
%%(7.19)%%%%%
\begin{equation}
%%%%%%%%%%%%
S_{}:
\Q^{\infty,m}_{n}(\R)
%=\lim_{d\to\infty}\Q^{d,m}_{n}(\R)
\to 
%\lim_{d\to\infty}\Omega_d
%Q^{m}_{n,\K}(I^2,\partial I\times I)
%\simeq 
\Omega
\mathcal{Q}^{m}_{n}(I^2,\partial I\times I)
=\Omega \mathcal{Q}^m_{n}(I^2,\sigma I^2),
%%%%
\end{equation}
%%%%
where $\Q^{\infty,m}_n(\R)$ denotes the stabilized space given by
(\ref{eq: stabilized space}).
\par\vspace{2mm}\par
(ii)
Next,
define the {\it horizontal scanning map}
%%%(7.20)%%%
\begin{equation}
sc^{d,m}_{n,\C}:\po^{d,m}_{n}(\C )\to \Omega \pol^{m}_{n}(I^2,\partial I\times I)
=\Omega \pol^{m}_{n}(I^2,\sigma I^2)
\end{equation}
%%%
of the space $\po^{d,m}_n(\R)$
as follows.
\par\vspace{2mm}\par
%%%%%%%%%%%%%%%%%%%%
For each $m$-tuple 
$\alpha =(\xi_1,\cdots ,\xi_m)\in \pol^{d,m}_{n}(\C)$
of configurations,
let
$
sc^{d,m}_{n,\C}(\alpha):\R \to \pol^{d,m}_{n}(I^2,\partial I\times I)
=\pol^{d,m}_{n}(I^2,\sigma I^2)$
denote the map given by
\begin{align*}
%%%%%%
\R\ni x
&\mapsto
(\xi_1\cap\overline{V}(x),\cdots ,\xi_m\cap\overline{V}(x))
\in
\pol^{m}_{n}(\overline{V}(x),\sigma \overline{V}(x))
\cong \pol^{m}_{n}(I^2,\sigma I^2),
\quad
%\mbox{if }x\in \R
%\\
%\infty &\mapsto
%(\emptyset ,\emptyset,\cdots ,\emptyset )\in Q^{m}_{n,\K}(I^2,\partial I\times I)
%%%%%
\end{align*}
%for $x\in \R$,
where 
we use the canonical identification
%\begin{equation}
$(\overline{V}(x),\sigma \overline{V}(x))
\cong (I^2,\sigma I^2).$
%\end{equation}
%%%
\par\vspace{2mm}\par
%%%%%
Since $\dis\lim_{x\to\pm \infty}sc^{d,m}_{n,\C}(\alpha)(x)
=(\emptyset ,\cdots ,\emptyset)$, %and $\dis\lim_{x\to-\infty}sc^{d,m}_n(\alpha)(x)=(\emptyset ,\cdots ,\emptyset)$, 
by setting $sc^{d,m}_{n,\C}(\alpha)(\infty)=(\emptyset ,\cdots ,\emptyset)$
we obtain a based loop
$sc^{d,m}_{n,\C}(\alpha)\in \Omega \pol^{m}_{n}(I^2,\sigma I^2),$
where we identify $S^1=\R \cup \infty$ and
we choose the empty configuration
$(\emptyset ,\cdots ,\emptyset)$ as the base point of
$\pol^{m}_{n}(I^2,\sigma I^2)$.
%%%
\par
%%%%%%%%
If we identify $\po^{d,m}_n(\C)\cong
\pol^{d,m}_{n}(\C)=\pol^{D_m;m}_n(\C)$,
finally  we obtain the map
%%%%%%%%%%
%\begin{equation}
%%%%%%%%%%
$sc^{d,m}_{n,\C}:
\po^{d,m}_{n}(\C)\to 
\Omega \pol^{m}_{n}(I^2,\sigma I^2).
$
%%%%%%%%%
%\end{equation}
%%%%%%%%%
\par\vspace{1mm}\par
%%%
%(iii)
%If we identify
%$\Q^{d,m}_{n}(\K)=Q^{d,m}_{n,\K}(\C)$ as in  (\ref{eq: QK})
%and
Since  $sc^{d+1,m}_{n,\C}\circ s^{d,m}_{n,\C}\simeq sc^{d,m}_{n,\C}$
(up to homotopy equivalence),
%up to homotopy equivalence,
by setting $\dis S_{\C}=\lim_{d\to\infty}sc^{d,m}_{n,\C}$
we obtain {\it  the stable horizontal scanning map}
%%(7.21)%%%%%
\begin{equation}
%%%%%%%%%%%%
S_{\C}:
\po^{\infty,m}_{n}(\C)
%=\lim_{d\to\infty}\Q^{d,m}_{n}(\R)
\to 
%\lim_{d\to\infty}\Omega_d
%Q^{m}_{n,\K}(I^2,\partial I\times I)
%\simeq 
\Omega
\pol^{m}_{n}(I^2,\sigma I^2),
%=\Omega Q^m_{n}(I^2,\sigma I^2).
%%%%
\end{equation}
%%%%
where
$\po^{\infty,m}_{n}(\C)$ denotes the stabilized space given by
(\ref{eq: stabilized space}). 
Although the scanning map itself depends on the choice of $\epsilon$, its homotopy class does not. 
%%%%
\par
(iii)
Let $\Z_2=\{\pm 1\}$ denote the multiplicative cyclic group of order $2$.
Complex conjugation in $\C$ naturally induces a $\Z_2$-action on $\po^{D;m}_n(\C)$. 
It is easy to see that
its fixed point set is
$\po^{d,m}_n(\C)^{\Z_2}=\po^{d,m}_n(\R)$.
Since the stabilization maps $\{s^{d,m}_{n,\C}\}_{d\geq 1}^{}$ are $\Z_2$-equivariant maps,
$\po^{\infty,m}_n(\C)^{\Z_2}=\po^{\infty,m}_n(\R)$.
Moreover, since the scanning maps $\{sc^{d,m}_{n,\C}\}_{d\geq 1}$ are also $\Z_2$-equivariant maps,
by setting $S_{\R}=S_{\C}\vert \po^{\infty,m}_{n}(\R)$ we also obtain the {\it the stable horizontal scanning map}
%%(7.22)%%
\begin{equation}
S_{\R}:
\po^{\infty,m}_{n}(\R)
\to 
\Omega
\pol^{m}_{n}(I^2,\sigma I^2)^{\Z_2}.
\end{equation}
%%%%%%%%%%
\par
(iv)
We define the map
%%(7.23)%%
\begin{equation}
r_I:\mathcal{Q}^{m}_n(I^2,\sigma I^2)\to \mathcal{Q}^m_n(I,\partial I)
\end{equation}
%%%%%%%%%
by the restriction $r_I (\xi_1,\cdots ,\xi_m)=(\xi_1\cap \R,\cdots , \xi_m\cap \R)$
for $(\xi_1,\cdots ,\xi_m)\in \mathcal{Q}^m_n(I^2,\sigma I^2)$.
}
%%%
\end{definition}
%%(End of Definition 7.4)%%%

%%%
%%%(Scanning map Theorem)%%%%%
%%(Theorem 7.5)%%%
\begin{thm}[\cite{Se} (cf. \cite{Gu1}]\label{thm: scanning map}
%%%%%%%%%%%%%%%%%%%
If $mn\geq 3$, the stable horizontal scanning maps
$$
\begin{cases}
S_{  }\ :&
\Q^{\infty,m}_{n}(\R)
\stackrel{\simeq}{\longrightarrow}
\Omega
\mathcal{Q}^{m}_{n}(I^2,\sigma I^2)
\\
S_{\C}:&\po^{\infty,m}_n(\C)
\stackrel{\simeq}{\longrightarrow}
\Omega \pol^m_n(I^2,\sigma I^2)
\\
S_{\R}:&\po^{\infty,m}_n(\R)
\stackrel{\simeq}{\longrightarrow}
\Omega \pol^m_n(I^2,\sigma I^2)^{\Z_2}
\end{cases}
$$
are homotopy equivalences.
\end{thm}
%%%%
\begin{proof}
%%%%%%

The assertions can be proved by using the idea indicated in \cite[Proposition 3.2, Lemma 3.4]{Se} 
(and also in \cite[Proposition 2]{Gu1}).
However, since the argument in \cite{Se} is very sketchy, we provided a detailed proof for the map $S$ in  \cite[Theorem 5.6]{KY10}.\footnote{%
%%%%%(Footnote 7)%%%%
Although the proof of \cite[Theorem 5.6]{KY10} was given under the condition
$mn\geq 4$, the same proof works for the case $mn=3$.
}
%%(End of Footnote 7)%%
%Moreover, although $mn\geq 4$ is assumed in \cite[Theorem 5.6]{KY10}, this result holds for $mn= 3$, too.
%(we do not use the result for the case $mn=3$ in \cite{KY10} and we assumed
%$mn\geq 4$ there).
The proof for the map $S_{\C}$ carries over word for word to the present case,  if we replace the condition
\lq\lq $(*)_n^{\R}$ \rq\rq   by the condition \lq\lq $(*)_n$ \rq\rq.
A very similar argument also works the case of the map $S_{\R}$, and this completes the proof.
%%%%%%%%%
\end{proof}
%%%%(End of proof of Theorem 7.5)%%%% 
%%
%%%(Corollary 7.6)%%
\begin{crl}
%%%
If $mn\geq 3$, the map
$
S_{\C}:\po^{\infty,m}_n(\C)
\stackrel{\simeq}{\longrightarrow}
\Omega \pol^m_n(I^2,\sigma I^2)
$
is a $\Z_2$-equivariant homotopy equivalence.
%%%
\end{crl}
%%%%
\begin{proof}
%%%
Consider the $\Z_2$-action on the space $\po^{\infty,m}_n(\C)$
induced from the conjugation on $\C$.
If we consider  $S^1$ as a $\Z_2$-space with trivial $\Z_2$-action, we see that
$(\Omega \pol^m_n(I^2,\sigma I^2))^{\Z_2}=
\Omega \mathcal{Q}^m_n(I^2,\sigma I^2)^{\Z_2}.$
Since $(S_{\C})^{\Z_2}=S_{\R}$,
 the assertion easily follows from Theorem \ref{thm: scanning map}.
%%%%%
\end{proof}
%%(End of Proof of Corollary 7.6)%%
%%%%
%%%%
%%(Lemma 7.7)%%
\begin{lemma}\label{lmm: deformation retract}
%%%%%%%%%%%
The map 
$r_I:\mathcal{Q}^{m}_n(I^2,\sigma I^2)\stackrel{\simeq}{\longrightarrow} \mathcal{Q}^m_n(I,\partial I)$
is a deformation retraction.
%%%%%%%%%%
\end{lemma}
%%%%%%%%%%%
\begin{proof}
%%%%
We identify
$I^2=\{a+b\sqrt{-1}\in \C: -1\leq a,b\leq 1\}\subset \C$ as before.
Let $\Pi\subset I^2$ denote the subspace defined by
$\Pi =\{a+b\sqrt{-1}\in I^2:b\in \{0,\pm \frac{1}{2}\}\}.$
For $b\in \R$, let $\epsilon (b)=\frac{b}{|b|}$ if $b\not=0$ and $\epsilon (0)=0$.
Now consider the homotopy $\varphi :I^2\times [0,1]\to I^2$ given by
$\varphi (\alpha,t)=a+\{(1-t)b+\frac{\epsilon (b)t}{2}\}\sqrt{-1}$
for
$\alpha =a+b\sqrt{-1}\in I^2$
$(a,b\in\R)$.
By means of this homotopy, one can define a deformation retraction
$
R:\mathcal{Q}^m_n (I^2,\sigma I^2)
\stackrel{\simeq}{\longrightarrow}
\mathcal{Q}^m_n (\Pi ,\partial I\times \{0,\pm \frac{1}{2}\}).
$
\par
Next, by using the homotopy given by
$f_t(a+b\sqrt{-1})=ta+(1-t)+b\sqrt{-1}$ if $b=\pm \frac{1}{2}$
and $f_t(a+b\sqrt{-1})=a$ if $b=0$,
one can also define a deformation retraction
$\varphi :
\mathcal{Q}^m_n (\Pi, \partial I\times \{0,\pm \frac{1}{2}\})
\stackrel{\simeq}{\longrightarrow}
\mathcal{Q}^m_n (I,\partial I).$
Since $r_I=\varphi\circ R$, it is also
a deformation retraction.
%%%
\end{proof}
%%(End of proof of Lemma 7.7)%%%
%%
%%(Proposition 7.8)%%
\begin{prp}\label{prp: stable result1}
%%%%%%
\begin{enumerate}
\item[$\I$]
If $mn\geq 3$, there is a homotopy equivalence
$$
\Omega \pol^{m}_n(I^2,\sigma I^2)\stackrel{\simeq}{\longrightarrow} \Omega^2S^{2mn-1}.
%\quad
%\mbox{and}
%\quad
%\Omega \mathcal{Q}^{m}_n(I,\partial I) \simeq \Omega S^{mn-1}.
$$
\item[$\II$]
If $mn\geq 3$, there is a homotopy equivalence
%There is a map
$$
\Omega \mathcal{Q}^{m}_n(I,\partial I) 
\stackrel{\simeq}{\longrightarrow}\Omega S^{mn-1}.
$$
%which is a homotopy equivalence if $mn\geq 4$ and a homology equivalence
%if $mn=3$.
%%%
%%%%%
\end{enumerate}
\end{prp}
%%%%%
\begin{proof}
%%%%%%%
(i)
It follows from \cite[Theorem 1.8]{KY8} that there is a homotopy equivalence
$\po^{\infty,m}_n(\C)\simeq \Omega^2S^{2mn-1}$.
Since the map $S_{\C}$ is a homotopy equivalence by Theorem \ref{thm: scanning map}, we obtain the homotopy equivalence
$\Omega \pol^{m}_n(I^2,\sigma I^2)\simeq \Omega^2S^{2mn-1}.$
\par
(ii)
Consider the map 
$\hat{S}:\Q^{\infty,m}_n(\R)\stackrel{}{\longrightarrow}
\Omega \mathcal{Q}^m_n(I,\partial I)$
defined as the composite of maps $\hat{S}=(\Omega r_I)\circ S$.
It follows from Theorem \ref{thm: scanning map} and Lemma \ref{lmm: deformation retract}
that $\hat{S}$ is a homotopy equivalence. 
Moreover,
it follows from Theorem \ref{thm: KY10} that
there is a map 
$\dis i^{\infty,m}_n=\lim_{d\to\infty}i^{d,m}_n:\Q^{\infty,m}_n(\R)
\stackrel{\simeq}{\longrightarrow} \Omega S^{mn-1}$
which is a homotopy equivalence for $mn\geq 4$ and is a
homology equivalence for $mn=3$.
Now consider the composite map
$$
\begin{CD}
i^{\infty,m}_n\circ \tilde{S}:
\Omega \mathcal{Q}^m_n(I,\partial I)
@>\tilde{S}>\simeq>
\Q^{\infty,m}_n(\R)
@>i^{\infty,m}_n>\simeq>
\Omega S^{mn-1},
\end{CD}
$$
where $\tilde{S}$ denotes a homotopy inverse of $\hat{S}$.
This map is the desired homotopy equivalence 
if $mn\geq 4$, and it is a homology equivalence if $mn=3$.
\par
However, when $mn=3$,
since two spaces $\Omega \mathcal{Q}^m_n(I,\partial I)$ and
$\Omega S^{mn-1}$ are loop spaces, they are H-spaces.
Thus, the above map $i^{\infty,m}_n\circ \tilde{S}$ is indeed a homotopy equivalence even when $mn=3$ and this completes the proof.
%%%%%%%%
\end{proof}
%%(End of proof of Proposition 7.8)%
%%
%%(Theorem 7.9)%%
\begin{thm}\label{thm: stable result2}
%%%%%%%%%%%%
If $mn\geq 3$, there is a homotopy equivalence
%There is a map
$$
\po^{\infty,m}_n(\R) 
\stackrel{\simeq}{\longrightarrow}
 \Omega^2S^{2mn-1}\times \Omega S^{mn-1}.
$$
%which is a homotopy equivalence if $mn\geq 4$ and a homology equivalence
%if $mn=3$.
%%%%
\end{thm}
%%%%%%
\begin{proof}
%%%%
Since $S_{\R}$ is a homotopy equivalence,
it follows from Proposition \ref{prp: stable result1} that 
it suffices to prove that
there is a homotopy equivalence
%%(7.24)%%
\begin{equation}\label{eq:homotopy equiv pol}
%%%%%%%
\Omega \pol^m_n(I^2,\sigma I^2)^{\Z_2}\simeq
\Omega \pol^m_n(I^2,\sigma I^2) \times
\Omega \mathcal{Q}^m_n(I,\partial I).
\end{equation}
%%%%%%
Consider the  map
$\hat{r}_I:\pol^m_n(I^2,\sigma I^2)^{\Z_2}
\to \mathcal{Q}^m_n(I,\partial I)$ given by the restriction
$
\hat{r}_I(\xi_1,\cdots ,\xi_m)=
(\xi_1 \cap \R,\cdots ,\xi_m\cap \R).
$
Note that $\hat{r}_I$ is a quasifibration with fiber
\begin{align*}
F&=
\pol^m_n(I\times ([-1,0)\cup (0,1]),\partial I\times ([-1,0)\cup (0,1]))^{\Z_2}
\\
&\cong
\pol^m_n(I\times (0,1],\partial I\times (0,1])
=\bigcup_{0<\epsilon<1}\pol^m_n(I\times [\epsilon,1],\partial I\times [\epsilon,1]).
\end{align*}
%%%%
Since the space
$\P(\epsilon)=\pol^m_n(I\times [\epsilon,1],\partial I\times [\epsilon,1])$ is homotopy equivalent to
the space $\P(1/2)=\pol^m_n(I\times [1/2,1],\partial I\times [1/2,1])$ for any $0<\epsilon <1$ by radial expansion,
there is a homotopy equivalence
$$
F\simeq \P(1/2)=\pol^m_n(I\times [1/2,1],\partial I\times [1/2,1])\cong \pol^m_n(I^2,\sigma I^2).
$$
Thus, we obtain a fibration sequence (up to homotopy equivalence)
%%%(7.25)%%
\begin{equation}
\pol^m_n(I^2,\sigma I^2) 
\stackrel{s_I}{\longrightarrow}
\pol^m_n(I^2,\sigma I^2)^{\Z_2}
\stackrel{\hat{r}_I}{\longrightarrow}
\mathcal{Q}^m_n(I,\partial I).
\end{equation}
%%%%%%%
Let $i:(I,\partial I)\stackrel{\subset}{\longrightarrow}(I^2,\sigma I^2)$ denote the natural inclusion map
given by $i(x)=(x,0)$ for $x\in I$.
This inclusion naturally extends to the map
$i_{\#}:\mathcal{Q}^m_n(I,\partial I)\to \pol^m_n(I^2,\sigma I^2)$.
Since $\hat{r}_I\circ i_{\#}=
\mbox{id}_{\mathcal{Q}^m_n(I,\partial I)}$, the loop sum map
$$
\Omega s_I+\Omega i_{\#}:
\Omega \pol^m_n(I^2,\sigma I^2) \times
\Omega \mathcal{Q}^m_n(I,\partial I)
\stackrel{\simeq}{\longrightarrow}
\Omega \pol^m_n(I^2,\sigma I^2)^{\Z_2}
$$
is a homotopy equivalence.
%%%%
\end{proof}
%%%%%(End of Proof of Theorem 7.9)%%

%%(Corollary 7.10)%%
\begin{crl}\label{crl: stable result2}
%%%%%%%%%%%%
If $mn\geq 3$, there is a homotopy equivalence
%There is a map
$$
j^{\infty,m}_n:
\po^{\infty,m}_n(\R) 
\stackrel{\simeq}{\longrightarrow}
 (\Omega^2\CP^{mn-1})^{\Z_2}.
$$
%which is a homotopy equivalence if $mn\geq 4$ and a homology equivalence
%if $mn=3$.
%%%%
\end{crl}
%%%%%
\begin{proof}
%%%%%%
%%%%
%by  (\ref{eq: homotopy equiv}),
The assertion  follows from Theorem \ref{thm: stable result2} and
(\ref{eq: homotopy equiv}).
%%%%%%%%%%
\end{proof}
%%%(End of proof of Corollary 7.10)%%
%
%Now we can prove the following two important results.
%%%%
%%%%(Corollary 7.11)%%%
\begin{crl}\label{thm: II}
%%%%%%
If $mn\geq 3$, 
there is a map
$$
f^{d,m}_n:
\po^{d,m}_n(\R) \to (\Omega^2_d\CP^{mn-1})^{\Z_2}
\simeq \Omega^2S^{2mn-1}\times \Omega S^{mn-1}
$$
which is a homotopy equivalence through dimension
$D(d;m,n)$ if $mn\geq 4$, and a homology equivalence
through dimension $D(d;m,n)$ if $mn=3$.
%%%%%%%%%%
\end{crl}
%%%%%%%%%%
%%%%%%%%%%
%%%%(Proof of Corollary 7.11)%%%%%%
\begin{proof}
%%%%%%%%%%%%%%%%%%%%%%
Let $\hat{\iota}^{d,m}_n:\po^{d,m}_n(\R)\to \po^{\infty,m}_n(\R)$ denote the natural map, and let us consider
the composite of maps
$$
\begin{CD}
f^{d,m}_n:
\po^{d,m}_n(\R)
@>{\hat{\iota}^{d,m}_n}>> \po^{\infty,m}_n(\R)
@>j^{\infty,m}_n>\simeq>
(\Omega^2_d\CP^{mn-1})^{\Z_2}.
\end{CD}
$$
Then by using Theorem \ref{thm: stab1}, Corollaries 
 \ref{crl: stabilization} and \ref{crl: stable result2},
 we see that
 the map $f^{d,m}_n$ is a homotopy equivalence through dimension
 $D(d;m,n)$
 if $mn\geq 4$ and  it is a homology equivalence through dimension
$D(d;m,n)$ if $mn=3$.
%%%%%%%%
\end{proof}
%%%(End of proof of Corollary 7.11)
%%
%%
%%
%%
%%%%%%%%%%%%%%%%%%%%%%%%%%
%%(SECTION 8)%%
\section{The homotopy type of $\po^{d,m}_n(\R)$}
\label{section: the main result}
%%%%%%%%%%%%%%%%%%%%%%%%%%
%%
%%
%%
%%
In this section we give proofs 
of our main results  
(Theorems \ref{thm: KY13} and \ref{thm: KY13; stable homotopy type}) and their corollaries
(Corollaries \ref{crl: surjection}, \ref{crl: jet embedding}, 
\ref{crl: GKY4 type theorem} and \ref{crl: Q+poH+}).
%%%
\par
Before giving the proof of Theorem \ref{thm: KY13},
we make some general comments about it.
%\par
It seems plausible that one could give a proof of Theorem \ref{thm: KY13} by a method similar to that used in the proof of  \cite[Theorem 1.8]{KY8}.
However,  this approach would require a study of 
$\Z_2$-equivariant homotopy of 
spaces of $\Z_2$-equivariant maps, which seems difficult. 
For this reason we decided to use an indirect approach, combining the corresponding results given in  \cite[Theorem 1.8]{KY8} and 
\cite[Theorem 1.8]{KY10}.
%}
%\end{remark}
%%(End of Remark 7.1)%%%%%
%Now it is ready to give the proof of Theorem \ref{thm: KY13}.

%%%(Proof of Theorem 2.7)%%
\begin{proof}[Proof of Theorem \ref{thm: KY13}]
%%%%%%%%%%%%%%%%%
Assume that $mn\geq 3$.
From now on,
we use the same notations as in  Lemma \ref{lmm: key diagram}.
We denote by
%%(8.1)%%
\begin{equation}
%%%%
r_{S^1}:\Map^*_0(D^2,S^1;\CP^{mn-1},\RP^{mn-1})
\to \Omega_0\RP^{mn-1}
\end{equation}
%%%%%%%%%%%
the restriction map given by $r_{S^1}(f)=f\vert S^1$, and let
%%(8.2)%%
\begin{equation}\label{eq: projections}
%%%%
\Omega^2 S^{2mn-1}\stackrel{\tiny q_1}{\longleftarrow}
\Omega^2S^{2mn-1}\times \Omega S^{mn-1}
\stackrel{q_2}{\longrightarrow} \Omega S^{mn-1}
\end{equation}
%%%%%
denote the projections onto the first and the second factor, respectively. 
Then if we write
%%(8.3)%%
\begin{equation}
%%%
E_0^*=\Map^*_0(D^2,S^1;\CP^{mn-1},\RP^{mn-1}),
\end{equation}
%%
%if we write 
%$\Map^{*}_0(\CP^{mn-1},\RP^{mn-1})=\Map^*_0(D^2,S^1;\CP^{mn-1},\RP^{mn-1})$,
it follows from diagram (\ref{CD: key diagram}) and the definitions of the three natural maps
$i^{d,m}_{n,\Ha_+}$,
$i^{d,m}_{n,\R}$ and $i^{d,m}_{n}$ 
that the following diagram is homotopy commutative :
\par\vspace{1mm}\par
{\small
%%%(8.4)%%
\begin{equation}\label{CD: The main commutative diagram}
\begin{CD}
%%%%%
\po^{d,m}_n(\R;\Ha_+) 
@>\iota^{d,m}_{n,\Ha_+}>\subset> 
\po^{d,m}_n(\R) 
@>\iota^{d,m}_n>\subset> \Q^{d,m}_n(\R)
%%%%%
\\
@V{i^{d,m}_{n,\Ha+}}VV @V{i^{d,m}_{n,\R}}VV @V{i^{d,m}_n}VV
\\
%%%%
\Omega^2_d\CP^{mn-1} @. (\Omega^2_d\CP^{mn-1})^{\Z_2}
@. \Omega_{[d]_2}\RP^{mn-1}
\\
@V{\iota_{\C}^{\p}}V{\simeq}V @V{\iota^{\p}}V{\simeq}V @V{\iota_{\R}^{\p}}V{\simeq}V
\\
\Map^*_0(D^2,S^1;\CP^{mn-1},*) 
@>\hat{j}>\subset>
E_0^* 
%\Map^*_0(D^2,S^1;\CP^{mn-1},\RP^{mn-1})
@>r_{S^1}>>\Omega_{0}\RP^{mn-1}
%%%
\\
@V{\iota_{\C}}V{\simeq}V @V{\iota_{\C\R}}V{\simeq}V @V{\iota_{\R}}V{\simeq}V
%%%
\\
%%%
\Omega^2S^{2mn-1} @<{q_1}<< \Omega^2S^{2mn-1}\times \Omega S^{mn-1}
@>q_2>> \Omega S^{mn-1}
%%%
\end{CD}
\end{equation} 
%%%%%%%%%%
}
%%%
\par\vspace{0.5mm}\par
Now consider the following three maps given by
%%(8.5)%%
\begin{equation}\label{eq: map I}
%%%%%%%%
\begin{cases}
I^{d,m}_n&=\iota_{\C\R} \circ \iota^{\p}\circ i^{d,m}_{n,\R}:
\po^{d,m}_n(\R)\to \Omega^2S^{2mn-1}\times \Omega S^{mn-1},
\\
I^{d,m}_{n,\Ha_+}&=\iota_{\C}\circ\iota_{\C}^{\p}\circ i^{d,m}_{n,\Ha_+}
:\po^{d,m}_n(\R;\Ha_+)\to \Omega^2S^{2mn-1},
\\
I^{d,m}_{n,\R}&=\iota_{\R}\circ\iota_{\R}^{\p}\circ i^{d,m}_{n}
:\Q^{d,m}_n(\R)\to \Omega S^{mn-1}.
\end{cases}
%%%% 
\end{equation}
%%%%%
Now let  $\F$ be any fixed field. 
%\par
It follows from Theorems \ref{thm: Theorem H} and \ref{thm: KY10} that 
the two maps $I^{d,m}_{n,\Ha_+}$ are $I^{d,m}_{n,\R}$
are homology equivalences
through dimension $D(\lfloor d/2\rfloor;m,n;\C)$
and $D(d;m,n)$, respectively.
%So that the map $I^{d,m}_{n,\Ha_+}$ induces an isomorphism on homology group
%$H_s(\  ;\F)$ for any $s\leq D(\lfloor d/2\rfloor;m,n;\C)$, and 
%the map $I^{d,m}_{n,\R}$ also induces an isomorphism on homology group
%$H_s(\  ;\F)$ for any $s\leq D(d;m,n)$.
\par
Since $D(d;m,n)<D(\lfloor d/2\rfloor;m,n;\C)$ (by Lemma  \ref{lmm: inequality}),
by using the  diagram (\ref{CD: The main commutative diagram}), 
we see that the  induced homomorphism
%%%
%%(8.6)%%
\begin{equation}\label{eq: HJ}
(I^{d,m}_n)_*:H_s(\po^{d,m}_n(\R);\F)\to 
H_s(\Omega^2S^{2mn-1}\times \Omega S^{mn-1};\F)
\end{equation}
%%%
is an epimorphism for any  $s\leq D(d;m,n)$.
However, since 
$$
\dim_{\F}H_s(\po^{d,m}_{n}(\R);\F)=
\dim_{\F}H_s(\Omega^2S^{2mn-1}\times \Omega S^{mn-1};\F)<\infty
$$
for any $s\leq D(d;m,n)$ (by Corollary \ref{thm: II}), 
%%%
we notice that the homomorphism $(I^{d,m}_n)_*$ 
(given by (\ref{eq: HJ}))
is an isomorphism 
for any $s\leq D(d;m,n)$.
\par
Then, by putting $\F=\Z/p$ ($p$: prime) or $\mathbb{Q}$,
 it follows from the universal coefficient Theorem that
the map $I^{d,m}_n$ induces an isomorphism on the homology  group
$H_s(\ ;\Z)$ for any $s\leq D(d;m,n)$.
Thus, we see that the map $i^{d,m}_{n,\R}$
is a homology equivalence through dimension $D(d;m,n)$.
\par
Next, assume that $mn\geq 4$.
Then by Lemma \ref{lmm: abelian}, we see that
two spaces $\po^{d,m}_n(\R)$ and  $\Omega^2S^{2mn-1}\times \Omega S^{mn-1}$ are
simply connected.
Hence, we obtain that the map
$i^{d,m}_{n,\R}$ is  a homotopy equivalence through dimension $D(d;m,n)$.
This completes the proof of Theorem \ref{thm: KY13}.
%%%
\end{proof}
%%%%(End of Proof of Theorem 2.7)%%%%%
%%%%
%%%
%%%(Proof of Corollary 2.9)%%
\begin{proof}[Proof of Corollary \ref{crl: surjection}]
The assertion follows easily from the proof of Theorem \ref{thm: KY13}, so we omit the details. 
%%%%%%%%%%%%%%%%%%%%
\end{proof}
%%(End of proof of Corollary 2.9)%%%
%%%
%%(Proof of Corollary 2.10)%%
\begin{proof}[Proof of Corollary \ref{crl: jet embedding}]
%%%%
Consider the following commutative diagram:
%%
%%(8.7)%%
\begin{equation}\label{CD: jet}
\begin{CD}
\po^{d,1}_n(\R) @>i^{d,1}_{n,\R}>> (\Omega^2_d\CP^{n-1})^{\Z_2}@>>\simeq> \Omega^2S^{2mn-1}\times \Omega S^{mn-1}
\\
@V{j^d_n}VV \Vert @. \Vert @.
\\
\po^{d,n}_1(\R) @>i^{d,n}_{1,\R}>> (\Omega^2_d\CP^{n-1})^{\Z_2}@>>\simeq>\Omega^2S^{2mn-1}\times \Omega S^{mn-1}
\end{CD}
\end{equation}
%%%%%
First, consider the case $n\geq 4$.
It follows from Theorem \ref{thm: KY13} that
the maps $i^{d,1}_{n,\R}$ and $i^{d,n}_{1,\R}$ are homotopy equivalences
through dimension $(D;1,n)$ and $D(d;n,1)$, respectively.
Since $D(d;1,n)<D(d;n,1)$, by using the diagram (\ref{CD: jet})
we see that the map $j^d_n$ is a homotopy equivalence through dimension
$D(d;1,n)=(n-2)(\lfloor d/n\rfloor +1)-1.$
\par
Next, if $n=3$, by Theorem \ref{thm: KY13} we see that
the maps $i^{d,1}_{3,\R}$ and $i^{d,3}_{1,\R}$ are homology equivalences
through dimension $(D;1,3)=\lfloor d/3\rfloor$ and $D(d;3,1)=d$, respectively.
Thus, we see that the map
$j^d_3$ is a homology equivalence through dimension $\lfloor d/3\rfloor$.
%%%
\end{proof}
%%(End of proof of Corollary 2.10)%%%
%%%
%%%(Corollary 8.1)%%
\begin{crl}\label{lmm: pi1}
%%%%%%%%
$\I$
If $mn=3$ and $d\geq n$, two  maps $i^{d,m}_{n,\R}$ 
and $s^{d,m}_{n,\R}$ induce  isomorphisms
$$
\begin{cases}
(i^{d,m}_{n,\R})_*:
\pi_1(\po^{d,m}_n(\R))
\stackrel{\cong}{\longrightarrow}
\pi_1((\Omega^2_d\CP^{2})^{\Z_2})
\cong \pi_1(\Omega^2S^5\times \Omega S^2)\cong \Z,
\\
(s^{d,m}_{n,\R})_*:
\pi_1(\po^{d,m}_n(\R))
\stackrel{\cong}{\longrightarrow}
\pi_1(\po^{d+1,m}_n(\R))\cong \Z.
\end{cases}
$$
\par
$\II$ If $d\geq 3$,  the smap $j^d_3:\po^{d,1}_3(\R)\to \po^{d,3}_1(\R)$
induces an isomorphism
$$
(j^d_3)_*:\pi_1(\po^{d,1}_3(\R))
\stackrel{\cong}{\longrightarrow}
\pi_1(\po^{d,3}_1(\R))\cong \Z .
$$
%%%
\end{crl}
%%(Proof of Corollary 8.1)%%
\begin{proof}
$\I$
Let $mn=3$ and
consider the following  commutative diagram:
%%()%%
\begin{equation*}\label{CD: pi1}
%%%%%%%%%%%
%%(The first diagram)%%
\begin{CD}
\pi_1(\po^{d+1,m}_n(\R))
@<(s^{d,m}_{n,\R})_*<<
\pi_1(\po^{d,m}_n(\R))
@>(i^{d,m}_{n,\R})_*>>
\pi_1((\Omega^2_d\CP^{2})^{\Z_2})
\\
@V{h_1}V{\cong}V
@V{h_2}V{\cong}V @V{h_3}V{\cong}V
\\
H_1(\po^{d+1,m}_n(\R);\Z)
@<(s^{d,m}_{n,\R})_{\#}<\cong<
H_1(\po^{d,m}_n(\R);\Z)
@>(i^{d,m}_{n,\R})_{\#}>\cong>
H_1((\Omega^2_d\CP^{2})^{\Z_2};\Z)
\end{CD}
\end{equation*}
%%%%()%%%%%%%%
where  $h_k$ 
$(k=1,2,3)$ are 
corresponding Hurewicz homomorphisms.
\par
Since 
$\pi_1(\po^{d,m}_n(\R))\cong \pi_1(\po^{d+1,m}_n(\R))\cong
\pi_1((\Omega^2_d\CP^{mn-1})^{\Z_2})\cong\Z$
by Lemma \ref{lmm: abelian},
by the Hurewicz theorem each  $h_k$ 
is an isomorphism.
It follows from
Theorems  \ref{thm: KY13} and \ref{thm: stab1}
that $(i^{d,m}_{n,\R})_{\#}$ and $(s^{d,m}_{n,\R})_{\#}$ are
isomorphisms, and the assertion  (i) easily follows.
\par
(ii) Consider the commutative diagram
%%(8.7)%%
\begin{equation*}\label{CD: jet}
\begin{CD}
\pi_1(\po^{d,1}_3(\R)) @>(i^{d,1}_{3,\R})_*>\cong> \pi_1((\Omega^2_d\CP^{2})^{\Z_2})@>>\cong> 
\pi_1(\Omega^2S^{5}\times \Omega S^{2})=\Z
\\
@V{(j^d_3)_*}VV \Vert @. \Vert @.
\\
\pi_1(\po^{d,3}_1(\R)) @>(i^{d,1}_{3,\R})_*>\cong> \pi_1((\Omega^2_d\CP^{2})^{\Z_2})@>>\cong> 
\pi_1(\Omega^2S^{5}\times \Omega S^{2})=\Z
\end{CD}
\end{equation*}
%%%
Since $(i^{d,1}_{3,\R})_*$ and $(i^{d,3}_{1,\R})_*$ are isomorphisms (by (i)),
the homomorphism $(j^d_3)_*$ is also an isomorphism.
%%%%%%%%%%%%%%
\end{proof}
%%%%(End of proof of Corollary 8.1)%%%%%%

\par\vspace{2mm}\par
Next, consider the stable homotopy type of $\po^{d,m}_n(\R)$.
%%%
For this purpose, recall the following result:
%%
%%%%(Lemma 8.2)%%%
\begin{lemma}\label{lmm: dimension Poly}
%%%%
For any field $\F$ and any $s\geq 1$,
%%()%%
\begin{equation*}
\dim_{\F}H_s(\po^{d,m}_n(\R);\F)\leq \dim_{\F}H_s(\P^{d,m}_n;\F)<\infty .
\end{equation*} 
%%%%%%%
\end{lemma}
%%%%(Proof of Lemma 8.2)%%
\begin{proof}
Consider the Vassiliev type spectral sequence
$$
\{E^{t;d}_{k,s},d^t:E^{t;d}_{k,s}\to E^{t;d}_{k+t,s+t-1}\}
\Rightarrow
H_{s-k}(\po^{d,m}_n(\R);\F).
$$
Since $E^1_s=\bigoplus_{k\in \Z}E^{1;d}_{k,k+s}\cong H_s(\P^{d,m}_n;\F)$ for any $s\geq 1$
by Lemma \ref{lmm: total E}, the assertion easily follows.
\end{proof}
%%%%(End of proof of Lemma 8.2)%%
%%%%%
%\par\vspace{1mm}\par
 Now we can give the proof of 
 Theorem \ref{thm: KY13; stable homotopy type}.
%%
%%
%%
%%%%(Proof of Theorem 2.11)%%%
\begin{proof}[Proof of Theorem \ref{thm: KY13; stable homotopy type}]
%%%%%%%%%%%%%%%
From now on, let $d_0=\lfloor d/2\rfloor$, and we write 
%%(8.8)%%
\begin{equation}
\po^{d_0,m}_n=\po^{d_0,m}_n(\C),
\quad
\Q^{d,m}_n=\Q^{d,m}_n(\R).
\end{equation}
%%%%%%
Let $\mathcal{F}^{d,m}_n=
%\{(i,j)\in \Z^2:i,j\geq 1,i+2j\leq \lfloor d/n\rfloor\}=
\{(i,j)\in \N^2:i+2j\leq \lfloor d/n\rfloor\}$ as in
(\ref{eq: set F}), and
recall the spaces $B^{d,m}_n$ and $\P^{d,m}_n(\R)$ given by (see  (\ref{eq: the space B}) and (\ref{eq: stable type of Poly}))
%%%
%%%()%%
\begin{align*}\label{eq: space B}
%%%%%%%%%%%%%%
B^{d,m}_n&=
\bigvee_{(i,j)\in \mathcal{F}^{d,m}_n}S^{(mn-2)i}\wedge \Sigma^{2(mn-2)j}D_j,
 \
 \P^{d,m}_n=\po^{d_0,m}_n \vee B^{d,m}_n\vee \Q^{d,m}_n.
%\\
\end{align*}
%%%%%%
%%
%%%
It follows from (\ref{eq: the space poly}) and (\ref{eq: stable Q}) that there are stable homotopy equivalences
%%%
%%%(8.9)%%
\begin{equation}\label{eq: Q type}
%%%
 \theta_{\C}: \bigvee_{j=1}^{\lfloor d_0/n\rfloor}\Sigma^{2(mn-2)j}D_j
\stackrel{\simeq_s}{\longrightarrow}
\po^{d_0,m}_n,
\quad
 \theta_{\R}: \bigvee_{i=1}^{\lfloor d/n\rfloor}S^{(mn-2)i}
 \stackrel{\simeq_s}{\longrightarrow}
 \Q^{d,m}_n.
\end{equation}
%%%%%%%%
Hence, the space $\P^{d,m}_n$ is stably homotopy equivalent 
to the following space
%%()%%
\begin{equation*}
%\P^{d,m}_n
%\simeq_s
\Big(\bigvee_{j=1}^{\lfloor d_0/n\rfloor}\Sigma^{2(mn-2)j}D_j\Big)
\vee
\Big(
\bigvee_{(i,j)\in \mathcal{F}^{d,m}_n}
\big(S^{(mn-2)i}\wedge\Sigma^{2(mn-2)j}D_j\big)\Big)\vee
\Big(\bigvee_{i=1}^{\lfloor d/n\rfloor}S^{(mn-2)i}\Big).
\end{equation*}
%%%%%%%%%
%On the other hand,
It follows from \cite{Ja}  and the Snaith splitting \cite{Sn} 
that there are  two stable homotopy equivalences
%%%%%
%%
%%(8.10)%%
\begin{equation}
%%%%
\begin{cases}
\dis \Omega^2 S^{2mn-1}\simeq_s A^m_n(1):=
\bigvee_{j=1}^{\infty}
\Sigma^{2(mn-2)j}D_j,
%%%
\\
\dis \Omega S^{mn-1}\ \ \simeq_s A^m_n(3):=
\bigvee_{i=1}^{\infty}S^{(mn-2)i}.
\end{cases}
\end{equation}
%%%
%where two spaces $A^m_n(1)$ and $A^m_n(3)$ are defined by
%%% 
%%%(7.8)%%
%\begin{equation}
%%%%%%%%%
%A^m_n(1)=
%\bigvee_{j=1}^{\infty}
%\Sigma^{2(mn-2)j}D_j
%\quad
%\mbox{and}
%\quad
%A^m_n(3)=\bigvee_{i=1}^{\infty}S^{(mn-2)i}.
%%%%%%%%%
%\end{equation}
%%%%%%%
%%%
Thus, there is a  stable homotopy equivalence
%%()%%
\begin{equation*}\label{eq: stable splitting 2}
%%%%%%%%%%%%
\Omega^2S^{2mn-1}\times \Omega S^{mn-1}
\simeq_s
\bigvee_{k=1}^3A^m_n(k)=
A^m_n(1)\vee A^m_n(2)\vee A^m_n (3),
%%%
\end{equation*}
%%%%
where the space $A^m_n(2)$ is defined by
%%%
%%(8.11)%%
\begin{align}
%%%%%%%%
A^m_n(2)&=A^m_n(1)\wedge A^m_n(3)
%=\bigvee_{i,j\geq 1}
%\Sigma^{(mn-2)(i+2j)}D_j
%\\
%\nonumber
%&=
=
\bigvee_{i,j\geq 1}S^{(mn-2)i}\wedge \Sigma^{2(mn-2)j}D_j.
\end{align}
%%%
Let $q_1^{\p}$ and $q_3^{\p}$ denote the natural projections to the first and the third factors given by
%%(8.12)%%
\begin{equation}
\begin{CD}
%%%%%%
A^m_n(1)
@<q_1^{\p}<< 
A^m_n(1)\vee A^m_n(2)\vee A^m_n(3) @>q_3^{\p}>> A^{m}_n(3).
%%%%%%
\end{CD}
\end{equation}
%%%
Similarly,
let $p_{\C}$ and $p_{\R}$ also denote the natural projections to the first and the third factors given by
%%(8.13)%%
\begin{equation}
\begin{CD}
%%%%%%
\po^{\lfloor d_0\rfloor,m}_n
@<p_{\C}<< 
\P^{d,m}_n=
\po^{\lfloor d_0\rfloor,m}_n
\vee
B^{d,m}_n\vee \Q^{d,m}_n @>p_{\R}>> \Q^{d,m}_n.
%%%%%%
\end{CD}
\end{equation}
%%%
%Next, 
%recall $\mathcal{F}^{d,m}_n=\{(i,j)\in \N^2:i+2j\leq \lfloor d/n\rfloor\}$, and 
Let $p_1$, $p_2$ and $p_3$ denote the corresponding natural pinching map
given by
\begin{align*}
%%%%%
p_1:A^m_n(1)& =\bigvee_{j=1}^{\infty}\Sigma^{2(mn-2)j}D_j\to
\P(1):= 
\bigvee_{j=1}^{\lfloor d_0/n\rfloor}\Sigma^{2(mn-2)j}D_j,
\\
p_2:A^m_n(2)&=\bigvee_{i,j\geq 1}S^{(mn-2)i}\wedge \Sigma^{2(mn-2)j}D_j\to 
\bigvee_{(i,j)\in \mathcal{F}^{d,m}_n}S^{(mn-2)i}\wedge\Sigma^{2(mn-2)j}D_j,
\\
p_3:A^m_n(3)&=\bigvee_{i=1}^{\infty}S^{(mn-2)i}\to
\P(3):=
\bigvee_{i=1}^{\lfloor d/n\rfloor}S^{(mn-2)i}.
%%%%%
\end{align*}
%%%%%%%
By the diagram (\ref{CD: The main commutative diagram})
and (\ref{eq: map I}), we obtain the following
commutative diagram
%%%
{\small
%%%(8.14)%%
\begin{equation}
\label{CD: The second main commutative diagram}
\begin{CD}
%%%%%
\po^{d,m}_n(\R;\Ha_+) 
@>\iota^{d,m}_{n,\Ha_+}>\subset> 
\po^{d,m}_n(\R) 
@>\iota^{d,m}_n>\subset> \Q^{d,m}_n(\R)
%%%%%
\\
@V{I^{d,m}_{n,\Ha+}}VV @V{I^{d,m}_{n}}VV @V{I^{d,m}_{n,\R}}VV
\\
%%%%
\Omega^2S^{2mn-1} @<{q_1}<< \Omega^2S^{2mn-1}\times \Omega S^{mn-1}
@>q_2>> \Omega S^{mn-1}
%%
%%%
\\
@V{\pi^{\p}}V{\simeq_s}V @V{\pi}V{\simeq_s}V @V{\pi^{\p\p}}V{\simeq_s}V
\\
%%%%%%%%%%%%%%%%%%%
\dis
A^m_n(1)
@<q_1^{\p}<<
A^m_1(1)\vee A^m_n(2)\vee A^m_n(3)
@>q_3^{\p}>>
A^m_n(3)
%%%
\\
@V{p_1}VV @V{p_1\vee p_2\vee p_3}VV @V{p_3}VV
\\
\dis \P(1)
%\bigvee_{j=1}^{\lfloor d_0/n\rfloor}\Sigma^{2(mn-2)j}D_j
@<p_1^{\p}<< 
\P (1)\vee B^{d,m}_n \vee \P (3)
@>p_3^{\p}>> \P(3)
 %\dis \bigvee_{i=1}^{\lfloor d/n\rfloor}S^{(mn-2)i}
%%%
\\
@V{\theta_{\C}}V{\simeq_s}V @V{\theta_{\C}\vee \mbox{\tiny id}\vee \theta_{\R}}V{\simeq_s}V @V{\theta_{\R}}V{\simeq_s}V
\\
\po^{d_0,m}_n
@<p_{\C}<<
\dis
\P^{d,m}_n=
\po^{d_0,m}_n
\vee B^{d,m}_n\vee \Q^{d,m}_n
@>p_{\R}>>
\Q^{d,m}_n
%\\
%@. \Vert @. @.
%\\
%@. \P^{d,m}_n @.
%%
\end{CD}
\end{equation} 
%%%%%%
}
%%
%where we write
%$\dis
%\P (1)=\bigvee_{j=1}^{\lfloor d_0/n\rfloor}\Sigma^{2(mn-2)j}D_j,$  
%\ and \ 
%$\dis \P (3)=\bigvee_{i=1}^{\lfloor d/n\rfloor}S^{(mn-2)i}.$
%%%%
\par\vspace{1mm}\par
\noindent Now consider the following three maps 
$$
\begin{cases}
%%%
J^{d,m}_{n,\C}:\po^{d,m}_n(\R;\Ha_+)\to \po^{d_0,m}_n,
\quad\quad
J^{d,m}_{n,\R}:\Q^{d,m}_n \to \Q^{d,m}_n,
\\
J^{d,m}_n:\po^{d,m}_n(\R) \to \P^{d,m}_n=\po^{d_0,m}_n
\vee B^{d,m}_n\vee \Q^{d,m}_n
%\\
%J^{d,m}_{n,\C}:\po^{d,m}_n(\R;\Ha_+)\to \po^{d_0,m}_n,
%\qquad\quad
%J^{d,m}_{n,\R}:\Q^{d,m}_n \to \Q^{d,m}_n
\end{cases}
$$
defined by
%%
%%(8.15)%%
\begin{equation}
%%%
\begin{cases}
%%%%
J^{d,m}_{n,\C}&=\theta_{\C}\circ p_1\circ \pi^{\p}
\circ I^{d,m}_{n,\Ha_+},
\quad
J^{d,m}_{n,\R}=
\theta_{\R}\circ p_3\circ \pi^{\p\p}\circ I^{d,m}_{n,\R},
\\
J^{d,m}_n&=(\theta_{\C}\vee \mbox{id}\vee \theta_{\R})
\circ (p_1\vee p_2 \vee p_3)\circ \pi
 \circ  I^{d,m}_n.
%%%
%\\
%J^{d,m}_{n,\C}&=\theta_{\C}\circ p_1\circ \pi^{\p}
%\circ I^{d,m}_{n,\Ha_+},
%\qquad\quad
%J^{d,m}_{n,\R}=
%\theta_{\R}\circ p_3\circ \pi^{\p\p}\circ I^{d,m}_{n,\R}.
\end{cases}
\end{equation}
%%%
It suffices to prove that the map $J^{d,m}_n$ is a stable homotopy equivalence.
It is easy to see that 
%%(8.16)%%
\begin{equation}\label{eq: map of Q}
J^{d,m}_{n,\R}=\mbox{id}:\Q^{d,m}_n(\R)\to \Q^{d,m}_n(\R)
\quad
\mbox{ (up to homotopy equivalence).}
%%%%%%%%%
\end{equation}
%%%%%%%%%
\par
Now
let $\F$  be any fixed field, and 
consider the homomorphism
$$
(J^{d,m}_{n,\C})_*:H_s(\po^{d,m}_n(\R;\Ha_+);\F)
\to H_s(\po^{d_0,m}_n;\F)
\quad
\mbox{for }s\geq 1.
$$
%%%%%%%%
We will need the following lemma. 
%%%
%%%
%%%%(Lemma 8.3)%%
\begin{lemma}\label{lmm: (*)}
%%%%%%%
The induced homomorphism
%%(8.17)%%
\begin{equation}\label{eq: homo J}
%%%%%%%
(J^{d,m}_n)_*:H_s(\po^{d,m}_n(\R);\F)
\to
H_s(\P^{d,m}_m;\F)
\end{equation}
%%%%%%%
is an epimorphism for any $s\geq 1$ and for any field $\F$.
\end{lemma}
%%%%%%
We postpone the proof of Lemma \ref{lmm: (*)},
and first complete the proof of Theorem \ref{thm: KY13; stable homotopy type} by using  this lemma.
%%
%Let $\F$ be a field.
By Lemma \ref{lmm: (*)} we have
%%%(8.18)%%
\begin{equation}\label{eq: inequality 2}
\dim_{\F}H_s(\po^{d,m}_n(\R);\F)
\geq \dim_{\F}H_s(\P^{d,m}_n;\F)
\quad
\mbox{for any }s\geq 1.
\end{equation}
%%%
Combining this with Lemma \ref{lmm: dimension Poly} we
obtain the equality:
%%(8.19)%%
\begin{equation}\label{eq: equality dim}
\dim_{\F}H_s(\po^{d,m}_n(\R);\F)=
\dim_{\F}H_s(\P^{d,m}_n;\F)<\infty
\quad
\mbox{for any }s\geq 1.
\end{equation}
%%%
Since $(J^{d,m}_n)_*$  is an epimorphism, 
 (\ref{eq: equality dim}) implies that it is, in fact,
 an isomorphism for any $s\geq 1$.
By putting $\F=\Z /p$ ($p$: prime) and $\F=\mathbb{Q}$,
and using  the universal coefficients theorem, 
we conclude that
%%()%%
%\begin{equation}
%%%%
$$
(J^{d,m}_n)_*:
H_s(\po^{d,m}_n(\R);\Z)
\stackrel{\cong}{\longrightarrow}
H_s(\P^{d,m}_n;\Z)
$$
%%%%%%%%
%\end{equation} 
%%%%%%%%
is an isomorphism for every $s\geq 1$.
Hence, 
the map
%%()%%
%\begin{equation}
%%%%%%%
$J^{d,m}_n$
%:\po^{d,m}_n(\R)
%\stackrel{\simeq_s}{\longrightarrow} 
%\P^{d,m}_n=
%\po^{d_0,m}_n \vee B^{d,m}_n\vee Q^{d,m}_n
%\end{equation}
%%%
is a stable homotopy equivalence.
This completes the proof of Theorem \ref{thm: KY13; stable homotopy type}.
%%%%%%
\end{proof}
%%(End of Proof of Theorem 2.11)%%%%
%%
%%
%%
%%
%%
%%%%%%%%%%%%%%%%%%%%%%%%%
%%%(Proof of Lemma 8.3)%%%%
\begin{proof}[Proof of Lemma \ref{lmm: (*)}]
%%%%%%%%%%%%%%%%%%%%%%%%%
It follows from Corollary \ref{crl: surjection}, Lemma \ref{lmm: space poly(Ha+)}
 and (\ref{eq: map of Q})
that the maps $J^{d,m}_{n,\C}$ and $J^{d,m}_{n,\R}$ induce epimorphisms on homology groups $H_*(\ ;\F)$.
Moreover, since
$\P^{d,m}_n=
\po^{d_0, m}_n(\C) \vee B^{d,m}_n\vee \Q^{d,m}_n(\R)$
(by  (\ref{eq: stable type of Poly})),
to prove Lemma \ref{lmm: (*)},
 it suffices  to prove the following assertion. 
 %%(\dagger)%%%%
\begin{enumerate}
\item[$(\dagger)$]
If $x\not= 0\in H_s(B^{d,m}_n;\F)$, there is an element $y\in H_s(\po^{d,m}_n(\R);\F)$ such that
$(J^{d,m}_n)_*(y)=x$.
\end{enumerate}
%%%%%%%%%%%%
Since 
$\dis B^{d,m}_n=\bigvee_{(i,j)\in \mathcal{F}^{d,m}_n}
\Sigma^{(mn-2)(i+2j)}D_j$, 
we may assume, without a loss of generality,
 that
$
x\in
H_{s}(\Sigma^{(mn-2)(i+2j)}D_j;\F)\cong
H_{s-(mn-2)i}(\Sigma^{2(mn-2)j}D_j;\F)
$
for some $(i,j)\in \mathcal{F}^{d,m}_n$.
%where 
%$$
%x_1\in H_{(mn-2)i}(S^{(mn-2)i};\F),\ 
%x_2\in H_k(\Sigma^{2(mn-2)j}D_j;\F), \ 
%k+(mn-2)i=s.
%$$
By Lemma \ref{lmm: inequality}, $1\leq j\leq \lfloor d_0/n\rfloor.$

%%%
\par
On the other hand,
it follows  from (\ref{eq: the space poly}) and Lemma
\ref{lmm: space poly(Ha+)} that
there is a stable homotopy equivalence 
$$
\po^{d,m}_n(\R;\Ha_+)\simeq
\po^{d_0,m}_n(\C)\simeq_s
\bigvee_{k=1}^{\lfloor d_0/n\rfloor}
\Sigma^{2(mn-2)k}D_k.
$$
Thus, there exists an element
$y_1\in H_{s-(mn-2)i}(\po^{d,m}_n(\R;\Ha_+);\F)$
such that $(J^{d,m}_{n,\C})_*(\sigma^{(mn-2)i}(y_1))=x,$
where $\sigma^{k}$ denotes the $k$-fold suspension isomorphism.
Then by using the commutative diagram
(\ref{CD: The second main commutative diagram}), we see that
$$
x=
(J^{d,m}_n)_*((\iota^{d,m}_{n,\Ha_+})_*(\sigma^{(mn-2)i}(y_1)).
$$
Hence, if
we put 
$y=(\iota^{d,m}_{n,\Ha_+})_*(\sigma^{(mn-2)i}(y_1))\in H_s(\po^{d,m}_n(\R);\F)$,
the assertion
$(\dagger)$ is satisfied.
This completes the proof of Lemma \ref{lmm: (*)}.
\end{proof}
%%%(End of proof of Lemma 8.3)%%
%%
%%
The following assertion easily follows from (\ref{eq: equality dim}).
%%(Corollary 8.4)%%
\begin{crl}
%%%%%%%%%%%
If $mn\geq 3$,
the Vassiliev spectral sequence
$$
\{E^{t;d}_{k,s},d^t:E^{t;d}_{k,s}\to E^{t;d}_{k+t,s+t-1}\}
\Rightarrow
H_{s-k}(\po^{d,m}_n(\R);\Z).
$$
collapses at $E^1$-terms, i.e.
$E^{1;d}_{**}=E^{\infty ;d}_{**}$.
\qed
\end{crl}
%%%
%%%(End of proof of corollary 8.4)%%%
%%
%%
Next we  give the proofs of Corollaries 
\ref{crl: GKY4 type theorem} and \ref{crl: Q+poH+}.

%%(Proof of Corollary 2.13)%%
\begin{proof}[Proof of Corollary \ref{crl: GKY4 type theorem}]
%%%%%%
It follows from Theorem \ref{thm: KY13; stable homotopy type} that
there is a stable homotopy equivalence
%%()%%
\begin{equation*}
\po^{\lfloor d/n\rfloor, mn}_1(\R)
\simeq_s
\Big(\bigvee_{i=1}^{\lfloor d/n\rfloor}S^{(mn-2)i}\Big)
\vee
\Big(\bigvee_{i\geq 0,j\geq 1,i+2j\leq \lfloor d/n\rfloor}
\Sigma^{(mn-2)(i+2j)}D_j\Big).
\end{equation*}
%%%
Thus, by using Corollary \ref{crl: KY13; stable homotopy type}
we easily obtain  the  stable homotopy equivalence
(\ref{eq: KY13 stable equiv}).
%%%
\end{proof}
%%(End of proof of Corollary 2.13)%%
%%
%%
%%
%%(Proof of Corollary 2.14)%%%%%%%%%
\begin{proof}[Proof of Corollary \ref{crl: Q+poH+}]
%%%%%%
The assertion (i) follows from Theorem \ref{thm : I} and it remains to show (ii).
Recall from Lemma \ref{lmm: space poly(Ha+)} that there is a homeomorphism
$\po^{d,m}_n(\R;\Ha_+)\cong \po^{\lfloor d/2\rfloor,m}_n(\C)$.  
Consider  the following homotopy commutative diagram
$$
\begin{CD}
\po^{d,m}_n(\R;\Ha_+) 
@>\iota^{d,m}_{n,\Ha_+}>\subset>
\po^{d,m}_n(\R)
\\
\Vert @. @VV{\simeq_s}V
\\
\po^{d,m}_n(\R;\Ha_+)
\cong 
\po^{\lfloor d/2\rfloor,m}_n(\C)
@>j_1>\subset>
\po^{\lfloor d/2\rfloor,m}_n(\C)\vee B^{d,m}_n \vee \Q^{d,m}_n (\R)
\end{CD}
$$
where $j_1$ denotes the inclusion to the first factor.
Then the assertion (ii) easily follows from 
 Theorem \ref{thm: KY13; stable homotopy type}.
\end{proof}
%%(End of proof of Corollary 2.14)%%
%%

%%
%%
%%(Remark 8.5)%%
\begin{remark}
%%%%%%%%%%%%%%
{\rm
%(i)
Let 
$mn\geq 3$. 
%and let $d_0=\lfloor d/2\rfloor$.
%%%%%
Then
it follows from (\ref{eq: the space poly}), (\ref{eq: stable Q}) and
Theorem \ref{thm: KY13; stable homotopy type} that 
there are stable homotopy equivalences
%%(8.20)%%
\begin{align*}
%%%
\po^{\lfloor d/2\rfloor,m}_n(\C)\times \Q^{d,m}_n(\R)
&\simeq_s
\po^{\lfloor d/2\rfloor,m}_n(\C)\vee A^{d,m}_n \vee \Q^{d,m}_n(\R),
\\
\po^{d,m}_n(\R)
\quad
&\simeq_s
\po^{\lfloor d/2\rfloor,m}_n(\C)\vee B^{d,m}_n \vee \Q^{d,m}_n(\R),
\end{align*}
%%%%%%%
where  
 $A^{d,m}_n$ denotes the space defined by
%%%(8.21)%%
\begin{equation}\label{eq:Admn}
%%%%%%%%%%
A^{d,m}_n=\bigvee_{1\leq i\leq \lfloor d/n\rfloor,
1\leq j\leq \lfloor d_0/n\rfloor}
\Sigma^{(mn-2)(i+2j)}D_j. 
%%%
\end{equation}
%%%%
Since 
$
\dis B^{d,m}_n=
\bigvee_{(i,j)\in \mathcal{F}^{d,m}_n}
\Sigma^{(mn-2)(i+2j)}D_j
\subset A^{d,m}_n
$
(by Lemma \ref{lmm: inequality}),
 the space  
$\po^{d,m}_n(\R)$ can be regarded as
the subspace of 
$\po^{\lfloor d/2\rfloor,m}_n(\C)\times \Q^{d,m}_n(\R)$
in the stable homotopy category.
%\par
%(ii)
%We hope that the map $i^{d,m}_{n,\R}$ will be a homotopy equivalence through dimension $D(d;m,n)$ even if $mn=3$, and we would like to study about this problem in the subsequent paper.
%There is a result which supports that this conjecture may be true
%(see Lemma \ref{lmm: pi1} below).
\qed
}
%%%
\end{remark}
%%%(End of Remark 8. 5)%%
%%
%%%(Remark 8.6)%%
%\begin{remark}\label{remark: po2_1}
%%%%%%%%%%%%%%%
%{\rm
%Let $\po^{d,1}_{2;j}(\R)\subset \po^{d,1}_2(\R)$ denote the subspace of all monic polynomials
%$f(z)\in \po^{d,1}_2(\R)$ which has just $j$ distinct real roots.
%If $g(z)\in \R[z]$ has a complex root $\alpha\in \C\setminus \R$, its conjugate $\overline{\alpha}$ is also root of $g(z)$.
%Thus, the space $\po^{d,1}_2(\R)$ has the following decomposition of path-components:
%%%()%%
%\begin{equation}
%\po^{d,1}_2(\R)=\coprod_{j=0}^{\lfloor d/2\rfloor}
%\end{equation}
%}
%\end{remark}
%%%(End of Remark 8.6)
%%%
%%%
%%%
%%%(End of SECTION 8)%%%%
%%%(End of SECTION 8)%%%%
%%%(End of SECTION 8)%%%%
%%%(End of SECTION 8)%%%%
%%%
%%%
%%%%%%%%%%%%%%
\section{Appendix: The case $(m,n)=(1,2)$}\label{section: (m,n)=(1,2)}\label{section: appendix}
%%%%%%%%%%%%%
%%%%
%%%%
In this section we consider the homotopy type of the space
$\po^{d,m}_n(\R)$ for the case $(m,n)=(1,2)$. 
%Although this space was not studied in Segal's seminal article \cite{Se},
%This space was studied in Segal's seminal article \cite{Se} and
In fact, the homology stability  follows directly from Segal's seminal article \cite{Se}. 
%However,  as 
However, as it does not appear to be stated anywhere, we provide a detailed proof below.

%%%%%%%%%%%%%%
\par\vspace{2mm}\par
Let $f(z)\in \po^{d,1}_2(\R)$.
Then
$f(z)\in \R [z]$ is a  monic polynomial of degree $d$  %with real coefficients and 
without multiple roots.
If $\alpha\in \C \setminus \R$ is a complex root of $f(z)$, its conjugate
$\overline{\alpha}$ is also a root of $f(z)$.
Thus it can be represented as
%%(9.1)%%
\begin{equation}\label{eq: fz}
f(z)=\Big(\prod_{i=1}^{d-2j}(z-x_i)\Big)
\Big(\prod_{k=1}^j(z-\alpha_k)(z-\overline{\alpha}_k)\Big)
\end{equation}
%%%
for some $(\{x_i\}_{i=1}^{d-2j},\{\alpha_k\}_{k=1}^j)
\in C_{d-2j}(\R)\times C_j(\Ha_+)$. 
%where
%$\Ha_+$ denotes the upper half plane
%$\Ha_+=\{\alpha \in \C:\mbox{Im }\alpha >0\}$.
%%%
%%%%(Defintion 9.1)%%
\begin{definition}
%%%%%%%%
{\rm 
For each non-negative integer
$0\leq j\leq \lfloor d/2\rfloor$, let 
%be non-negative integer, and let
$\po^{d,1}_{2,j}(\R)$ denote the subspace of $\po^{d,1}_2(\R)$
consisting 
of all monic polynomials $f(z)\in \po^{d,1}_2(\R)$
which have only $2j$ non-real roots.
\par 
It is easy to see that each polynomial $f(z)\in \po^{d,1}_{2,j}(\R)$
can be represented in the form of (\ref{eq: fz}).
%\par
%(ii)
%Let
%%%(10.2)%%
%\begin{equation}
%s^{d,1}_{2,j}=s^{d,1}_{2,\R}\vert \po^{d,1}_{2,j}(\R) :\po^{d,1}_{2,j}(\R)\to \po^{d+2,1}_{2,j+1}(\R)
%\end{equation}
%%%%
%denote the stabilization map given 
%by the restriction.
%by adding roots from infinity
%as in \cite[Definition 5.4]{KY13}.
\qed
%%%%%%%
}
%%%%%%%%%%%%
\end{definition}
%%%%(End of Definition 9.1)%%%
%%%%

%%%%(Remark 9.2)%%
\begin{remark}
%%%%%
{\rm
It is also easy to show that there is a homeomorphism
%%(9.2)%%
\begin{equation}\label{eq: homeo po1}
\po^{d,1}_{2,j}(\R)\cong C_{d-2j}(\R)\times C_j(\Ha_+).
\end{equation}
%%%%%%
Let $\varphi :\Ha_+\stackrel{\cong}{\longrightarrow}\C$ be any homeomorphism (which we now fix), and 
let
%%(9.3)%%
\begin{equation}\label{eq: homeo po2}
%%%%%%%
%%%%%%%
\overline{\varphi}:C_j(\Ha_+)
\stackrel{\cong}{\longrightarrow}C_j(\C)
\end{equation}
denote the homeomorphism given by
$\overline{\varphi}(\{\alpha_k\}_{k=1}^j)=
\{\varphi(\alpha_k)\}_{k=1}^j.$
Since there is a homeomorphism $C_{d-2j}(\R)\cong \R^{d-2j}$
and a homotopy equivalence
$C_j(\C)\simeq  K({\rm Br}(j),1)$,
%it follows from (\ref{eq: homeo po1}) and (\ref{eq: homeo po2})
%that
there is a homotopy equivalence
%%(9.4)%%
\begin{equation}\label{eq: C-homotopy equiv}
\po^{d,1}_{2,j}(\R)\stackrel{\simeq}{\longrightarrow}
C_j(\C)\simeq K({\rm Br}(j),1),
\end{equation}
%%%
where ${\rm Br}(j)$ denotes the Artin braid  group on $j$ strands.
}
%%%%%%%
\end{remark}
%%(End of Remark 9.2)%%%
%%
%%
%%%%%(Theorem 9.3)%%
\begin{thm}\label{thm: the case (m,n)=(1,2)}
%%%%
%\begin{enumerate}
%\item[$\I$]
$\I$
The space $\po^{d,1}_2(\R)$ consists of 
$(\lfloor d/2\rfloor +1)$ connected components
$
\{\po^{d,1}_{2,j}(\R):0\leq j\leq \lfloor d/2\rfloor\},
$
%\item[$\II$]
%\par
%$\II$
and
there is a homotopy equivalence
$$
\po^{d,1}_{2,j}(\R)\simeq K({\rm Br}(j),1).
$$
%where ${\rm Br}(j)$ denotes the Artin braid group on $j$ strings.
%In particular, the space 
%$\po^{d,1}_{2,j}(\R)$ is contractible if $j\in \{0,1\}$ and it is homotopy equivalent to the space $S^1$ if $j=2$.
\par
%\item[$\III$]
$\II$
There is a natural map
$$
i^{d,1}_{2,j}:\Po^{d,1}_{2,j}(\R)\to \Omega^2_j\CP^1\simeq
\Omega^2_{j}S^2
\simeq \Omega^2S^3
$$
which
is a homology equivalence up to dimension $\lfloor j/2\rfloor$ if $j\geq 3$, and
a homotopy equivalence through dimension $1$ if $j=2$.
%%%%%%%
\end{thm}
%%%
%%(Proof of Theorem 9.3)%%%%%%
\begin{proof}
%%%%%%%%%%%%
Since (i) is obvious, we only deal with (ii).
We make the usual identification $S^2=\C \cup \infty$ 
and assume that $j\geq 2$.
Define the map
$
i^{}_{2,j}:C_j(\C)\to \Omega^2_{j}\CP^1
$
by
%%(9.5)%%
\begin{equation}\label{eq: poly map}
%%%%%
i^{}_{2,j}(\{a_k\}_{k=1}^j)(\alpha)=
\begin{cases}
[f(\alpha):f(\alpha)+f^{\p}(\alpha)] 
=[F_2(f)(\alpha)]
& \mbox{ if }\alpha \in \C
\\
[1:1] & \mbox{ if }\alpha =\infty
\end{cases}
\end{equation}
%%%
for $\alpha \in S^2=\C\cup \infty$ and
$f(z)=\prod_{k=1}^j(z-a_k)$.
%$f(z)\in \P_j(\C)$ denotes the monic polynomial of degree $j$
%given by
%$f(z)=\prod_{k=1}^j(z-a_k)$.
%%%%
Note that
%%(9.6)%%
\begin{equation}\label{eq: electric field}
%%%%%%%
\frac{f(z)+f^{\p}(z)}{f(z)}=
1+\frac{f^{\p}(z)}{f(z)}=1+\sum_{k=1}^j\frac{1}{z-a_k}
\end{equation}
%%%%
It then follows from \cite[page 42]{Se} that the map
$i^{}_{2,j}$ is a homology equivalence up to dimension
$\lfloor j/2\rfloor$.
%\footnote{%
%%%(Footnote 1)%%
%Alternatively this can be proved by using Vassiliev spectral sequence
%as in Theorem \ref{thm: stab1}.
%%\cite{KY8}.
%}
%%%(End of Footnote 1)%%
%%%
Now  let $2\leq j\leq \lfloor d/2\rfloor$, and consider the map
%%%(9.7)%% $i^{d,1}_{2,j}$
\begin{equation}
%%%%
i^{d}_{2,j}:\po^{d,1}_{2,j}(\R)\to \Omega^2_{j}\CP^1\cong \Omega^2_jS^2\simeq \Omega^2S^3
\end{equation}
%%% 
given by the composite of maps
%%(9.8)%%
\begin{equation}
\po^{d,1}_{2,j}(\R)\stackrel{\simeq}{\longrightarrow}
C_j(\C)
\stackrel{i^{}_{2,j}}{\longrightarrow}
\Omega^2_j\CP^1\simeq \Omega^2_jS^2\simeq \Omega^2S^3.
\end{equation}
%%%%
Since $i_{2,j}$ is a homology equivalence up to dimension
$\lfloor j/2\rfloor$, the map
$i^{d,1}_{2,j}$ is also a homology equivalence up to 
dimension $\lfloor j/2\rfloor$.
\par
Finally consider the case $j=2$.
Since ${\rm Br}(2)\cong \Z$, there is a homotopy equivalence
$\po^{d,1}_{2,2}(\R)\simeq S^1.$
Since the map
$i^{d,1}_{2,2}$ induces an epimorphism
$$
(i^{d,1}_{2,2})_*:\Z\cong H_1(\po^{d,1}_{2,2}(\R);\Z)
\stackrel{}{\longrightarrow}
H_1(\Omega^2_jS^2;\Z)\cong \Z,
$$
it is indeed an isomorphism.
Moreover, note that there is an isomorphism
$$
\pi_1(\po^{d,1}_{2,2}(\R))\cong \pi_1(S^1)\cong \Z
\cong \pi_1(\Omega^2_jS^2).
$$
Next, by the Hurewicz Theorem we know that the map
$i^{d,1}_{2,2}$ induces an isomorphism on the fundamental group
$\pi_1(\ )$.   We have proved (ii). 
%Thus, $i^{d,2}_{1,j}$ is a homotopy equivalence through dimension $1$ for $j=2$.
%%%%%%%%%%%%%%%
\end{proof}
%%(End of proof of Theorem 4.3)%%%
%%
%%
%%
%%(Remark 4.4)%%
\begin{remark}\label{rmk: homology stab (m,n)=(1,2)}
{\rm
(i)
It follows from (\ref{eq: electric field}) that 
%we see that
 the map $i^{}_{2,j}$ is equivalent to 
 {\it the electric field map}
$E_c:\C\cup \infty \to \C\cup \infty$,  
described in \cite[page 213]{Se0}. 
\par
Moreover, we can easily  see that
$i^{d,1}_{2,j} =i^{d,2}_{1,\R}\vert \po^{d,2}_{1,j} (\R)$
(up to homotopy).
\par
(ii)
%By  (ii) of Theorem \ref{thm: the case (m,n)=(1,2)}, we see that
%the homology stability holds for the space $\Po^{d,1}_{2,j}(\R)$.
%%
Since 
$$
\pi_1(\po^{d,1}_{2,j}(\R))={\rm Br}(j)\not\cong
\Z=\pi_1(\Omega^2S^3)
%{\rm Br}(j+1)=
%\pi_1(\po^{d+2,1}_{2,j+1}(\R))
\quad \mbox{ for }j\geq 3,
$$
the homotopy stability does not hold for the map $i^{d,2}_{1,j}$ when $j\geq 3$.
\qed
}
\end{remark}
%(End of Remark 4.4)%%%%%%%
%%

\par\vspace{0.5mm}\par
\noindent{\bf Acknowledgements. }
%%%%%%%%
The authors would like to take this opportunity to thank
Professor Martin Guest 
for his many valuable  insights and suggestions, especially concerning  
scanning maps.
%%%
%\par
The second author was supported by 
JSPS KAKENHI Grant Number JP22K03283. 
This work was also supported by the Research Institute for Mathematical Sciences, a Joint Usage/Research Center located in Kyoto University.

%%%(References)%%%%%%%%%
 
%%%%%%
%%%%%
\end{document}